\documentclass[11pt]{article}
\usepackage{epsfig}
\usepackage{amssymb,amsmath,amsthm,amscd}
\usepackage{latexsym}

\newtheorem{thm}{Theorem}[section]
\newtheorem{prop}[thm]{Proposition}

\newtheorem{lem}[thm]{Lemma}

\theoremstyle{definition}
\newtheorem{defn}[thm]{Definition}
\newtheorem{rem}[thm]{Remark}
 
\newcounter{labelflag} 

\newcommand{\Label}[1]{
                       \ifnum\thelabelflag=1
                          \ifmmode
                             \makebox[0in][l]{\qquad\fbox{\rm#1}}
                          \else
                             \marginpar{\vspace{0.7\baselineskip}
                                        \hspace{-1.1\textwidth}
                                        \fbox{\rm#1}}
                          \fi
                       \fi
                       \label{#1} }

\newcommand{\be}{\begin{equation}}
\newcommand{\ee}{\end{equation}}

\newcommand{\R}{\mathbb{R}}
\newcommand{\N}{\mathbb{N}}

\def \calf {{  {\mathcal{F}} }}

\def \calftwo {{  {\mathcal{F}}_2  }}
\def \cala {{  {\mathcal{A}}  }}
\def \calb {{  {\mathcal{B}}  }}
\def \cald {{  {\mathcal{D}}  }}
\def \caln {{  {\mathcal{N}}  }}
\def \omegat { {\tilde{\omega} }}

\def \thonet {{  \theta_{1,t}  }}
\def \thtwot {{  \theta_{2,t}  }}

 \def  \ltwo {{L^2 (\R^n)}}
  \def  \hone {{H^1 (\R^n)}}
  
  \newcommand{\ii}{\int_{\R^n}}
\newcommand{\zto}{{z(\theta_{2, t} \omega)} }
\newcommand{\rh }{\rho ({\frac {|x|^2}{k^2}})}
\newcommand{\rhp }{\rho^\prime ({\frac {|x|^2}{k^2}})}

\begin{document}

\begin{titlepage}
\title{\Large  \bf  Sufficient and Necessary Criteria for 
Existence of Pullback Attractors  for Non-compact  Random Dynamical Systems }
\vspace{10mm}

\author{
Bixiang Wang
\vspace{5mm}\\
Department of Mathematics\\
 New Mexico Institute of Mining and
Technology  \\ Socorro,  NM~87801, USA  \vspace{3mm}\\
Email: bwang@nmt.edu }
\date{}
\end{titlepage}

\maketitle

\medskip

\begin{abstract}
We study pullback attractors of non-autonomous 
 non-compact dynamical  systems generated by
differential  equations  with non-autonomous 
  deterministic as well as stochastic forcing terms. 
  We first introduce the concepts of 
  pullback attractors and asymptotic 
  compactness for such systems.
We then prove a sufficient and necessary
 condition for existence of pullback attractors. 
 We also introduce the concept of 
 complete orbits for this sort of systems 
 and use  these special solutions to 
 characterize the structures of pullback 
 attractors. For random systems containing 
 periodic deterministic forcing terms, we show 
 the pullback attractors are  also periodic
 under certain conditions.
As an application of the abstract theory, we 
prove the existence of a unique pullback 
attractor for Reaction-Diffusion equations  
on $\R^n$ with both deterministic and random 
external terms.   Since Sobolev embeddings are 
not compact on unbounded domains,  the uniform 
estimates on the tails of solutions are employed to 
establish the asymptotic compactness of solutions.
\end{abstract}

{\bf Key words.}       Pullback  attractor;  
  periodic  random attractor;  random complete solution;  \\
 unbounded domain.

 {\bf MSC 2000.} Primary 35B40. Secondary 35B41, 37L30.

\baselineskip=1.3\baselineskip

\section{Introduction}
\setcounter{equation}{0}

 This paper is concerned  with   the theory of   pullback   attractors 
of non-autonomous non-compact  dynamical  systems generated by
differential  equations   with  both non-autonomous   
deterministic and stochastic
forcing terms.  We will prove a  sufficient   and necessary   
condition for existence of pullback attractors 
for  such systems.  We will
also  introduce   the concept of  $\cald$-complete orbits  for
 random  systems
and   use these special solutions to
characterize  the  structures of $\cald$-pullback attractors. 
In the deterministic case,  the characterization  of attractors 
by complete orbits is well known,
see, e.g., \cite{bal2, tem1} for autonomous systems
and \cite{carva1, carva2, che1} for non-autonomous systems.
However,   for random dynamical systems,
  no such characterization of attractors  is  available 
 in the literature. 
  For random  systems  containing   periodic  
     deterministic forcing,  we prove  
 the   pullback  attractors    are also
 periodic under some conditions.
 It seems   that    periodic
 random attractors  for PDEs 
 was first studied 
 in \cite{dua1}  where   a
 quasigeostrophic flow model
 was   considered. 
As an application of our abstract results, 
we will investigate   the pullback attractors of
 the  following Reaction-Diffusion equation defined on $\R^n$: 
\be
  \label{intro1}
  du +  (\lambda u - \Delta u) dt    =  f(x, u) dt  +  g(x, t) dt + h(x) d \omega,
 \ee
 with   initial data  $
 u(x, \tau) = u_\tau (x)$,  
 where  $t > \tau$ with  $\tau \in\R$,  
 $x \in \R^n$,    $\lambda$ is a  positive constant,
$g \in L^2_{loc}(\R, \ltwo)$    and   $h \in H^2(\R^n)\bigcap W^{2,p} (\R^n)$
for some $p \ge 2$.
$f$  is a smooth nonlinear function satisfying some dissipativeness 
 and growth conditions, and $\omega$
  is a  two-sided real-valued Wiener process on 
  a probability space. 
  The reader is also referred to 
  \cite{wan32} 
  for    existence of   periodic random attractors
  for   the Navier-Stokes equations, 
  and to \cite{wan33} for bifurcation of
  random periodic solutions  for parabolic equations.

  Stochastic  equations  like    \eqref{intro1} have  been used as  
   models to study the phenomena of stochastic resonance
    in biology and physics, where
  $g$ is a  time-dependent  input  signal and $\omega$
  is a Wiener process  used to test the impact of stochastic 
  fluctuations  on  $g$. 
  In this  respect,   we refer the reader to \cite{gamma1,
  moss1, tuckwell1, tuckwell2, wiesen1, wiesen2}
  for  more details,  where the authors have demonstrated 
  that, under certain circumstances,  the noise can 
  help the system to detect 
    weak   signals.  This is an interesting 
   phenomenon  because noise is  generally considered as 
   a nuisance  in the process of signal transmission. 
  In the present paper, we will  investigate   the long term 
  behavior of   stochastic  equations like \eqref{intro1},  
  especially the pullback attractors
  of the equation.

   The concept of pullback attractors for
   random dynamical systems was introduced by the authors
   in \cite{cra2, fla1, schm1}, and the existence of such attractors
   for compact systems
   was   established in  \cite{arn1, car1, car2, chu2, cra1, cra2,
   dua1,  fla1, huang1, kloe1, schm1} and the references therein.
   By a compact random dynamical system, we mean
   the system  has a compact pullback absorbing set.
   The  results   of \cite{cra2, fla1, schm1} can be used  to show the existence of
   pullback attractors for  many PDEs defined in bounded domains.
   However,  these  results do not apply
   to PDEs defined on unbounded domains or
    to weakly dissipative  PDEs even defined in bounded domains.
    The reason is simply because the dynamical systems 
    generated by the  equations  in such a case are no longer
    compact.
     To solve the problem,  the authors of  \cite{bat1}
     introduced    the concept
    of   asymptotic compactness
    for random dynamical systems which is an extension
     of    deterministic systems. 
    This type of  asymptotic compactness does not require 
     a system to have a compact pullback  absorbing set, 
     and hence it works for dissipative ODEs on infinite lattices 
      as well as   PDEs   on unbounded domains, 
      see e.g., \cite{bat1, bat2, wan2, wan3, wan31}.
    In this paper,   we will further extend 
      the concept of asymptotic compactness 
        to the case of differential  equations with
    both non-autonomous deterministic 
    and random forcing terms,
    and  as an example,  prove   
    the existence of pullback attractors  
     for equation
    \eqref{intro1}   defined on the 
    unbounded domain $\R^n$.

    In order to deal with  pullback   attractors 
    for equations like \eqref{intro1}, we  need to 
     introduce        the concept of cocycles
    over two parametric spaces $\Omega_1$   and $\Omega_2$
    (see Definition \ref{ds1} below),
    where $\Omega_1$  is responsible for the 
    non-autonomous deterministic forcing terms 
    and  $\Omega_2$ is responsible   for the stochastic 
     forcing terms.  This kind of cocycle 
    is an extension of the classical non-autonomous 
     deterministic cocycles  (see, e.g., \cite{car3, car4,  
     carva1, carva2, wan4}) as well as  the random cocycles
    (see, e.g., \cite{arn1, cra2, fla1, schm1}).
    More precisely,  if $\Omega_2$  is a singleton,  
    then Definition \ref{ds1} reduces  to the concept of  
    non-autonomous deterministic cocycles;   if $\Omega_1$
      is  a singleton instead,   the definition reduces to that of 
      random cocycles; if both $\Omega_1$  and $\Omega_2$  
      are singletons,    we get the
   autonomous dynamical systems as discussed 
   in \cite{bab1, hal1, sel1, tem1}.  
   In the case   where the non-autonomous 
   deterministic terms are
   periodic, we will examine the periodicity of the pullback attractors. 
   To this end,  we      introduce the concept of   periodic cocycles  in 
      Definition \ref{periodic-ds1}.
   To define a pullback   attractor,  as usual, 
   we need  to specify  its  attraction domain 
   which, in general,   is a collection  $\cald$ of some families of nonempty 
   sets.    In the next section,  we will extend the concepts of 
   $\cald$-pullback absorbing sets, $\cald$-pullback  
   asymptotic compactness and $\cald$-pullback attractors
     to the cocycles with two parametric spaces,  
      and then prove a sufficient and necessary 
      condition for existence of $\cald$-pullback 
      attractors for such cocycles (see Theorem \ref{att} below).
   When a cocycle  is $T$-periodic,   we   will   show    
   its  $\cald$-pullback attractor  is  also $T$-periodic  
   under some conditions (see 
      Theorems \ref{periodatt} and \ref{periodatt2}). 
    In order   to characterize the structures of pullback attractors, 
    we    introduce   the concept of $\cald$-complete orbit of
     cocycle. As demonstrated by
    Theorems \ref{att}  and \ref{periodatt},  the structure 
    of a $\cald$-pullback attractor  is fully determined by 
    the $\cald$-complete orbits.
    It seems that the
     concept of $\cald$-complete orbit 
     of  cocycle    was not 
     introduced in the literature before,  and was  not  used to
     describe   the 
      structures of $\cald$-pullback  random attractors.

Section 3 is devoted to applications of Theorems \ref{att},
 \ref{periodatt} and \ref{periodatt2}
to the cocycles generated by differential equations with 
deterministic and random forcing terms. Particularly,  we 
discuss how to choose the parametric space $\Omega_1$ 
for non-autonomous deterministic terms so that the solution 
operators of an equation can be formulated into the setting of 
cocycles presented in Section 2.  As we will see,  there are at 
least two options for choosing the space $\Omega_1$. We may
 take  $\Omega_1$   either
as the set of all  initial times (i.e., $\Omega_1 =\R$) or as the 
set of all translations of the deterministic terms. We further 
demonstrate  that the results
obtained by these two approaches are essentially   the same.
By appropriately choosing $\Omega_1$  and $\Omega_2$, we 
show  that the classical existence results of attractors are  
covered by Theorems \ref{att}  and \ref{periodatt} as special cases. 
 We will emphasize 
the    definitions of
$\cald$-complete orbits for different parametric spaces.

In the last section,  we prove the existence of $\cald$-pullback attractors  for
equation \eqref{intro1} on $\R^n$ by using the results obtained in Sections  2
and 3.  Since the domain $\R^n$ is unbounded and Sobolev embeddings 
are not compact in this case,  we have to appeal to the idea of uniform 
estimates on the tails of solutions to establish   the $\cald$-pullback  
asymptotic compactness of the equation. This idea  can be found in \cite{wan1}
for autonomous equations and in \cite{wan4} for non-autonomous 
deterministic equations and
in \cite{bat1, bat2} for random equations with only stochastic  forcing terms.
The existence of $\cald$-pullback attractor for
  equation \eqref{intro1} presented in Section 4  is an extension of
   those results from \cite{bat2, wan4}.

Throughout the rest of this paper,    we use  
$\| \cdot \|$ and $(\cdot, \cdot)$ to denote  the norm and the inner product
of $\ltwo$,  respectively.   The norm of  $L^p(\R^n)$ is written   as
 $\| \cdot\|_{p}$   and the norm of a       Banach space $X$  is written as    $\|\cdot\|_{X}$.
   The letters $c$ and $c_i$ ($i=1, 2, \ldots$)
are  generic positive constants  which may change their  values occasionally.

\section{Pullback Attractors  for Cocycles}
\setcounter{equation}{0}

In this section,  we discuss the theory of pullback attractors  
for random dynamical systems which is applicable to differential 
equations with both non-autonomous  deterministic and random terms.
This is an extension of the theory  
for random  systems  with only stochastic terms as developed 
in \cite{bat1, cra1, cra2, fla1, schm1}.
Our treatment  is closest to that of \cite{bat1, fla1}.
The reader is also referred to  
\cite{bab1, bal2, hal1, sel1, tem1}  for the attractors theory for
 autonomous dynamical systems.

Let $\Omega_1$ be a nonempty  set,   $(\Omega_2, \calftwo, P)$
be a probability space, and    $(X, d)$   be  a complete
separable  metric space with  Borel $\sigma$-algebra $\calb (X)$.
If $x \in X$ and  $B \subseteq X$, we write
$d(x, B) = \inf \{ d(x, b):  b \in B\}$. If $B \subseteq X$ and
$C\subseteq X$, we  write
$d(C, B) = \sup \{ d(x, B):   x \in C\}$
for  the   Hausdorff semi-metric between $C$ and $B$.
 Given $\varepsilon>0$ and $B\subseteq X$, the open $\varepsilon$-neighborhood of
$B$ in $X$  is denoted by
$$
\caln _\varepsilon (B) = \{ x \in X:  d(x, B) < \varepsilon \}.
$$

Let $2^X$ be the collection of all subsets of $X$. A set-valued mapping
$K: \Omega_1 \times \Omega_2 \to  2^X$   is called measurable
with respect to $\calftwo$
in $\Omega_2$
if  the value  $K(\omega_1, \omega_2)$
is a closed  nonempty subset  of $X$
for all $\omega_1 \in \Omega_1$ and $\omega_2 \in \Omega_2$,
 and  the mapping
$ \omega_2 \in  \Omega_2
 \to d(x, K(\omega_1, \omega_2) )$
is $(  \calftwo, \ \calb(\R) )$-measurable
for every  fixed $x \in X$ and $\omega_1 \in \Omega_1$.
In this case, we also say  the family,
$\{K(\omega_1, \omega_2): \omega_1 \in \Omega_1, \omega_2 \in \Omega_2 \}$,
consisting of  all values of $K$,   is measurable
with respect to $\calftwo$
 in $\Omega_2$.

Suppose  that there are two groups
$\{\theta_1 (t)\}_{t \in \R} $ and
$\{\theta_2 (t)\}_{t \in \R} $ acting on $\Omega_1$ and $\Omega_2$,
respectively. More precisely,   $\theta_1 : \R \times \Omega_1 \to \Omega_1$
  is a    mapping
   such that $\theta_1(0, \cdot) $ is the
identity on $\Omega_1$, $\theta_1 (s+t, \cdot) = \theta_1 (t, \cdot) \circ \theta_1 (s, \cdot)$ for all
$t, s \in \R$.
Similarly,   $\theta_2 : \R \times \Omega_2 \to \Omega_2$
is a    $(\calb (\R) \times \calftwo, \calftwo)$-measurable mapping
 such that $\theta_2(0,\cdot) $ is the
identity on $\Omega_2$, $\theta_2 (s+t,\cdot) = \theta_2 (t,\cdot) \circ \theta_2 (s,\cdot)$ for all
$t, s \in \R$
and $P\theta_2 (t,\cdot)  =P$
for all $t \in \R$.
For convenience, we often write
$\theta_1 (t, \cdot)$ and
$\theta_2 (t, \cdot)$ as $\thonet$
and $\thtwot$, respectively.
In the sequel, we will call both
$(\Omega_1,  \{\thonet\}_{t \in \R})$
and
$(\Omega_2, \calftwo, P,  \{\thtwot\}_{t \in \R})$
a parametric   dynamical system.

\begin{defn} \label{ds1}
 Let
$(\Omega_1,  \{\thonet\}_{t \in \R})$
and
$(\Omega_2, \calftwo, P,  \{\thtwot\}_{t \in \R})$
be parametric  dynamical systems.
A mapping $\Phi$: $ \R^+ \times \Omega_1 \times \Omega_2 \times X
\to X$ is called a continuous  cocycle on $X$
over $(\Omega_1,  \{\thonet\}_{t \in \R})$
and
$(\Omega_2, \calftwo, P,  \{\thtwot\}_{t \in \R})$
if   for all
  $\omega_1\in \Omega_1$,
  $\omega_2 \in   \Omega_2 $
  and    $t, \tau \in \R^+$,  the following conditions are satisfied:
\begin{itemize}
\item [(i)]   $\Phi (\cdot, \omega_1, \cdot, \cdot): \R ^+ \times \Omega_2 \times X
\to X$ is
 $(\calb (\R^+)   \times \calftwo \times \calb (X), \
\calb(X))$-measurable;

\item[(ii)]    $\Phi(0, \omega_1, \omega_2, \cdot) $ is the identity on $X$;

\item[(iii)]    $\Phi(t+\tau, \omega_1, \omega_2, \cdot) = \Phi(t, \theta_{1,\tau} 
\omega_1,  \theta_{2,\tau} \omega_2, \cdot)
 \circ \Phi(\tau, \omega_1, \omega_2, \cdot)$;

\item[(iv)]    $\Phi(t, \omega_1, \omega_2,  \cdot): X \to  X$ is continuous.
    \end{itemize}
\end{defn}

In this section, we always assume that
$\Phi$ is a continuous  cocycle on $X$
over $(\Omega_1,  \{\thonet\}_{t \in \R})$
and
$(\Omega_2, \calftwo, P,  \{\thtwot\}_{t \in \R})$.
The purpose of this section is to study the dynamics of   continuous cocycles  on  $X$,
including the periodic  ones. The periodic cocycles are defined as follows.

\begin{defn}
\label{periodic-ds1}
Let  $T$ be a positive number and
$\Phi$  be  a continuous  cocycle on $X$
over $(\Omega_1,  \{\thonet\}_{t \in \R})$
and
$(\Omega_2, \calftwo, P,  \{\thtwot\}_{t \in \R})$.
We say $\Phi$ is a  continuous periodic  cocycle  on $X$ with period $T$
if for every $t\ge 0$, $\omega_1 \in \Omega_1$  and $\omega_2 \in \Omega_2$,
there holds
$$
\Phi(t, \theta_{1, T} \omega_1, \omega_2, \cdot)
= \Phi(t, \omega_1,  \omega_2, \cdot ).
$$
\end{defn}

Let $D$ be   a family of some subsets of $X$
which is  parametrized by $(\omega_1, \omega_2) \in \Omega_1 \times \Omega_2$:
  $D =\{D(\omega_1, \omega_2 ) \subseteq X:    \omega_1 \in \Omega_1, \
  \omega_2 \in \Omega_2 \}  $. Then   we can associate with $D$
a set-valued map $f_D: \Omega_1 \times \Omega_2
\to 2^X$   such that
$f_D(\omega_1, \omega_2) = D(\omega_1, \omega_2)$
for all  $\omega_1 \in \Omega_1$ and $
  \omega_2 \in \Omega_2$.
  It is clear that  $D = f_D (\Omega_1 \times \Omega_2)$.
  We say two  such families  $B$  and $D$
    are equal if and only if the corresponding maps
  $f_B$ and $f_D$   are equal. In other words, we have the  following definition.

  \begin{defn}
  \label{familyequal}
  Let $B$ and $D$
be  two families   of    subsets of $X$ which  are  
parametrized by $(\omega_1, \omega_2) \in \Omega_1 \times \Omega_2$.
Then  $B$ and $D$  are said to be equal      if
$B(\omega_1, \omega_2)
= D(\omega_1, \omega_2)$ for all $\omega_1 \in \Omega_1$
and $\omega_2 \in \Omega_2$.
\end{defn}

In the sequel, we use  $\cald$ to denote
 a  collection  of  some families of  nonempty subsets of $X$:
\be
\label{defcald}
{\cald} = \{ D =\{\emptyset\neq  D(\omega_1, \omega_2 ) \subseteq X:
  \omega_1 \in \Omega_1, \
  \omega_2 \in \Omega_2\}: f_D
  \mbox{ satisfies  some  conditions} \}.
\ee
Note  that a family $D$ belongs to $\cald$  if and only   the 
corresponding map $f_D$  satisfies  certain conditions  rather than
the image $f_D(\Omega_1 \times \Omega_2)$  of $f_D$.
We will discuss existence of   attractors for $\Phi$   which pullback  attract
all elements of $\cald$. To this end, we require that the collection
$\cald$  is neighborhood closed in the following sense.

\begin{defn} 
\label{defepsneigh1}
A collection $\cald$ of some families 
of nonempty subsets of $X$
is said  to be   neighborhood closed if   for each
$D=\{D(\omega_1, \omega_2): 
\omega_1 \in \Omega_1, \omega_2 \in \Omega_2 \}
\in \cald$,   there exists a positive number
$\varepsilon$ depending on $D$ such that  the family
\be\label{defepsneigh2}
 \{ {B}(\omega_1, \omega_2) :
 {B}(\omega_1, \omega_2) \mbox{ is a  nonempty subset of }
 \caln_\varepsilon ( D (\omega_1, \omega_2) ),  \forall \
 \omega_1 \in \Omega_1,  \forall\  \omega_2
\in  \Omega_2\}
\ee
also belongs to $\cald$.
\end{defn}

 Note that the   neighborhood closedness  of $\cald$
 implies    for each 
$D\in \cald$,   
\be
\label{inclu1}
 \{ \tilde{D}(\omega_1, \omega_2) :
 \tilde{D}(\omega_1, \omega_2) \mbox{ is a  nonempty subset of }
  D (\omega_1, \omega_2 ),  \forall \
 \omega_1 \in \Omega_1,  \forall\  \omega_2
\in  \Omega_2\} \in \cald.
\ee 
A collection $\cald$ satisfying  \eqref{inclu1} 
 is said to be inclusion-closed in the literature, see, e.g., \cite{fla1}.
The concept of  inclusion-closedness  of  $\cald$  is used to find  a sufficient
condition for existence of a $\cald$-pullback attractor. In the present  paper,
we want to investigate  both sufficient and necessary  criteria  for existence of
such attractors.  As  we will see later,  we need  the concept of 
neighborhood closedness  of $\cald$ 
 in order to  derive a necessary condition for  the  existence of 
$\cald$-pullback attractors.
It is evident that  the  neighborhood  closedness of $\cald$  implies 
the  inclusion-closedness of $\cald$.

When we study the dynamics of periodic cocycles, we require that
the collection $\cald$ is closed under translations as defined below.
Let $T \in \R$    and $D=\{D(\omega_1, \omega_2):
\omega_1 \in \Omega_1, \omega_2 \in \Omega_2 \} \in \cald$.
For each $ \omega_1 \in \Omega_1$ and $\omega_2 \in \Omega_2$,  define a set
$D_T(\omega_1, \omega_2)$ by
\be\label{DTfamily1}
D_T(\omega_1, \omega_2) = D(\theta_{1, T} \omega_1,  \omega_2) .
\ee  Let  $D_T$ be the collection of all $ D_T(\omega_1, \omega_2)$
as defined by \eqref{DTfamily1}, that is,
\be
\label{DTfamily2}
D_T = \{ D_T(\omega_1, \omega_2): D_T(\omega_1, \omega_2) \mbox{ is given by }
 \eqref{DTfamily1}, \omega_1 \in \Omega_1,
\omega_2 \in \Omega_2
\}.
\ee
Then the family $D_T$ given by \eqref{DTfamily2} is called
the  $T$-translation of
$D \in \cald$. Note that, as sets, $D_T$  and $D$ are equal for any
$D \in \cald$. However, in general,  $D_T$  and $D$ are not equal
in terms of Definition \ref{familyequal}.

\begin{defn} 
\label{defTlation}
Suppose   $T  \in \R$     and
  $\cald$  is  a collection  of some families of nonempty subsets of $X$
  as given by \eqref{defcald}.
  Let $\cald_T$ be the collection of
    $T$-translations of  all elements of $\cald$, that is,
  $$
  \cald_T = \{ D_T: D_T \mbox{ is the }  T   \mbox{-translation of } D,
  \  D \in \cald \}.
  $$
  Then $\cald_T$ is called the $T$-translation of $\cald$.
  If $\cald_T \subseteq \cald$,   we say $\cald$ is
      $T$-translation   closed.
      If  $\cald_T =  \cald$,   we say $\cald$ is
      $T$-translation  invariant.
\end{defn}

Regarding $T$-translations, we have the following lemma.

\begin{lem}
\label{Tlation1}
Suppose   $T  \in \R$     and
  $\cald$  is  a collection  of some families of nonempty subsets of $X$
  as given by \eqref{defcald}. 
  Then $\cald_{-T} \subseteq \cald$ if and only if
  $\cald \subseteq \cald_T$. 
  \end{lem}

\begin{proof}
It follows from Definition \ref{defTlation} that, for every $B \in \cald$,
\be\label{ptlation2}
(B_T)_{-T} = ( B_{-T})_T = B,
\ee
where $B_T$   and $B_{-T}$ are $T$-translation and $-T$-translation
of $B$, respectively.
Note that \eqref{ptlation2}
should   be understood in the sense of Definition
\ref{familyequal}.
Suppose $\cald_{-T} \subseteq \cald$ and  $B \in \cald$.
Then we have $B_{-T} \in \cald$ and hence 
$(B_{-T})_T \in \cald_T$, which along with
\eqref{ptlation2} shows that
$B \in  \cald_T$.  Since $B$ is an arbitrary element
of $\cald$ we get 
$\cald \subseteq \cald_T$. 

We now suppose $\cald \subseteq \cald_T$ and
$B \in \cald_{-T}$.  Then there exists
$E \in \cald$   such that $B = E_{-T}$.
By $E \in \cald$  and $\cald \subseteq \cald_T$,  we find that there exists
$G \in \cald$   such that $E= G_T$.
It follows from
\eqref{ptlation2} that
$B= E_{-T} = (G_T)_{-T} =G$. 
Since $ G \in \cald$ we  have
 $B \in \cald$. 
 Since $B$ is an arbitrary element
 of $\cald_{-T}$ we get 
     $\cald_{-T} \subseteq \cald$.
     This completes the proof.
\end{proof}

As an immediate consequence of Lemma \ref{Tlation1} and Definition
\ref{defTlation}, we have the following invariance criterion for $\cald$.
\begin{lem}
\label{Tlation2}
Suppose   $T  \in \R$     and
  $\cald$  is  a collection  of some families of nonempty subsets of $X$
  as given by \eqref{defcald}. Then $\cald$ is $T$-translation  invariant
  if    and only if $\cald$ is both $T$-translation closed   and $-T$-translation 
  closed. 
  \end{lem}

 For later purpose, we need the concept of a complete orbit
 of $\Phi$ which is given below.

\begin{defn}
\label{comporbit}
 Let $\cald$ be a collection of some families of
 nonempty  subsets of $X$. A mapping $\psi: \R \times \Omega_1 \times \Omega_2$
 $\to X$ is called a complete orbit of $\Phi$ if for every $\tau \in \R$, $t \ge 0$,
 $\omega_1 \in \Omega_1$ and $\omega_2 \in \Omega_2$,  the following holds:
\be
\label{comporbit1}
 \Phi (t, \theta_{1, \tau} \omega_1, \theta_{2, \tau} \omega_2,
  \psi (\tau, \omega_1, \omega_2) )
  = \psi (t + \tau, \omega_1, \omega_2 ).
\ee
 If, in  addition,    there exists $D=\{D(\omega_1, \omega_2): \omega_1 \in \Omega,
 \omega_2 \in \Omega_2 \}\in \cald$ such that
 $\psi(t, \omega_1, \omega_2)$ belongs to
 $D(\theta_{1,t} \omega_1, \theta_{2, t} \omega_2 )$
 for every  $t \in \R$, $\omega_1 \in \Omega_1$
 and $\omega_2 \in \Omega_2$, then $\psi$ is called a
 $\cald$-complete orbit of $\Phi$.
 \end{defn}

The following are  the concepts regarding invariance
of a family of subsets of $X$.

\begin{defn}
Let $B=\{B(\omega_1, \omega_2): \omega_1 \in \Omega_1, \ \omega_2  \in \Omega_2\}$
be a family of nonempty subsets of $X$.
We say that $B$ is positively  invariant under $\Phi$ if
$$ \Phi(t, \omega_1, \omega_2, B(\omega_1, \omega_2)   )
\subseteq B (\theta_{1,t} \omega_1, \theta_{2,t} \omega_2
), \ \mbox{ for all } t \ge 0, \  \omega_1 \in \Omega_1, \
\omega_2 \in \Omega_2.
$$
 We say that $B$ is    invariant under $\Phi$ if
$$ \Phi(t, \omega_1, \omega_2, B(\omega_1, \omega_2)   )
= B (\theta_{1,t} \omega_1, \theta_{2,t} \omega_2
), \ \mbox{ for all } t \ge 0, \  \omega_1 \in \Omega_1, \
\omega_2 \in \Omega_2.
$$
We say that $B$ is  quasi-invariant  under $\Phi$ if
 for each map $y: \Omega_1 \times \Omega_2 \to  X$   with
 $y(\omega_1, \omega_2) \in B( \omega_1, \omega_2)$
for  all   $\omega_1 \in \Omega_1$
and $\omega_2 \in \Omega_2$,  there exists a complete
orbit $\psi$ such that $\psi(0, \omega_1, \omega_2)
= y(\omega_1, \omega_2)$
and $\psi (t, \omega_1, \omega_2) \in B( \theta_{1, t} \omega_1,
\theta_{2, t} \omega_2 )$ for all $t \in \R$,
$\omega_1 \in \Omega_1$
and $\omega_2 \in \Omega_2$.
\end{defn}

It can be proved that   a family of nonempty subsets of $X$
is invariant under $\Phi$ if and only if the family is
positively invariant and quasi-invariant, which is given by the
following lemma.

\begin{lem}
\label{posquainv}
Let $B=\{B(\omega_1, \omega_2): \omega_1 \in \Omega_1, \ \omega_2  \in \Omega_2\}$
be a family of nonempty subsets of $X$.

(i)  If $B$ is quasi-invariant under $\Phi$,  then
\be
\label{pposquainv1}
 B (\theta_{1,t} \omega_1, \theta_{2,t} \omega_2)
\subseteq  \Phi(t, \omega_1, \omega_2, B(\omega_1, \omega_2)   )
, \ \mbox{ for all } t \ge 0, \  \omega_1 \in \Omega_1, \
\omega_2 \in \Omega_2.
\ee

(ii) If $B$ is invariant under $\Phi$,  then $B$ is quasi-invariant  under $\Phi$.

(iii)    $B$ is
invariant under $\Phi$ if and only if $B$  is
positively invariant and quasi-invariant.
\end{lem}

\begin{proof} (i).
Suppose $B$ is quasi-invariant under $\Phi$.    We want to show
\eqref{pposquainv1}.
Let $t\ge 0$ be fixed.
For each  $\omega_1 \in \Omega_1$
and $\omega_2 \in \Omega_2$, let
$y(\omega_1, \omega_2)$ be an arbitrary
element in $ B(\omega_1, \omega_2)$.
Then the quasi-invariance of $B$ implies that
there exists a complete
orbit $\psi$ such that $\psi(0, \omega_1, \omega_2)
= y(\omega_1, \omega_2)$
and $\psi (s, \omega_1, \omega_2) \in B( \theta_{1, s} \omega_1,
\theta_{2, s} \omega_2 )$ for all $s \in \R$, for all $\omega_1  \in \Omega_1$
and $\omega_2 \in \Omega_2$.
Furthermore,  for each   $\omega_1 \in \Omega_1$
and $\omega_2 \in \Omega_2$, we have
$$
\Phi(t, \theta_{1,-t} \omega_1, \theta_{2,-t} \omega_2,
\psi(-t, \omega_1, \omega_2)) = \psi(0, \omega_1, \omega_2),
$$
which implies that
 \be
\label{pposquainv2}
\Phi(t,   \omega_1,   \omega_2,
\psi(-t, \theta_{1, t} \omega_1, \theta_{2, t} \omega_2))
 = \psi(0, \theta_{1, t} \omega_1, \theta_{1, t} \omega_2).
\ee
Since
$\psi(0,  \theta_{1, t}\omega_1,  \theta_{2, t}\omega_2)
= y( \theta_{1, t}\omega_1,  \theta_{1, t}\omega_2)$
and $\psi (-t,  \theta_{1, t} \omega_1, \theta_{2, t} \omega_2
 ) \in B(  \omega_1,
  \omega_2 )$,  we
  obtain \eqref{pposquainv1}
  from \eqref{pposquainv2}
  immediately.
  
  (ii).   Suppose  that
$B$ is invariant and    $y: \Omega_1 \times \Omega_2 \to  X$  is
 a mapping such that
 $y(\omega_1, \omega_2) \in B( \omega_1, \omega_2)$
for all   $\omega_1 \in \Omega_1$
and $\omega_2 \in \Omega_2$.  Then we need to show
that  there exists a complete
orbit $\psi$ such that $\psi(0, \omega_1, \omega_2)
= y(\omega_1, \omega_2)$
and $\psi (t, \omega_1, \omega_2) \in B( \theta_{1, t} \omega_1,
\theta_{2, t} \omega_2 )$ for all $t \in \R$,
$\omega_1 \in \Omega_1$
and $\omega_2 \in \Omega_2$.
Fix $\omega_1 \in \Omega_1$  and $\omega_2 \in \Omega_2$. We first construct
 a mapping $\psi(\cdot, \omega_1, \omega_2): \R \to X$
 with desired properties.
 By the invariance of $B$ we have, for all $t \ge 0$, ${\tilde{\omega}}_1 \in \Omega_1$
 and ${\tilde{\omega}}  _2 \in \Omega_2$,
 \be
 \label{pposquainv3}
 \Phi (t, \theta_{1, -t} {\tilde{\omega}}_1, \theta_{2, -t}  {\tilde{\omega}}   _2,
 B(\theta_{1, -t}  {\tilde{\omega}}   _1, \theta_{2, -t}  {\tilde{\omega}}  _2) )
 =B(  {\tilde{\omega}}  _1,   {\tilde{\omega}}  _2).
 \ee
 Since $y(\omega_1,  \omega_2) \in B(\omega_1,  \omega_2)$,
 it follows from \eqref{pposquainv3} with $t =1$  that
 there exists $z_1 \in B(\theta_{1, -1} \omega_1, \theta_{2, -1} \omega_2)$
 such that
 \be
 \label{pposquainv4}
 \Phi (1, \theta_{1, -1} \omega_1, \theta_{2, -1} \omega_2,
 z_1 ) =y( \omega_1,  \omega_2).
 \ee
 By \eqref{pposquainv4} and the cocycle property of $\Phi$
  we get that, for all $t \ge 1$,
 $$
 \Phi (t, \theta_{1, -1} \omega_1, \theta_{2, -1} \omega_2,
 z_1 )
 = \Phi (t-1,  \omega_1,   \omega_2,
 \Phi (1, \theta_{1, -1} \omega_1, \theta_{2, -1} \omega_2,
 z_1 ) )
 = \Phi (t-1,  \omega_1,   \omega_2, y(\omega_1, \omega_2)).
 $$
 Note that  the invariance of $B$  implies
 $\Phi (t, \theta_{1, -1} \omega_1, \theta_{2, -1} \omega_2,
 z_1 ) \in B(\theta_{1, t-1} \omega_1, \theta_{2, t-1} \omega_2)$.
 Replacing $ {\tilde{\omega}}  _1$   and $ {\tilde{\omega}}   _2$ in \eqref{pposquainv3}
 by $\theta_{1, -1} \omega_1$   and $\theta_{2, -1}\omega_2$,
 respectively, we get, for all $ t \ge 0$,
 \be
 \label{pposquainv5}
 \Phi (t, \theta_{1, -t-1} \omega_1, \theta_{2, -t-1} \omega_2,
 B(\theta_{1, -t-1} \omega_1, \theta_{2, -t-1} \omega_2) )
 =B( \theta_{1, -1}\omega_1,  \theta_{2, -1}\omega_2).
 \ee
 Since $z_1  \in B(\theta_{1, -1}\omega_1,  \theta_{2, -1}\omega_2)$,
   from \eqref{pposquainv5} with $t =1$  we  find  that
 there exists $z_2 \in B(\theta_{1, -2} \omega_1, \theta_{2, -2} \omega_2)$
 such that
$$
 \Phi (1, \theta_{1, -2} \omega_1, \theta_{2, -2} \omega_2,
 z_2 ) =z_1,
$$
which along with  the cocycle property of $\Phi$
implies that,  for all $t \ge 1$,
 $$
 \Phi (t, \theta_{1, -2} \omega_1, \theta_{2, -2} \omega_2,
 z_2 )
 = \Phi (t-1,  \theta_{1, -1}\omega_1,  \theta_{2, -1} \omega_2, z_1 ).
 $$
The invariance of $B$  again  implies
 $\Phi (t, \theta_{1, -2} \omega_1, \theta_{2, -2} \omega_2,
 z_2 ) \in B(\theta_{1, t-2} \omega_1, \theta_{2, t-2} \omega_2)$.
 Proceeding inductively,
  we find that for every $n=1,2, \ldots$,
 there exists $z_n \in B(\theta_{1, -n} \omega_1, \theta_{2, -n} \omega_2)$
 such that
 $\Phi (n, \theta_{1, -n} \omega_1, \theta_{2,-n} \omega_2 , z_n)
 =y(\omega_1, \omega_2)$ and
 for all $t \ge 0$,
 \be\label{pposquainv8}
 \Phi (t, \theta_{1, -n} \omega_1, \theta_{2,-n} \omega_2 , z_n )
 \in B( \theta_{1, t-n} \omega_1, \theta_{2, t-n}\omega_2 ).
\ee
 In addition,  for all $t \ge 1$, we have
 $$
 \Phi (t-1, \theta_{1, -n+1} \omega_1, \theta_{2, -n+1} \omega_2, z_{n-1} )
 =\Phi (t, \theta_{1, -n} \omega_1, \theta_{2,-n} \omega_2 , z_n ).
 $$
Define a map  $\psi(\cdot, \omega_1, \omega_2)$ from $ \R$
into  $ X$ such that
 for every $t \in \R$,  $\psi(t, \omega_1, \omega_2)$
 is  the common value of
 $\Phi(t+n, \theta_{1, -n} \omega_1, \theta_{2,-n} \omega_2, z_n )$
 for all $n \ge -t$.
 Then we have
 $\psi (0, \omega_1, \omega_2)  =y(\omega_1, \omega_2)$
 and
 $\psi (t, \omega_1, \omega_2) 
  \in B( \theta_{1, t} \omega_1, \theta_{2, t}\omega_2 )$ 
    by \eqref{pposquainv8}.
 If  $t \ge 0$,  $\tau \in \R$ and
   $n \ge -\tau$, by  definition
 we  find that
 $$
 \Phi(t, \theta_{1, \tau}\omega_1, \theta_{2, \tau}\omega_2,
 \psi(\tau, \omega_1, \omega_2) )
 =
 \Phi(t, \theta_{1, \tau}\omega_1, \theta_{2, \tau}\omega_2,
  \Phi(\tau +n, \theta_{1, -n} \omega_1, \theta_{2,-n} \omega_2, z_n ))
 $$
 $$
  =
  \Phi(t+ \tau +n, \theta_{1, -n} \omega_1, \theta_{2,-n} \omega_2, z_n )
  =\psi(t+\tau, \omega_1, \omega_2).
$$
So, $\psi$ is a complete orbit of $\Phi$ with desired properties
and thus $B$ is    quasi-invariant.

Note that  (iii)  follows    from (i)  and (ii) immediately and thus the proof is completed.
\end{proof}

\begin{defn}
\label{defomlit}
Let $B=\{B(\omega_1, \omega_2): \omega_1 \in \Omega_1, \ \omega_2  \in \Omega_2\}$
be a family of nonempty subsets of $X$.
For every $\omega_1 \in \Omega_1$ and
$\omega_2 \in \Omega_2$,  let
\be\label{omegalimit}
\Omega (B, \omega_1, \omega_2)
= \bigcap_{\tau \ge 0}
\  \overline{ \bigcup_{t\ge \tau} \Phi(t, \theta_{1,-t} \omega_1, 
\theta_{2, -t} \omega_2, B(\theta_{1,-t} \omega_1, \theta_{2,-t}\omega_2  ))}.
\ee
Then
the  family
 $\{\Omega (B, \omega_1, \omega_2): \omega_1 \in \Omega_1, \omega_2 \in \Omega_2 \}$
 is called the $\Omega$-limit set of $B$
 and is denoted by $\Omega(B)$.
 \end{defn}

 \begin{rem}
 \label{defomegalimitrem1}
 Given $\omega_1 \in \Omega_1$  and
  $\omega_2 \in \Omega_2$,
 it follows from  \eqref{omegalimit} that
 a point $y \in X$ belongs to
 $\Omega(B, \omega_1, \omega_2)$ if and only if  there exist two  sequences
 $t_n \to \infty$ and $x_n \in B(\theta_{1, -t_n}\omega_1, \theta_{2,  -t_n}\omega_2)$
 such that
 $$
 \lim_{n \to \infty} \Phi(t_n, \theta_{1, -t_n}\omega_1,
 \theta_{2, -t_n} \omega_2, x_n)
 = y.
 $$
 \end{rem}

\begin{defn}
Let $\cald$ be a collection of some families of nonempty subsets of $X$ and
$K=\{K(\omega_1, \omega_2): \omega_1 
\in \Omega_1, \ \omega_2  \in \Omega_2\} \in \mathcal{D}$. Then
$K$  is called a  $\cald$-pullback
 absorbing
set for   $\Phi$   if
for all $\omega_1 \in \Omega_1$,
$\omega_2 \in \Omega_2 $
and  for every $B \in \cald$,
 there exists $T= T(B, \omega_1, \omega_2)>0$ such
that
\be
\label{abs1}
\Phi(t, \theta_{1,-t} \omega_1, \theta_{2, -t} \omega_2, 
B(\theta_{1,-t} \omega_1, \theta_{2,-t} \omega_2  )) 
 \subseteq  K(\omega_1, \omega_2)
\quad \mbox{for all} \ t \ge T.
\ee
If, in addition, for all $\omega_1 \in \Omega_1$ and $\omega_2 \in \Omega_2$,
   $K(\omega_1, \omega_2)$ is a closed nonempty subset of $X$
   and $K$ is measurable with respect to the $P$-completion of $\calftwo$
   in $\Omega_2$,
 then we say $K$ is a  closed measurable
  $\cald$-pullback absorbing  set for $\Phi$.
\end{defn}

\begin{defn}
\label{asycomp}
 Let $\cald$ be a collection of  some families of  nonempty
 subsets of $X$.
 Then
$\Phi$ is said to be  $\cald$-pullback asymptotically
compact in $X$ if
for all $\omega_1 \in \Omega_1$ and
$\omega_2 \in \Omega_2$,    the sequence
\be
\label{asycomp1}
\{\Phi(t_n, \theta_{1, -t_n} \omega_1, \theta_{2, -t_n} \omega_2,
x_n)\}_{n=1}^\infty \mbox{  has a convergent  subsequence  in }   X
\ee
 whenever
  $t_n \to \infty$, and $ x_n\in   B(\theta_{1, -t_n}\omega_1,
  \theta_{2, -t_n} \omega_2 )$   with
$\{B(\omega_1, \omega_2): \omega_1 \in \Omega_1, \ \omega_2 \in \Omega_2
\}   \in \mathcal{D}$.
\end{defn}

\begin{defn}
\label{defatt}
 Let $\cald$ be a collection of some families of
 nonempty  subsets of $X$
 and
 $\cala = \{\cala (\omega_1, \omega_2): \omega_1 \in \Omega_1,
  \omega_2 \in \Omega_2 \} \in \cald $.
Then     $\cala$
is called a    $\cald$-pullback    attractor  for
  $\Phi$
if the following  conditions (i)-(iii) are  fulfilled:
\begin{itemize}
\item [(i)]   $\cala$ is measurable
with respect to the $P$-completion of $\calftwo$ in $\Omega_2$ and
 $\cala(\omega_1, \omega_2)$ is compact for all $\omega_1 \in \Omega_1$
and    $\omega_2 \in \Omega_2$.

\item[(ii)]   $\cala$  is invariant, that is,
for every $\omega_1 \in \Omega_1$ and
 $\omega_2 \in \Omega_2$,
$$ \Phi(t, \omega_1, \omega_2, \cala(\omega_1, \omega_2)   )
= \cala (\theta_{1,t} \omega_1, \theta_{2,t} \omega_2
), \ \  \forall \   t \ge 0.
$$

\item[(iii)]   $\cala  $
attracts  every  member   of   $\cald$,  that is, for every
 $B = \{B(\omega_1, \omega_2): \omega_1 \in \Omega_1, \omega_2 \in \Omega_2\}
 \in \cald$ and for every $\omega_1 \in \Omega_1$ and
 $\omega_2 \in \Omega_2$,
$$ \lim_{t \to  \infty} d (\Phi(t, \theta_{1,-t}\omega_1,
 \theta_{2,-t}\omega_2, B(\theta_{1,-t}\omega_1, \theta_{2,-t}\omega_2) ) , 
 \cala (\omega_1, \omega_2 ))=0.
$$
 \end{itemize}
 If, in addition, there exists $T>0$ such that
 $$
 \cala(\theta_{1, T} \omega_1, \omega_2) = \cala(\omega_1,    \omega_2 ),
 \quad \forall \  \omega_1 \in \Omega_1, \forall \
  \omega_2 \in \Omega_2,
 $$
 then we say $\cala$ is periodic with period $T$.
\end{defn}

From   Definition \ref{defatt},   the following observation is evident.

\begin{rem}\label{defattrem1}
 The  $\cald$-pullback attractor of $\Phi$, if it exists,
 must be unique within the collection $\cald$. In other words,
   if $\{\cala (\omega_1, \omega_2) :  \omega_1 \in \Omega_1,
   \omega_2 \in \Omega_2 \}   \in \cald $
and
$\{\mathcal{\tilde{A}}(\omega_1, \omega_2): \omega_1 
\in \Omega, \omega_2 \in \Omega_2  \}  \in \cald $ are  both
$\cald$-pullback    attractors  for
  $\Phi$, then we must have
  $\cala (\omega_1, \omega_2 )=
  {\mathcal{\tilde{A}}} (\omega_1, \omega_2)$ for all $\omega_1 \in \Omega_1$
  and  $\omega_2 \in \Omega_2$.
  This follows immediately from conditions (i)-(iii) in Definition \ref{defatt}.
  Furthermore, If $\cala$ is a periodic  $\cald$-pullback attractor of $\Phi$
  with period $T>0$, then we have
  $$
  \cala(\theta_{1, -T} \omega_1, \omega_2) = \cala(\omega_1,    \omega_2 ),
 \quad \forall \  \omega_1 \in \Omega_1, \forall \
  \omega_2 \in \Omega_2,
  $$
  \end{rem}

We next discuss the properties of
   $\Omega$-limit sets of a family of subsets of $X$.

\begin{lem}
\label{omegacomp}
 Let $\cald$ be a  collection of some  families of   nonempty subsets of
$X$ and $\Phi$  be a continuous   cocycle on $X$
over $(\Omega_1,  \{\thonet\}_{t \in \R})$
and
$(\Omega_2, \calftwo, P,  \{\thtwot\}_{t \in \R})$.
Suppose $\Phi$ is $\cald$-pullback asymptotically compact in $X$.
Then the following statements are true:

(i) For every $D \in \cald$,  $\omega_1 \in \Omega_1$ and
$\omega_2 \in \Omega_2$, the $\Omega$-limit set   $\Omega(D, \omega_1, \omega_2)$
   is nonempty, compact  and attracts $D$.

(ii) For every $D \in \cald$, the $\Omega$-limit set
$\Omega(D) = \{ \Omega (D, \omega_1, \omega_2): \omega_1 \in \Omega_1,
\omega_2 \in \Omega_2 \}$ is quasi-invariant.

(iii) Let $D=\{ D(\omega_1, \omega_2): \omega_1 \in \Omega_1, \omega_2
\in \Omega_2 \} $  belong to $\cald$ with
$D(\omega_1, \omega_2)$ being closed in $X$
for all $\omega_1 \in \Omega_1$
and $ \omega_2
\in \Omega_2 $. If  for each
$\omega_1 \in \Omega_1$ and $ \omega_2
\in \Omega_2 $,   there exists $t_0 =t_0 (D, \omega_1, \omega_2) \ge 0$ such that
$\Phi (t, \theta_{1, -t} \omega_1, \theta_{2, -t} \omega_2,
D(\theta_{1, -t} \omega_1, \theta_{2, -t} \omega_2) )
\subseteq D(\omega_1, \omega_2)$ for all $t\ge t_0$, then  the $\Omega$-limit set
$\Omega(D)  $ is invariant.
In addition, $\Omega(D, \omega_1, \omega_2) \subseteq D(\omega_1, \omega_2)$
for all $\omega_1 \in \Omega_1$ and $ \omega_2
\in \Omega_2 $.
\end{lem}

\begin{proof}
 (i). Let $t_n \to \infty$ and $x_n \in D(\theta_{1, -t_n}\omega_1,
 \theta_{2, -t_n} \omega_2)$. Then the $\cald$-pullback
 asymptotic compactness of $\Phi$ implies that there exists
 $y \in X$ such that, up to a subsequence,
 $$
 \lim_{n \to \infty}
 \Phi (t_n, \theta_{1, -t_n}\omega_1,
 \theta_{2, -t_n} \omega_2, x_n ) =y,
 $$
 which along with  Remark \ref{defomegalimitrem1}
 shows that $y \in \Omega(D, \omega_1, \omega_2)$
 and hence $\Omega(D, \omega_1, \omega_2)$
 is nonempty
 for all
 $\omega_1 \in \Omega_1$   and $\omega_2 \in \Omega_2$.

 Let $\{y_n \}_{n=1}^\infty$ be a sequence in
 $\Omega(D, \omega_1, \omega_2)$. It follows from
 Remark \ref{defomegalimitrem1} that there exist
 $t_n \to \infty$ and $x_n \in D(\theta_{1, -t_n}\omega_1,
 \theta_{2, -t_n} \omega_2)$ such that, for all $n \in \N$,
 \be
 \label{pomegacomp1}
 d(\Phi (t_n, \theta_{1, -t_n}\omega_1,
 \theta_{2, -t_n} \omega_2, x_n ),\  y_n ) \le {\frac 1n}.
 \ee
 By the $\cald$-pullback
 asymptotic compactness of $\Phi$ we find that there exists
 $y \in X$ such that, up to a subsequence,
 $$
 \lim_{n \to \infty}
 \Phi (t_n, \theta_{1, -t_n}\omega_1,
 \theta_{2, -t_n} \omega_2, x_n ) =y,
 $$
 which together with Remark \ref{defomegalimitrem1}
 and inequality \eqref{pomegacomp1} implies that
 $y \in  \Omega(D, \omega_1, \omega_2)$
 and $y_n \to y$. Therefore, $\Omega(D, \omega_1, \omega_2)$
 is compact  for all
 $\omega_1 \in \Omega_1$   and $\omega_2 \in \Omega_2$.

We now prove that for all
 $\omega_1 \in \Omega_1$   and $\omega_2 \in \Omega_2$,
 \be
 \label{pomegacomp2}
 \lim_{t \to \infty}
 d (\Phi (t, \theta_{1, -t}\omega_1, \theta_{2, -t} \omega_2,
  D(\theta_{1, -t}\omega_1, \theta_{2, -t} \omega_2) ), \ 
  \Omega (D,\omega_1,  \omega_2) ) =0.
 \ee
 Suppose \eqref{pomegacomp2}  is not true. Then there exist
 $\omega_1 \in \Omega_1$,  $\omega_2 \in \Omega_2$,
$\varepsilon_0>0$,
 $t_n \to \infty$ and   $x_n 
 \in D(\theta_{1, -t_n}\omega_1, \theta_{2, -t_n} \omega_2 )$ such that
 for all $n \in  \N$,
 \be
 \label{pomegacomp3}
 d(  \Phi(t_n, \theta_{1, -t_n}\omega_1, \theta_{2, -t_n} \omega_2, x_n )  , 
 \  \Omega(D, \omega_1, \omega_2 )) \ge \varepsilon_0.
 \ee
 By the $\cald$-pullback compactness of $\Phi$ and 
  Remark \ref{defomegalimitrem1},
 there exist $y \in \Omega(D, \omega_1, \omega_2 )$
  such that, up to a subsequence,
 $\Phi(t_n, \theta_{1, -t_n}\omega_1, \theta_{2, -t_n} \omega_2, x_n )
 \to y$. On the other hand, it follows from   \eqref{pomegacomp3}
 that  $d( y, \  \Omega(D, \omega_1, \omega_2 )) \ge \varepsilon_0$, 
 a contradiction
 with $y \in \Omega(D, \omega_1, \omega_2 )$. 
 This  proves \eqref{pomegacomp2}.

(ii).  Let    $y: \Omega_1 \times \Omega_2 \to  X$  be
 a mapping such that
 $y(\omega_1, \omega_2) \in \Omega(D, \omega_1, \omega_2)$
for all   $\omega_1 \in \Omega_1$
and $\omega_2 \in \Omega_2$.  Then we need to show
that  there exists a complete
orbit $\psi$ such that $\psi(0, \omega_1, \omega_2)
= y(\omega_1, \omega_2)$
and $\psi (t, \omega_1, \omega_2) \in \Omega(D, \theta_{1, t} \omega_1,
\theta_{2, t} \omega_2 )$ for all $t \in \R$,
$\omega_1 \in \Omega_1$
and $\omega_2 \in \Omega_2$.
Fix $\omega_1 \in \Omega_1$  and $\omega_2 \in \Omega_2$. We first construct
 a mapping $\psi_{\omega_1, \omega_2}: \R \to X$ such that
 $\psi_{\omega_1, \omega_2}(0) = y(\omega_1, \omega_2)$
and $\psi_{\omega_1, \omega_2}(t) \in \Omega(D, \theta_{1, t} \omega_1,
\theta_{2, t} \omega_2 )$ for all $t \in \R$.
Since $y(\omega_1, \omega_2) \in   \Omega(D, \omega_1, \omega_2)$,
by Remark \ref{defomegalimitrem1} we find that
there exist $t_n \to \infty$ and $x_n
 \in D(\theta_{1,-t_n}\omega_1, \theta_{2, -t_n}\omega_2 )$ 
 such that
\be
\label{pocomp1}
\lim_{n \to \infty}
\Phi (t_n, \theta_{1,-t_n}\omega_1, \theta_{2, -t_n}\omega_2 , x_n)
=y(\omega_1, \omega_2).
\ee
By \eqref{pocomp1} and the continuity of $\Phi$  we have, for all $t \ge 0$,
$$
\lim_{n \to \infty}
\Phi(t, \omega_1, \omega_2, \Phi (t_n, \theta_{1,-t_n}\omega_1, \theta_{2, -t_n}\omega_2 , x_n) )
=\Phi(t, \omega_1, \omega_2, y(\omega_1, \omega_2) ).
$$
Then the cocycle property of $\Phi$ implies that,  for all $t \ge 0$,
\be
\label{pocomp2}
\lim_{n \to \infty}
\Phi(t_n +t,  \theta_{1,-t_n-t} \theta_{1,t}\omega_1, \theta_{2, -t_n-t}\theta_{2, t} \omega_2 , x_n) )
=\Phi(t, \omega_1, \omega_2, y(\omega_1, \omega_2) ).
\ee
It follows from \eqref{pocomp2} and Remark \ref{defomegalimitrem1}
that, for all $t \ge 0$,
$$
\Phi (t, \omega_1, \omega_2, y(\omega_1, \omega_2))
\in \Omega(D, \theta_{1, t} \omega_1, \theta_{2, t} \omega_2 ).
$$
Let $N_1$ be large enough such that
$t_n \ge 1$ for $n \ge N_1$.
Consider the sequence
$ \Phi(t_n -1,  \theta_{1,-t_n} \omega_1, \theta_{2, -t_n} \omega_2 , x_n )$
$ =\Phi(t_n -1,  \theta_{1,-t_n +1} \theta_{1, -1}\omega_1, 
\theta_{2, -t_n+1} \theta_{2, -1}  \omega_2 , x_n )$.
Then
the $\cald$-pullback asymptotic compactness of $\Phi$ implies that there
exists $y_1 \in X$ and
 a subsequence, which is still denoted by the same one, such that
\be
\label{pocomp3}
 \lim_{n \to \infty}
 \Phi(t_n -1,  \theta_{1,-t_n} \omega_1, \theta_{2, -t_n} \omega_2 , x_n )
 =\lim_{n \to \infty}
 \Phi(t_n -1,  \theta_{1,-t_n +1} \theta_{1, -1}\omega_1, 
 \theta_{2, -t_n+1} \theta_{2, -1}  \omega_2 , x_n )
 = y_1.
 \ee
 By \eqref{pocomp3} and Remark \ref{defomegalimitrem1}
 we have $y_1 \in \Omega(D, \theta_{1, -1}\omega_1, \theta_{2, -1}\omega_2 )$.
 By the continuity of $\Phi$ again, we get that, for all $t \ge 0$,
 $$\lim_{n \to \infty}
 \Phi (t, \theta_{1, -1} \omega_1, \theta_{2, -1} \omega_2,
 \Phi(t_n -1,  \theta_{1,-t_n} \omega_1, \theta_{2, -t_n} \omega_2 , x_n ))
 =\Phi (t, \theta_{1, -1} \omega_1, \theta_{2, -1} \omega_2, y_1),
 $$
 which along with the cocycle property of $\Phi$ shows that,  for $t \ge 0$,
 \be
 \label{pocomp4}
 \lim_{n \to \infty}
 \Phi(t_n +t -1,  \theta_{1,-t_n} \omega_1, \theta_{2, -t_n} \omega_2 , x_n)
 =\Phi (t, \theta_{1, -1} \omega_1, \theta_{2, -1} \omega_2, y_1).
\ee
By \eqref{pocomp4} and Remark  \ref{defomegalimitrem1} we obtain that, for $t \ge 0$,
$$
\Phi (t, \theta_{1, -1} \omega_1, \theta_{2, -1} \omega_2, y_1)
\in \Omega(D, \theta_{1, t-1} \omega_1, \theta_{2, t-1} \omega_2 ).$$
On the other hand, it follows from \eqref{pocomp2} that, for all $t \ge 1$,
\be
\label{pocomp5}
\lim_{n \to \infty}
 \Phi(t_n +t -1,  \theta_{1,-t_n} \omega_1, \theta_{2, -t_n} \omega_2 , x_n)
 =\Phi (t-1,   \omega_1,   \omega_2, y(\omega_1, \omega_2) ).
 \ee
 By \eqref{pocomp4}-\eqref{pocomp5} we get, for all $t \ge 1$,
 $$\Phi (t, \theta_{1, -1} \omega_1, \theta_{2, -1} \omega_2, y_1)
 = \Phi (t-1,   \omega_1,   \omega_2, y(\omega_1, \omega_2) ).
 $$
 In particular,   for $t =1$ we have
 $y(\omega_1, \omega_2) = \Phi(1, \theta_{1,-1} \omega_1, \theta_{2, -1} \omega_2, y_1)$.
 Continuing this process, by a diagonal argument,
  we find that for every $m=1,2, \ldots$,
 there exists $y_m \in \Omega(D,\theta_{1, -m} \omega_1, \theta_{2, -m} \omega_2)$
 such that
 $\Phi (m, \theta_{1, -m} \omega_1, \theta_{2,-m} \omega_2 , y_m )
 =y(\omega_1, \omega_2)$ and
 for all $t \ge 0$,
 \be\label{pocomp6}
 \Phi (t, \theta_{1, -m} \omega_1, \theta_{2,-m} \omega_2 , y_m )
 \in \Omega(D, \theta_{1, t-m} \omega_1, \theta_{2, t-m}\omega_2 ).
\ee
 Furthermore, for all $t \ge 1$, we have
 $$
 \Phi (t-1, \theta_{1, -m+1} \omega_1, \theta_{2, -m+1} \omega_2, y_{m-1} )
 =\Phi (t, \theta_{1, -m} \omega_1, \theta_{2,-m} \omega_2 , y_m ).
 $$

 We now define a mapping $\psi_{\omega_1, \omega_2}: \R \to X$ such that
 for every $t \in \R$,  $\psi_{\omega_1, \omega_2} (t)$
 is  the common value of
 $\Phi(t+m, \theta_{1, -m} \omega_1, \theta_{2,-m} \omega_2, y_m )$
 for all $m \ge -t$.
 Then we have
 $\psi_{\omega_1, \omega_2}(0) =y(\omega_1, \omega_2)$. It also follows from
 \eqref{pocomp6} that
 $\psi_{\omega_1, \omega_2}(t) \in \Omega(D, \theta_{1, t} \omega_1, 
 \theta_{2, t}\omega_2 )$ for all $t \in \R$.
 If  $t \ge 0$,  $\tau \in \R$ and
   $m \ge -\tau$, by the definition of $\psi_{\omega_1, \omega_2}$
 we have
 $$
 \Phi(t, \theta_{1, \tau}\omega_1, \theta_{2, \tau}\omega_2,
 \psi_{\omega_1, \omega_2}(\tau) )
 =
 \Phi(t, \theta_{1, \tau}\omega_1, \theta_{2, \tau}\omega_2,
  \Phi(\tau +m, \theta_{1, -m} \omega_1, \theta_{2,-m} \omega_2, y_m ))
  $$
 \be\label{pocomp9}
  =
  \Phi(t+ \tau +m, \theta_{1, -m} \omega_1, \theta_{2,-m} \omega_2, y_m )
  =\psi_{\omega_1, \omega_2} (t+\tau).
 \ee
 Let $\psi: \R \times \Omega_1 \times \Omega_2
 \to X$   be the map given  by
 $\psi (t, \omega_1, \omega_2) =\psi_{\omega_1, \omega_2} (t)$
 for all $t \in \R$, $\omega_1 \in \Omega_1$ and
   $\omega_2 \in \Omega_2$. It follows from
   \eqref{pocomp9} that
   $\psi$ is a complete orbit
   of $\Phi$. In addition,
$\psi(0,\omega_1, \omega_2) =y(\omega_1, \omega_2)$
and
 $\psi (t, \omega_1, \omega_2 )
  \in \Omega(D, \theta_{1, t} \omega_1, \theta_{2, t}\omega_2 )$ 
  for all $t \in \R$. This shows that $\Omega(D)$ is quasi-invariant under $\Phi$.

  (iii). We now  prove the invariance  of $\Omega(D)$.
  By the quasi-invariance (ii) of $\Omega (D)$
  and Lemma \ref{posquainv} we get that,
  for every $t\ge 0$, $\omega_1 \in \Omega_1$ and $\omega_2 \in \Omega_2$,
\be\label{pocomp40}
\Omega(D, \theta_{1, t} \omega_1, \theta_{2, t} \omega_2 )
 \subseteq \Phi(t, \omega_1, \omega_2, \Omega(D, \omega_1, \omega_2 )).
 \ee
   We now prove the converse inclusion relation  of  \eqref{pocomp40}.
    Note that \eqref{pocomp40} implies that,  for all $r \ge 0$,
    $\Omega(D, \omega_1, \omega_2)
    \subseteq  \Phi (r, \theta_{1, -r} \omega_1, \theta_{2, -r} \omega_2,
    \Omega (D, \theta_{1, -r} \omega_1, \theta_{2, -r} \omega_2 ))$ and therefore
    for every  $t \ge 0$    and  $r \ge 0$,
    $$
     \Phi(t, \omega_1, \omega_2, \Omega(D, \omega_1, \omega_2) )
     \subseteq
     \Phi (t, \omega_1, \omega_2,
     \Phi (r, \theta_{1, -r} \omega_1, \theta_{2, -r} \omega_2,
    \Omega (D, \theta_{1, -r} \omega_1, \theta_{2, -r} \omega_2 )))
    $$
    $$
    = \Phi (t+r, \theta_{1, -r} \omega_1, \theta_{2, -r} \omega_2,
    \Omega (D, \theta_{1, -r} \omega_1, \theta_{2, -r} \omega_2  ))
    $$
   \be
    \label{pocomp41}
   = \Phi (s, \theta_{1, -s}  \theta_{1,t} \omega_1, \theta_{2, -s} \theta_{2,t} \omega_2,
    \Omega (D, \theta_{1, -s}  \theta_{1,t} \omega_1, \theta_{2, -s} \theta_{2,t} \omega_2 )),
   \ee
    where $s = t+r$.
    Note that
    $\Phi(t, \theta_{1, -t} \omega_1, \theta_{2, -t} \omega_2, 
    D( \theta_{1, -t} \omega_1, \theta_{2, -t} ))  \subseteq D(\omega_1, \omega_2)$
    for all $t \ge t_0$,
    and hence, by Definition  \ref{defomlit} we get
  \be\label{pocomp42}
    \Omega(D, \omega_1, \omega_2)
    \subseteq {\overline{\bigcup_{t \ge t_0} 
    \Phi (t,  \theta_{1, -t} \omega_1, \theta_{2, -t} \omega_2, 
    D( \theta_{1, -t} \omega_1, \theta_{2, -t} \omega_2 ))  }}
     \subseteq {\overline{D(\omega_1, \omega_2) }} 
     = D(\omega_1, \omega_2).
  \ee
  Here we have used the assumption that
  $ D(\omega_1, \omega_2)$ is closed in $X$.
      It follows from \eqref{pomegacomp2} and \eqref{pocomp42}   that
    \be
    \label{pocomp43}
    \lim_{s  \to \infty}
     d ( \Phi (s, \theta_{1, -s}  \theta_{1,t} \omega_1, \theta_{2, -s} \theta_{2,t} \omega_2,
    \Omega (D, \theta_{1, -s}  \theta_{1,t} \omega_1, 
    \theta_{2, -s} \theta_{2,t} \omega_2 ) ), \Omega(D, \theta_{1,t} \omega_1, 
     \theta_{2,t} \omega_2 ) ) =0.
    \ee
    By \eqref{pocomp41}  and \eqref{pocomp43} we get
    $$
     d( \Phi(t, \omega_1, \omega_2, \Omega(D, \omega_1, \omega_2) ),
     \Omega(D, \theta_{1,t} \omega_1,  \theta_{2,t} \omega_2 ))
     =0,
    $$
   and hence we have
    $ \Phi(t, \omega_1, \omega_2, \Omega(D, \omega_1, \omega_2) )
    \subseteq
     \Omega(D, \theta_{1,t} \omega_1,  \theta_{2,t} \omega_2  )$,
     from which and   \eqref{pocomp40}
      the invariance of $\Omega(D)$ follows.
      The invariance of $\Omega(D)$ and \eqref{pocomp42}  yield
      (iii).
\end{proof}

 As stated below, the structure of the $\cald$-pullback attractor
 of $\Phi$ is fully  determined by the
 $\cald$-complete orbits.

\begin{lem}
\label{attstr1}
 Let $\cald$ be a  collection of some  families of   nonempty subsets of
$X$ and $\Phi$  be a continuous   cocycle on $X$
over $(\Omega_1,  \{\thonet\}_{t \in \R})$
and
$(\Omega_2, \calftwo, P,  \{\thtwot\}_{t \in \R})$.
Suppose that
$\Phi$ has a  $\cald$-pullback
attractor $\cala = \{\cala (\omega_1, \omega_2): \omega_1 \in \Omega_1, 
\omega_2 \in \Omega_2 \}   \in \cald$. Then, for
 every $\omega_1 \in \Omega_1$ and
$\omega_2 \in \Omega_2$,
\begin{center}
$\cala(\omega_1, \omega_2)
=\{
\psi(0, \omega_1, \omega_2):  \psi$ \mbox{is a } 
$\cald$-complete  \mbox{  orbit of } $\Phi
\}.
$
\end{center}
\end{lem}

\begin{proof}
 Given $\omega_1 \in \Omega_1$  and $\omega_2 \in \Omega_2$,  let
 $B(\omega_1, \omega_2)  $ be the set consisting of all
 $ \psi(0, \omega_1, \omega_2)$ for
    $\cald$-complete     orbits $\psi$.
    We first prove $B(\omega_1, \omega_2)  \subseteq \cala(\omega_1, \omega_2) $.
    Let $y \in B(\omega_1, \omega_2)$. Then there exist
    $ D \in \cald$  and a complete orbit $\psi$
    such that   $y = \psi(0, \omega_1, \omega_2)$ and
 $\psi(t, \omega_1, \omega_2) \in
  D(\theta_{1,t} \omega_1, \theta_{2, t} \omega_2 )$
 for all  $t \in \R$. By the attraction property of $\cala$ we have
$$
 \lim_{t \to \infty}
 d (\Phi(t, \theta_{1,-t}\omega_1, \theta_{2,-t}\omega_2, 
 D(\theta_{1,-t}\omega_1, \theta_{2,-t}\omega_2) ) , 
 \cala (\omega_1, \omega_2 ))=0,
$$
which implies that
 \be
 \label{pattstr1_1}
 \lim_{t \to \infty}
 d (\Phi(t, \theta_{1,-t}\omega_1, \theta_{2,-t}\omega_2, 
 \psi(-t, \omega_1, \omega_2)) , \cala (\omega_1, \omega_2 ))=0.
 \ee
It follows from \eqref{comporbit1}
and \eqref{pattstr1_1} that
$d(\psi(0, \omega_1, \omega_2), \cala (\omega_1, \omega_2 ) )
=0$, and hence $y =\psi(0, \omega_1, \omega_2 ) \in
\cala (\omega_1, \omega_2 )$.
 This shows that
 $B(\omega_1, \omega_2)  \subseteq \cala(\omega_1, \omega_2) $.

   We next prove
$ \cala(\omega_1, \omega_2) \subseteq  B(\omega_1, \omega_2)$.
Since $\cala$ is invariant, by Lemma \ref{posquainv} we find that
$\cala$ is quasi-invariant.  Therefore,
given $\omega_1 \in \Omega_1$,
$\omega_2 \in  \Omega_2$ and
 $y(\omega_1, \omega_2) \in \cala( \omega_1, \omega_2)$
  there exists a complete
orbit $\psi$ such that $\psi(0, \omega_1, \omega_2)
= y(\omega_1, \omega_2)$
and $\psi (t, \omega_1, \omega_2) \in \cala( \theta_{1, t} \omega_1,
\theta_{2, t} \omega_2 )$ for all $t \in \R$.
Since $\cala \in \cald$, we see that
$\psi$ is a $\cald$-complete orbit. Therefore,
by definition,  $y(\omega_1, \omega_2)$
$=\psi(0, \omega_1, \omega_2)
\in B(\omega_1, \omega_2)$. This proves
$ \cala(\omega_1, \omega_2) \subseteq  B(\omega_1, \omega_2)$.
\end{proof}

\begin{lem}
\label{aclem1}
 Let $\cald$ be a    collection of some  families of   nonempty subsets of
$X$ and $\Phi$  be a continuous   cocycle on $X$
over $(\Omega_1,  \{\thonet\}_{t \in \R})$
and
$(\Omega_2, \calftwo, P,  \{\thtwot\}_{t \in \R})$.
Suppose that
$\Phi$ has a  $\cald$-pullback
attractor $\cala = \{\cala (\omega_1, \omega_2): 
\omega_1 \in \Omega_1, \omega_2 \in \Omega_2 \}   \in \cald$. Then
$\Phi$ is $\cald$-pullback asymptotically
compact in $X$.
\end{lem}

\begin{proof}
 Let $t_n \to \infty$,  $B \in \cald$
and $x_n \in B(\theta_{1, -t_n} \omega_1, \theta_{2, -t_n}\omega_2)$.
 We will show  that the sequence
$\{ \Phi(t_n, \theta_{1, -t_n} \omega_1, 
\theta_{2, -t_n}\omega_2, x_n \}_{n=1}^\infty$
has a convergent subsequence in $X$.
By the attraction property of $\cala$,
for every $\omega_1 \in \Omega_1$,
$\omega_2 \in \Omega_2$ and
  $m \in \N$, there
   exist $z_m \in \cala(\omega_1, \omega_2)$
and     $ t_{n_m} \in     \{t_n: n\in \N \}$ with $t_{n_m} \ge m$  and
$  x_{n_m}   \in   \{x_n: n \in \N \}$ such that
\be
\label{patt_a1}
d (\Phi(t_{n_m}, \theta_{1, -t_{n_m}}\omega_1,
\theta_{2, -t_{n_m}}\omega_2, x_{n_m}), z_m) \le {\frac 1m}.
\ee
The compactness of $\cala(\omega_1, \omega_2)$ implies that
$\{z_m\}_{m=1}^\infty$ has a convergent  subsequence
in $X$, which along with \eqref{patt_a1} shows that
$\Phi(t_{n_m}, \theta_{1, -t_{n_m}}\omega_1,
\theta_{2, -t_{n_m}}\omega_2, x_{n_m} )$
has a convergent subsequence, and hence
$\Phi$ is $\cald$-pullback asymptotically
compact in $X$.
\end{proof}

\begin{lem}
\label{attlem1}
 Let $\cald$ be a  neighborhood closed  collection
  of some  families of   nonempty subsets of
$X$,  and $\Phi$  
be a continuous   cocycle on $X$
over $(\Omega_1,  \{\thonet\}_{t \in \R})$
and
$(\Omega_2, \calftwo, P,  \{\thtwot\}_{t \in \R})$.
Suppose that
$\Phi$ has a  $\cald$-pullback
attractor $\cala = \{\cala (\omega_1, \omega_2): 
\omega_1 \in \Omega_1, \omega_2 \in \Omega_2 \}  
 \in \cald$. Then  $\Phi$ has a  closed
     $\cald$-pullback absorbing set
  $K $    that is measurable
with respect to the $P$-completion of $\calftwo$ in $\Omega_2$.
\end{lem}

\begin{proof} 
Since  $\cala  \in \cald$ and $\cald$ is  neighborhood closed,
 there exists  
$\varepsilon>0$   such that  
\be\label{attlem1_p1a1}
 \{ {B}(\omega_1, \omega_2) :
 {B}(\omega_1, \omega_2) \mbox{ is a  nonempty subset of }
 \caln_\varepsilon ( \cala (\omega_1, \omega_2) ),  \forall \
 \omega_1 \in \Omega_1,  \forall\  \omega_2
\in  \Omega_2\} \in \cald.
\ee
Choose an arbitrary  number $\varepsilon_1 \in (0, \varepsilon)$.
Given $\omega_1 \in \Omega_1$ and $\omega_2 \in \Omega_2$,
let $K(\omega_1, \omega_2)$ be the closed neighborhood of
$\cala(\omega_1, \omega_2)$ with radius $\varepsilon_1$, that is,
\be\label{attlem1_p2}
K(\omega_1, \omega_2) = \{x \in X: \  d(x, \cala(\omega_1, \omega_2))
\le  \varepsilon_1\}.
\ee
By \eqref{attlem1_p1a1}  we have
$K=\{K(\omega_1, \omega_2): \omega_1 \in \Omega_1, \omega_2 \in \Omega_2 \}
\in \cald$.
 By the attraction property of $\cala$, we find that
$K$ is a closed $\cald$-pullback  absorbing set of $\Phi$.
We now prove the measurability of $K$  with respect to
the $P$-completion of $\calftwo$,  that is denoted by
${\bar{\calftwo}}$.
Given $\omega_1 \in \Omega_1$
and $x \in X$, by the measurability of $\cala$ we find the mapping:
$\omega_2 \in \Omega_2 \to d(x, \cala(\omega_1, \omega_2))$,
is measurable  with respect to  ${\bar{\calftwo}}$. In addition,
the mapping: $ x \in X \to   d(x, \cala(\omega_1, \omega_2))$
is continuous. Therefore  the mapping
$ (\omega_2, x )   \to d(x, \cala(\omega_1, \omega_2))$
is measurable with respect to $\bar{\calftwo} \times \calb (X)$.
Thus we have the set $\{ (\omega_2, x) \in \Omega_2 \times X: \
 x \in K (\omega_1, \omega_2 ) \} $
  $ = \{ (\omega_2, x) \in \Omega_2 \times X: \
d(x, \cala(\omega_1, \omega_2)) \le  \varepsilon_1\} \in
 \bar{\calftwo} \times \calb (X)$.
 This shows that the graph of the
 set-valued
 mapping: $ \omega_2 \to  K(\omega_1, \omega_2)$,
 is in
$\bar{\calftwo} \times \calb (X)$. Then
the measurability of $K$
with respect to
 $\bar{\calftwo}$ follows from
Theorem 8.1.4 in \cite{aubin1}.
\end{proof}

\begin{lem}
\label{attlem2}
 Let $\cald$ be an inclusion-closed  collection of some  families of   nonempty subsets of
$X$ and $\Phi$  be a continuous   cocycle on $X$
over $(\Omega_1,  \{\thonet\}_{t \in \R})$
and
$(\Omega_2, \calftwo, P,  \{\thtwot\}_{t \in \R})$.
 Suppose  that $K = \{K(\omega_1, \omega_2): \omega_1 \in \Omega_1,
  \omega_2 \in \Omega_2 ) \}    \in \cald $ is a   closed
    measurable (w.r.t. the $P$-completion of $\calftwo$)
     $\cald$-pullback absorbing set  for  $\Phi$  in
$\cald$ and $\Phi$ is $\cald$-pullback asymptotically
compact in $X$. Then $\Phi$ has a unique $\cald$-pullback
attractor $\cala = \{\cala (\omega_1, \omega_2): \omega_1 
\in \Omega_1, \omega_2 \in \Omega_2 \}   \in \cald$ which is
given by,
for all $\omega_1 \in \Omega_1$   and  $\omega_2 \in \Omega_2$,
\be
\label{att1}
\cala (\omega_1, \omega_2) =
\Omega(K, \omega_1, \omega_2)
=  \bigcap_{\tau \ge 0}
\  \overline{ \bigcup_{t\ge \tau} \Phi(t, \theta_{1,-t} \omega_1, 
\theta_{2, -t} \omega_2, K(\theta_{1,-t} \omega_1, \theta_{2,-t}\omega_2  ))}.
\ee
\end{lem}

\begin{proof}
Note that the
      uniqueness of $\cald$-pullback attractor
       follows directly from Definition  \ref{defatt}
       as stated in Remark \ref{defattrem1}.
Since $K$ is a closed $\cald$-pullback absorbing set of $\Phi$,
it follows from Lemma \ref{omegacomp}  that,
for each $\omega_1 \in \Omega_1$ and $\omega_2 \in \Omega_2$,
$\cala(\omega_1, \omega_2)$
$=\Omega (K, \omega_1, \omega_2)$
is a  nonempty compact subset of $X$
and $\cala(\omega_1, \omega_2)\subset K (\omega_1, \omega_2)$.
Since $\cald$ is   inclusion-closed and $K \in \cald$, we see
that $\cala =
\{\cala(\omega_1, \omega_2): \omega_1 \in \Omega_1,   \omega_2 \in \Omega_2 \}
\in \cald$.
Further,
   Lemma \ref{omegacomp} implies  that     $\cala  $
is invariant. Therefore,
the proof will be completed if we can show
 the measurability of $\cala$  with respect to the $P$-completion of
  $\calftwo$ and the attraction of $\cala$ on all members of $\cald$.

 We first prove that $\cala$ attracts all elements of $\cald$.
  By Lemma \ref{omegacomp} we find that  $\cala$    attracts $K$,
  that is,
 given $\varepsilon>0$,
  $\omega_1 \in \Omega_1$   and $\omega_2 \in \Omega_2$,
    there exists
   $t_1 =t_1(\varepsilon, \omega_1, \omega_2)>0$
   such that
   \be
   \label{patt5}
   d (\Phi (t_1, \theta_{1, -t_1}\omega_1, \theta_{2, -t_1} \omega_2, 
   K(\theta_{1, -t_1}\omega_1, \theta_{2, -t_1} \omega_2) ), 
   \ \cala (\omega_1,  \omega_2) ) < \varepsilon.
 \ee
 Then by the  absorption property of $K$ we  get that,
 for every $D \in \cald$, there exists $T=T(D, \omega_1, \omega_2, \varepsilon )$ 
 $ >0$ such that  for all $t > T$,
 $$
 \Phi (t,    \theta_{1, -t}(\theta_{1, -t_1} \omega_1),
 \theta_{2, -t} (\theta_{2,-t_1} \omega_2),
      D( \theta_{1, -t} (\theta_{1,-t_1} \omega_1),
 \theta_{2, -t}  (\theta_{2, -t_1} \omega_2 )))
 \subseteq
 K(\theta_{1,-t_1} \omega_1, \theta_{2,-t_1} \omega_2),
 $$
 which implies that, for all $t>T$,
 $$
 \Phi (t +t_1, \theta_{1, -t-t_1}\omega_1,
 \theta_{2, -t-t_1} \omega_2, D( \theta_{1, -t-t_1}\omega_1,
 \theta_{2, -t-t_1} \omega_2 ) )
 $$
 $$
 =
 \Phi (t_1, \theta_{1, -t_1}\omega_1,
 \theta_{2, -t_1} \omega_2,
 \Phi (t,    \theta_{1, -t-t_1}\omega_1,
 \theta_{2, -t-t_1} \omega_2,
      D( \theta_{1, -t-t_1}\omega_1,
 \theta_{2, -t-t_1} \omega_2 ) )
 $$
 \be\label{patt6}
 \subseteq
 \Phi (t_1, \theta_{1, -t_1}\omega_1,
 \theta_{2, -t_1} \omega_2,
 K(\theta_{1,-t_1} \omega_1, \theta_{2,-t_1} \omega_2)).
 \ee
 It follows  from \eqref{patt5}-\eqref{patt6} that, for all
 $t > T+t_1$,
  $$
  d (\Phi (t, \theta_{1, -t}\omega_1, \theta_{2, -t} \omega_2, 
  D(\theta_{1, -t}\omega_1, \theta_{2, -t} \omega_2) ), 
  \ \cala (\omega_1,  \omega_2) ) < \varepsilon.
$$
This shows that
 $\cala$  attracts all members of $\cald$.

 We next prove the measurability  of $\cala$ with respect to the $P$-completion of
  $\calftwo$.
It follows from
\eqref{att1} that
\be
\label{patt40}
\cala (\omega_1, \omega_2) =  \bigcap_{n=1}^\infty
\  \overline{ \bigcup_{t\ge n} \Phi(t, \theta_{1,-t} \omega_1, 
\theta_{2, -t} \omega_2, K(\theta_{1,-t} \omega_1, \theta_{2,-t}\omega_2  ))},
\ee
where $n\in \N$.  By  \eqref{patt40} we  find that, for $n, m \in \N$,
 \be
\label{patt41}
  \bigcap_{n=1}^\infty
\  \overline{ \bigcup_{m =n}^\infty \Phi(m, \theta_{1,-m} \omega_1, 
\theta_{2, -m} \omega_2, K(\theta_{1,-m} \omega_1, \theta_{2,-m}\omega_2  ))} 
 \subseteq  \cala (\omega_1, \omega_2) .
\ee
On the  other hand, since
$\cala(\omega_1, \omega_2)
\subseteq  K(\omega_1, \omega_2) $
for all $\omega_1 \in \Omega_1$ and
$\omega_2 \in \Omega_2$, we have for all $ m \in \N$,
$$
 \Phi(m, \theta_{1,-m} \omega_1, \theta_{2, -m} \omega_2, 
 \cala(\theta_{1,-m} \omega_1, \theta_{2,-m}\omega_2  ))
\subseteq
\Phi(m, \theta_{1,-m} \omega_1, \theta_{2, -m} \omega_2, 
K(\theta_{1,-m} \omega_1, \theta_{2,-m}\omega_2  ))
$$
which along with
 the invariance of $\cala$
yields  that for all $m \in \N$,
\be
\label{patt42}
\cala({\omega_1, \omega_2})
 \subseteq
\Phi(m, \theta_{1,-m} \omega_1, \theta_{2, -m} \omega_2, 
K(\theta_{1,-m} \omega_1, \theta_{2,-m}\omega_2  )).
\ee
Therefore we get that
$$
    \cala (\omega_1, \omega_2) \subseteq
  \bigcap_{n=1}^\infty
\  \overline{ \bigcup_{m =n}^\infty \Phi(m, \theta_{1,-m} \omega_1, 
\theta_{2, -m} \omega_2, K(\theta_{1,-m} \omega_1, \theta_{2,-m}\omega_2  ))}  ,
$$
from which  and  \eqref{patt41}  we obtain
\be
\label{patt43}
    \cala (\omega_1, \omega_2) =
  \bigcap_{n=1}^\infty
\  \overline{ \bigcup_{m =n}^\infty \Phi(m, \theta_{1,-m} \omega_1, 
\theta_{2, -m} \omega_2, K(\theta_{1,-m} \omega_1, \theta_{2,-m}\omega_2  ))}  .
\ee
By conditions (i) and (iv) in  Definition \ref{ds1}, and the
measurability of $\theta_2$, we find that,
 for every fixed  $m \in \N$  and  $\omega_1 \in \Omega_1$,
the mapping
 $\Phi(m, \theta_{1, -m} \omega_1,
\theta_{2, -m} \omega_2, x)$
    is $( \calftwo, \calb (X))$-measurable in $\omega_2$ and
  is  continuous in $x$. In addition, the set-valued  mapping $\omega_2 $
  $\to K(\theta_{1,-m} \omega_1, \theta_{2,-m}\omega_2 )$
  has   closed images and  is measurable
  with respect to the $P$-completion of $\calftwo$. Then it follows from \cite{aubin1}
  (Theorem 8.2.8)  that
  for every fixed  $m \in \N$  and  $\omega_1 \in \Omega_1$,
 the mapping
 $$\omega_2 \to
    \overline{
 \Phi(m, \theta_{1,-m} \omega_1, \theta_{2, -m} \omega_2, K(\theta_{1,-m}
  \omega_1, \theta_{2,-m}\omega_2  )) } $$
 is measurable with respect to the $P$-completion of $\calftwo$. Therefore, by 
 Theorem 8.2.4 in \cite{aubin1}, we find
  that, for every fixed $\omega_1 \in \Omega_1$,  the set-valued  mapping
 $$
 \omega_2 \to \bigcap_{n=1}^\infty
\  \overline{ \bigcup_{m =n}^\infty
 \overline{\Phi(m, \theta_{1,-m} \omega_1, \theta_{2, -m} \omega_2, 
 K(\theta_{1,-m} \omega_1, \theta_{2,-m}\omega_2  )) }}
$$
 is measurable with respect to the $P$-completion of $\calftwo$, which implies that
 the mapping
\be \label{patt44}
 \omega_2 \to \bigcap_{n=1}^\infty
\  \overline{ \bigcup_{m =n}^\infty
  \Phi(m, \theta_{1,-m} \omega_1, \theta_{2, -m} \omega_2, 
  K(\theta_{1,-m} \omega_1, \theta_{2,-m}\omega_2  )) }
\ee
 is measurable with respect to the $P$-completion of $\calftwo$.
 Then,   for every fixed $\omega_1 \in \Omega_1$,
 the measurability
  of  $\cala (\omega_1, \cdot)$
 with respect to the $P$-completion of $\calftwo$ follows from
  \eqref{patt43} and \eqref{patt44}.
\end{proof}

 Note that  the attractor constructed in Lemma 
   \ref{attlem2} is 
    measurable with respect to the $P$-completion of $\calftwo$
    rather than  $\calftwo$ itself.
    This is because 
   the measurability  of the attractor is obtained   
by  using Theorem 8.2.4 in \cite{aubin1}, and this 
   theorem requires  that the probability space is complete.

\begin{lem}
\label{attlem3}
Let all assumptions of Lemma  \ref{attlem2} hold. Then the unique
  $\cald$-pullback
attractor $\cala$ of $\Phi$ in $\cald$ can be characterized by,
 for all $\omega_1 \in \Omega_1$   and  $\omega_2 \in \Omega_2$,
\be\label{attlem3_a1}
\cala (\omega_1, \omega_2) =
  \bigcup_{B\in \cald} \Omega (B, \omega_1, \omega_2),
  \ee
  where $\Omega (B, \omega_1, \omega_2)$, as given by
  \eqref{omegalimit}, is the $\Omega$-limit set of $B$
  corresponding to $\omega_1$   and $\omega_2$.
  \end{lem}

  \begin{proof}
  It is evident from \eqref{att1} that
  $\cala (\omega_1, \omega_2) \subseteq
  \bigcup\limits_{B\in \cald} \Omega (B, \omega_1, \omega_2)$
  since the $\cald$-pullback absorbing set $K$ is in $\cald$.
  So we  only need to show the reverse inclusion relation.
  Let $y \in \bigcup\limits_{B\in \cald} \Omega (B, \omega_1, \omega_2)$.
  Then there exists $B \in \cald$ such that $y \in \Omega (B, \omega_1, \omega_2)$.
  It follows from Remark \ref{defomegalimitrem1} that
  there exist $t_n \to \infty$ and $x_n \in B(\theta_{1, -t_n} 
  \omega_1, \theta_{2, -t_n} \omega_2 )$ such that
  \be
  \label{pattlem3_1}
  \lim_{n \to \infty} \Phi (t_n, \theta_{1, -t_n}\omega_1,
  \theta_{2, -t_n} \omega_2, x_n ) =y.
  \ee
  By the attraction property of $\cala$ we get
  \be
  \label{pattlem3_2}
  \lim_{n \to \infty} d \left (
  \Phi (t_n, \theta_{1, -t_n}\omega_1,
  \theta_{2, -t_n} \omega_2,
   B(\theta_{1, -t_n} \omega_1, \theta_{2, -t_n} \omega_2 )),
   \cala (\omega_1, \omega_2 ) \right ) =0.
  \ee
  From \eqref{pattlem3_1}-\eqref{pattlem3_2} we obtain  that
  $  d \left (y,
   \cala (\omega_1, \omega_2 ) \right ) =0$, and hence
   $y \in \cala(\omega_1, \omega_2)$ due to the compactness
   of $\cala(\omega_1, \omega_2)$.
  This shows that  $  \bigcup\limits_{B\in \cald}
  \Omega (B, \omega_1, \omega_2)\subseteq \cala (\omega_1, \omega_2)$
  and thus \eqref{attlem3_a1} follows.
  \end{proof}

From   Lemma  \ref{attstr1} up to Lemma \ref{attlem3},  we   immediately
obtain the following
  main result of this section.

\begin{thm}
\label{att}
 Let $\cald$ be a  neighborhood closed  collection of some 
  families of   nonempty subsets of
$X$   and $\Phi$  be a continuous   cocycle on $X$
over $(\Omega_1,  \{\thonet\}_{t \in \R})$
and
$(\Omega_2, \calftwo, P,  \{\thtwot\}_{t \in \R})$.
Then
$\Phi$ has a  $\cald$-pullback
attractor $\cala$  in $\cald$
if and only if
$\Phi$ is $\cald$-pullback asymptotically
compact in $X$ and $\Phi$ has a  closed
   measurable (w.r.t. the $P$-completion of $\calftwo$)
     $\cald$-pullback absorbing set
  $K$ in $\cald$.
  The $\cald$-pullback
attractor $\cala$   is unique   and is given  by,
for each $\omega_1  \in \Omega_1$   and
$\omega_2 \in \Omega_2$,
\be\label{attform1}
\cala (\omega_1, \omega_2)
=\Omega(K, \omega_1, \omega_2)
=\bigcup_{B \in \cald} \Omega(B, \omega_1, \omega_2)
\ee
\be\label{attform2}
 =\{\psi(0, \omega_1, \omega_2): \psi \mbox{ is a  }  \cald {\rm -}
 \mbox{complete orbit of } \Phi\} .
 \ee
  \end{thm}

We now prove the periodicity of $\cald$-pullback of attractors
for a periodic cocycle under certain conditions.

\begin{thm}
\label{periodatt} 
Suppose  $\Phi$   is  a continuous  periodic   cocycle
with period $T>0$  on $X$
over $(\Omega_1,  \{\thonet\}_{t \in \R})$
and
$(\Omega_2, \calftwo, P,  \{\thtwot\}_{t \in \R})$.
 Let   $\cald$   be  a  neighborhood closed 
   and $T$-translation invariant collection of some  families of   nonempty subsets of
$X$.
 If
$\Phi$ is $\cald$-pullback asymptotically
compact in $X$ and $\Phi$ has a  closed
   measurable (w.r.t. the $P$-completion of $\calftwo$)
     $\cald$-pullback absorbing set
  $K$ in $\cald$, then $\Phi$
  has a unique periodic
   $\cald$-pullback
attractor $\cala \in \cald$    with period $T$,  i.e., 
$\cala (\theta_{1, T} \omega_1,  \omega_2)
=\cala(\omega_1, \omega_2)$.
\end{thm}

\begin{proof}
 It follows from Theorem \ref{att} that $\Phi$ has a unique $\cald$-pullback
 attractor $\cala \in \cald $ which is given  by
 \eqref{attform1}   and \eqref{attform2}.
 We may use either
  \eqref{attform1}   or \eqref{attform2} to
  prove Theorem \ref{periodatt}.
  We here   use the latter
  to first prove
  $\cala(\theta_{1, T} \omega_1 , \omega_2)
  \subseteq \cala (\omega_1,  \omega_2 ) $
  for all $\omega_1 \in \Omega_1$   and $\omega_2 \in \Omega_2$.

  Fix $\omega_1 \in \Omega_1$ and $\omega_2 \in \Omega_2$ and 
  take $y \in   \cala(\theta_{1, T} \omega_1 , \omega_2)$. It follows from
  \eqref{attform2}    that there exists a $\cald$-complete orbit
  $\psi$ of $\Phi$  such that
  $y = \psi (0, \theta_{1, T} \omega_1, \omega_2)$.
  Since $\psi $  is a complete orbit of $\Phi$, for any $t \ge 0$, 
  $\tau \in \R$, $\omegat_1 \in \Omega_1$
  and $\omegat_2 \in \Omega_2$,  we have
  \be
  \label{ppatt1}
  \Phi (t, \theta_{1, \tau +T} \omegat_1, \theta_{2, \tau } \omegat_2,
  \psi(\tau,   \theta_{1,    T} \omegat_1,  \omegat_2 ))
  = \psi (t+\tau, \theta_{1,    T} \omegat_1,   \omegat_2 ).
  \ee
  Since $\psi $  is a $\cald$-complete orbit, there exists
  $D \in \cald$   such that
  \be
  \label{ppatt2}
  \psi (t, \theta_{1,    T} \omegat_1,  \omegat_2)
  \in D(\theta_{1,   t+ T} \omegat_1, \thtwot \omegat_2 )
  = D_T(\theta_{1,   t} \omegat_1, \theta_{2,   t} \omegat_2 ),
  \ee
  where $D_T$ is the $T$-translation of $D$ given by \eqref{DTfamily1}.
  We now define a map $u: \R \times \Omega_1 \times \Omega_2 \to X$ such that
  \be
  \label{ppatt3}
  u(t, \omegat_1, \omegat_2)
  = \psi (t, \theta_{1,   T} \omegat_1,  \omegat_2 ),
  \quad \forall \ t \in  \R, \  \forall \ \omegat_1 \in \Omega_1, \
  \forall \ \omegat_2 \in \Omega_2.
  \ee
  We next prove $u$ is a complete orbit of $\Phi$.
  It follows from \eqref{ppatt3} that
   for all $t \ge 0$,  $\tau \in \R$, $\omegat_1 \in \Omega_1$ and
  $\omegat_2 \in \Omega_2$,  
  $$
  \Phi (t, \theta_{1, \tau} \omegat_1, \theta_{2, \tau} \omegat_2, u(\tau, \omegat_1, \omegat_2))
  = \Phi (t, \theta_{1, \tau} \omegat_1, \theta_{2, \tau} \omegat_2,  
  \psi (\tau, \theta_{1,   T} \omegat_1,  \omegat_2 ))
  $$
  $$
 = \Phi (t, \theta_{1, \tau +T} \omegat_1,  \theta_{2, \tau} \omegat_2,
  \psi (\tau, \theta_{1,   T} \omegat_1,  \omegat_2 ))
  $$
  \be \label{ppatt4}
  =
  \psi (t+ \tau, \theta_{1,   T} \omegat_1,  \omegat_2 )
  =u(t + \tau, \omegat_1, \omegat_2),
 \ee
  where we have used 
  the $T$-periodicity of $\Phi$, \eqref{ppatt1}   and \eqref{ppatt3}.
  By \eqref{ppatt4} we see that $u$ is a complete orbit of $\Phi$. 
  We now prove that  $u$ is  actually a $\cald$-complete orbit.
  It   follows from \eqref{ppatt2}   and \eqref{ppatt3} that
  for all $t \ge 0$,   $\omegat_1 \in \Omega_1$ and
  $\omegat_2 \in \Omega_2$,  
  \be\label{ppatt9}
  u(t, \omegat_1, \omegat_2)
  = \psi (t, \theta_{1,   T} \omegat_1,  \omegat_2 )
  \in
  D_T(\theta_{1,   t} \omegat_1, \theta_{2,   t} \omegat_2 ).
 \ee
 Since $D \in \cald$ and $\cald$ is $T$-translation invariant,  we know 
 $D_T \in \cald$, which along with
 \eqref{ppatt9} shows   that $u$ is a $\cald$-complete orbit.
 Then it follows   from \eqref{attform2}    that
 $u(0, \omega_1,  \omega_2)
 \in \cala(\omega_1,  \omega_2)$.
 By \eqref{ppatt3} we have
 $u(0, \omega_1,  \omega_2)
 = \psi(0,  \theta_{1,T} \omega_1, \omega_2) =y$.
 Therefore we  get
 $y \in \cala(\omega_1,  \omega_2)$.
 Since $y$ is an arbitrary element of 
 $\cala(\theta_{1, T} \omega_1 , \omega_2)$  we find that
 \be
 \label{ppatt10}
 \cala(\theta_{1, T} \omega_1 , \omega_2)
  \subseteq \cala (\omega_1,   \omega_2 ), \quad 
  \mbox{ for all }  \omega_1 \in \Omega_1 \mbox{ and  }  \omega_2 \in \Omega_2.
  \ee

  We now prove  the converse of \eqref{ppatt10}.
     Let $y \in   \cala(  \omega_1 ,\omega_2)$. 
     By
  \eqref{attform2}    we find  that there exists a $\cald$-complete orbit
  $\psi$ of $\Phi$  such that
  $y = \psi (0,   \omega_1, \omega_2)$.
  Since $\psi $  is a complete orbit, for any $t \ge 0$,
  $\tau \in \R$, $\omegat_1 \in \Omega_1$
  and $\omegat_2 \in \Omega_2$,  we have
  \be
  \label{ppatt1a}
  \Phi (t, \theta_{1, \tau -T} \omegat_1, \theta_{2, \tau } \omegat_2,
  \psi(\tau,   \theta_{1,    -T} \omegat_1,  \omegat_2 ))
  = \psi (t+\tau, \theta_{1,    -T} \omegat_1,  \omegat_2 ).
  \ee
  Since $\psi $  is a $\cald$-complete orbit, there exists
  $D \in \cald$   such that
  \be
  \label{ppatt2a}
  \psi (t, \theta_{1,    -T} \omegat_1,  \omegat_2)
  \in D(\theta_{1,   t-T} \omegat_1, \thtwot \omegat_2 )
  = D_{-T}(\theta_{1,   t} \omegat_1, \theta_{2,   t} \omegat_2 ).
  \ee 
  Define a map $v: \R \times \Omega_1 \times \Omega_2 \to X$ such that
  \be
  \label{ppatt3a}
  v(t, \omegat_1, \omegat_2)
  = \psi (t, \theta_{1,   -T} \omegat_1,  \omegat_2 ),
  \quad \forall \ t \in  \R, \  \forall \ \omegat_1 \in \Omega_1, \
  \forall \ \omegat_2 \in \Omega_2.
  \ee
  By using the $T$-periodicity of $\Phi$, \eqref{ppatt1a}   and \eqref{ppatt3a},
  following the proof for $u$, 
  we can show that   $v$ is a complete orbit of $\Phi$.
    In addition,    by  \eqref{ppatt2a}   and \eqref{ppatt3a},
    we find  that
  for all $t \ge 0$,   $\omegat_1 \in \Omega_1$ and
  $\omegat_2 \in \Omega_2$,
  \be\label{ppatt9a}
  v(t, \omegat_1, \omegat_2)
  = \psi (t, \theta_{1,   -T} \omegat_1,  \omegat_2 )
  \in
  D_{-T}(\theta_{1,   t} \omegat_1, \theta_{2,   t} \omegat_2 ).
 \ee
 Since $D \in \cald$ and $\cald$ is $T$-translation invariant, 
 it follows from Lemma \ref{Tlation2}   that   
 $D_{-T} \in \cald$, which along with
 \eqref{ppatt9a} shows   that $v$ is a $\cald$-complete orbit.
 Then it follows   from \eqref{attform2}    that
 $v(0, \theta_{1, T}\omega_1,  \omega_2)
 \in \cala(\theta_{1, T}\omega_1,  \omega_2)$.
 But by \eqref{ppatt3a} we have
 $v(0, \theta_{1, T}\omega_1,  \omega_2)
 = \psi(0,  \omega_1, \omega_2) =y$.
 Thus we obtain 
 $y \in \cala(\theta_{1, T}\omega_1,  \omega_2)$,
 from which we get  that
 $\cala(\omega_1,  \omega_2 )
 \subseteq \cala(\theta_{1, T} \omega_1 , \omega_2)$.
 This along   with \eqref{ppatt10} concludes   the proof.
\end{proof}

 We now provide a sufficient
 and  necessary  criterion for the 
 periodicity of random attractors
 based on the existence of 
 periodic pullback absorbing sets.
 
\begin{thm}
\label{periodatt2} 
 Let   $\cald$   be  a  neighborhood closed 
     collection of some  families of   nonempty subsets of
$X$.
Suppose  $\Phi$   is  a continuous  periodic   cocycle
with period $T>0$  on $X$
over $(\Omega_1,  \{\thonet\}_{t \in \R})$
and
$(\Omega_2, \calftwo, P,  \{\thtwot\}_{t \in \R})$.
 Suppose further that  $\Phi$  has a $\cald$-pullback attractor
 $\cala \in \cald$.  Then $\cala$ is periodic with period  $T$
 if and only if 
$\Phi$  has a  closed
   measurable (w.r.t. the $P$-completion of $\calftwo$)
     $\cald$-pullback absorbing set
  $K \in \cald$ with $K$ being periodic with period $T$. 
\end{thm}

\begin{proof}
 Suppose $\cala$ is $T$-periodic, i.e.,   for every
 $\omega_1 \in \Omega_1$ and $\omega_2 \in \Omega_2$,
 \be\label{peratt2_p1}
 \cala(\theta_{1, T} \omega_1, \omega_2)
 = \cala(\omega_1, \omega_2).
 \ee
 Since $\cala \in \cald$ and $\cald$ is neighborhood closed,
  by    the proof of Lemma \ref{attlem1} we know   that
 $\Phi$ has a 
   closed
   measurable
     $\cald$-pullback absorbing set
  $K \in \cald$ which is given by
  \eqref{attlem1_p2}.
  By \eqref{peratt2_p1} and
     \eqref{attlem1_p2} we find   that,
  for every
 $\omega_1 \in \Omega_1$ and $\omega_2 \in \Omega_2$,
 $$
 K(\theta_{1, T} \omega_1, \omega_2)
 =  \{x \in X:  d(x, \cala(\theta_{1, T} \omega_1, \omega_2) )
 \le \varepsilon_1 \}
 =  \{x \in X:  d(x, \cala( \omega_1, \omega_2) )
 \le \varepsilon_1 \}
 =K( \omega_1, \omega_2).
 $$
 and hence $K$ is periodic with period $T$.

 On the other hand,  
 if $\Phi$  has a  closed
   measurable  
     $\cald$-pullback absorbing set
  $K \in \cald$ such that   $K$ is $T$-periodic,
  then by  \eqref{omegalimit},  Theorem \ref{att}
  and the $T$-periodicity of $\Phi$  we have,
  for every
 $\omega_1 \in \Omega_1$ and $\omega_2 \in \Omega_2$,
 $$ \cala (\theta_{1, T} \omega_1, \omega_2)
 =\Omega(K, \theta_{1, T}\omega_1, \omega_2)
= \bigcap_{\tau \ge 0}
\  \overline{ \bigcup_{t\ge \tau} 
\Phi(t, \theta_{1,-t}  \theta_{1, T}\omega_1, \theta_{2, -t} 
\omega_2, K(\theta_{1,-t} \theta_{1, T}\omega_1, \theta_{2,-t}\omega_2  ))}
$$
$$
= \bigcap_{\tau \ge 0}
\  \overline{ \bigcup_{t\ge \tau} 
\Phi(t, \theta_{1,-t} \omega_1, \theta_{2, -t} 
\omega_2, K(\theta_{1,-t} \omega_1, \theta_{2,-t}\omega_2  ))}
 =\Omega (K, \omega_1, \omega_2)
 =\cala(\omega_1, \omega_2).
$$
This shows  that $\cala$ is $T$-periodic and thus 
completes   the proof.
\end{proof}

\section{Pullback Attractors for Differential Equations}
\setcounter{equation}{0}
 
In this section,    we discuss how to
choose the parametric spaces $\Omega_1$
and $\Omega_2$ to 
  study   pullback   attractors  of
 differential   equations by  using
   the abstract  theory   presented in the preceding section.
Actually,   the concept of   continuous cocycle
given by Definition \ref{ds1} is motivated
by the solution operators of
differential equations
with  both external  non-autonomous deterministic    and
 stochastic  terms.
 Two special cases are covered by Definition \ref{ds1}:
 one is the solution operators generated by equations with only
   deterministic forcing  terms; and the other is the process generated by
   equations with only stochastic forcing terms.
   In this sense,
Definition \ref{ds1}   is an extension of the  classical concepts
of autonomous, deterministic non-autonomous
and random  dynamical systems.
  Indeed,
These classical concepts can be cast into the framework
of  Definition
\ref{ds1} by appropriately choosing the spaces $\Omega_1$ and $\Omega_2$.
To deal with  autonomous differential equations,
  we  may  simply take both  $\Omega_1$ and 
  $\Omega_2$ as singletons (say, $\Omega_1 =\{\omega^*_1\}$
  and $\Omega_2 =\{\omega^*_2\}$),
  and  define $\thonet \omega_1^* = \omega_1^*$
  and  $\thtwot \omega_2^* = \omega_2^*$
  for all $t \in \R$.  Then
a continuous cocycle $\Phi$  in $X$  in terms of
 Definition \ref{ds1}  is  exactly  a
 classical continuous autonomous dynamical system as defined in
 the literature (see, e.g.,
 \cite{bab1, hal1, sel1, tem1}).   In this case, if we define
 \be\label{autorD}
 \cald
 =\{ D = \{  D(\omega_1^*, \omega_2^*): 
     D(\omega_1^*, \omega_2^*)  
     \mbox{ is a  bounded nonempty  subset of }  X  \} \},
 \ee
  then
 the concept of a  $\cald$-pullback attractor given by Definition \ref{defatt}
 coincides with that of a global attractor of $\Phi$.
 Note  that   when  $\Omega_1 =\{\omega^*_1\}$
  and $\Omega_2 =\{\omega^*_2\}$ 
  are singletons,  every member $D$ of $\cald$
  as defined  by \eqref{autorD} is  also a singleton, i.e., 
  $D$
  contains  a single bounded  nonempty  set.
  Therefore, in this case,  
       $D$  does not depend on time   since 
  $\thonet$ and $\thtwot$ are both  
  constant for all $t \in \R$.
  This implies  that   $\cald$ given by \eqref{autorD}
 is inclusion-closed. Actually, it is  also neighborhood closed.
  In the case of autonomous systems,  Theorem
 \ref{att} is well known.
 Next,  we   discuss the non-autonomous systems  and random systems
in detail.

\subsection{Pullback Attractors  for  Equations   with   both
Non-autonomous   Deterministic   and
 Random Forcing  Terms}
 \label{sec31}
 
 In this subsection,  we discuss pullback attractors for differential 
 equations with non-autonomous deterministic  
  as well as stochastic  terms.
 In this  case, 
 there   are  at  least  two options   for choosing the 
 space $\Omega_1$   to formulate  the solution 
 operators of an equation  into  the setting of cocycles.  
 To demonstrate the idea, the pullback attractors 
 of Reaction-Diffusion  equations
 defined on $\R^n$ will be investigated in the next  section. 
 For non-autonomous equations, the space $\Omega_1$ 
  is used to deal with the dependence of solutions on initial times. 
   So  the first  choice of $\Omega_1$ is $\R$,  
  where $\R$   represents  the set of all initial times.
  For the   second   choice  of  $\Omega_1$,
   we need to use the translations of
  non-autonomous parameters  involved in   differential equations
  which may be  non-autonomous coefficients or external terms  or both.
  Let $f: \R \to X$  be a function that is   contained in a  
   differential   equation. Given $h \in \R$,  define  $f^h: \R \to  X$ by
  $f^h (\cdot) = f(\cdot + h)$.   The function $f^h$ 
  is actually a translation
  of $f$   by  $h$.  Let $\Omega_f$ be the collection of all translations of $f$, that  is, 
  $ \Omega_f =\{ f^h:    h \in \R \}$. 
  Then the  second choice of $\Omega_1$ is  $\Omega_f$.
  In the sequel, we first discuss the case $\Omega_1 =\R$
  and then $\Omega_1 =\Omega_f$. 
  As  we will see   later,  these two approaches are actually 
  equivalent  provided $f$ is  not a  periodic  function.

  Suppose  now    $\Omega_1 = \R$. 
  Define  a family $\{\thonet\}_{t \in \R}$    of shift operators   by
  \be
  \label{shift1}
  \thonet (h) =  h + t,
  \quad \mbox{ for all } \ t,\   h  \in \R.
  \ee
      Let  $\Phi$:   $ \R^+ \times \R \times \Omega_2 \times X$ 
 $  \to  X$  be   a continuous  cocycle on $X$
over $(\R,  \{\thonet\}_{t \in \R})$
and $(\Omega_2, \calftwo, P, \{\thtwot\}_{t \in \R})$,
where $  \{\thonet\}_{t \in \R} $ is defined by
\eqref{shift1}.  
  By  Definition \ref{comporbit},   a map 
   $\psi$:    $\R \times \R \times \Omega_2$ 
   $\to X$ is    
      a complete orbit of $\Phi$ if  
       for every  $t \ge 0$,  $\tau \in \R$,
 $ s  \in \R$  and $\omega \in \Omega_2$,     the following   holds:
\be
\label{nonau21}
 \Phi (t,   \tau + s, \theta_{2, \tau} \omega, 
  \psi (\tau, s, \omega ) )
  = \psi (t + \tau, s,  \omega).
\ee
Let $\xi: \R \times \Omega_2 $  $ \to X$ satisfy, for all $t \ge 0$,  $\tau \in \R$
and $\omega \in \Omega_2$,
\be
\label{nonau22}
\Phi(t, \tau, \omega,  \xi (\tau, \omega) ) = \xi (t+ \tau, \thtwot \omega).
\ee
Any function $\xi: \R \times  \Omega_2 \to X$ with
property  \eqref{nonau22} is said to be 
 a  complete quasi-solution of
$\Phi$ in the sense of \eqref{nonau22}. 
If,  in addition,  
 there exists $D=\{ D(t, \omega ):  t \in \R,  \omega \in \Omega_2
  \} \in \cald$
such that $\xi (t, \omega  ) \in D( t, \omega)$ for all $ t \in \R$
and $\omega \in \Omega_2$,
then, we say       $\xi$  is   a $\cald$-complete  quasi-solution
of $\Phi$    in the sense of
\eqref{nonau22}.
If $\psi: \R \times \R \times \Omega_2$  $ \to X$ is a complete  orbit
of $\Phi$ according to  Definition \ref{comporbit}, 
then 
for every $s \in \R$,  it is easy  to check   that
 the  map
$\xi^s$:  $ (\tau, \omega)  \in \R \times \Omega_2$
$\to    \psi (\tau -s,  s,  \theta_{2, s-\tau} \omega)$ is a complete 
quasi-solution 
of $\Phi$ in the sense of \eqref{nonau22}.
Further,  $\xi^s (s, \omega) = \psi (0, s, \omega)$.
  Similarly,  if   $\xi: \R \times \Omega_2  \to X$  is a 
  complete  quasi-solution of  $\Phi$ under \eqref{nonau22},
  then  the  map $\psi: \R \times \R \times \Omega_2 \to X$  defined by
\be\label{nonau23}
\psi (t , s, \omega) = \xi (t+s, \thtwot\omega),
\quad \mbox{ for all } \ t, \ s   \in \R \  \mbox{ and }
\omega \in \Omega_2,
\ee
is a complete orbit of $\Phi$  in the sense of \eqref{nonau21}. 
In addition, $\psi (0, s, \omega) = \xi(s, \omega)$.
Using \eqref{nonau21}-\eqref{nonau23} and applying
Theorem \ref{att} to the case $\Omega_1 =\R$, we get
the following result.

\begin{prop}
\label{app5}
Let   $\cald$ be a  
neighborhood closed  collection of some 
 families of   nonempty subsets of
$X$ and $\Phi$  be a continuous   cocycle on $X$
over $(\R,  \{\thonet\}_{t \in \R})$
and
$(\Omega_2, \calftwo, P,  \{\thtwot\}_{t \in \R})$,
where $\{\thonet\}_{t \in \R}$ is defined by \eqref{shift1}.
Then
$\Phi$ has a  $\cald$-pullback
attractor $\cala$  in $\cald$
if and only if
$\Phi$ is $\cald$-pullback asymptotically
compact in $X$ and $\Phi$ has a  closed
   measurable (w.r.t. the $P$-completion of $\calftwo$)
     $\cald$-pullback absorbing set
  $K$ in $\cald$.
  The $\cald$-pullback
attractor $\cala$   is unique   and is given  by,
for each $\tau   \in \R$   and
$\omega  \in \Omega_2$,
\be\label{app5_a1}
\cala (\tau, \omega )
=\Omega(K, \tau, \omega )
=\bigcup_{B \in \cald} \Omega(B, \tau, \omega )
\ee
\be\label{app5_a2}
 =\{\psi(0, \tau, \omega ): \psi \mbox{ is a   }  \cald {\rm -}
 \mbox{complete orbit of } \Phi
  \  \mbox{ under Definition }
  \eqref{comporbit}
 \} 
\ee
\be\label{app5_a3}
   = \{ \xi (\tau, \omega): \xi \mbox{ is a  }\  \cald{\rm -}
\mbox{complete  quasi-solution  of }
\Phi \mbox{  in the sense of } \ \eqref{nonau22} \}.
\ee 
  \end{prop}

   We now consider  periodic pullback attractors.
   Let $T>0$   
and 
 $\Phi$  be  a periodic cocycle with period $T$, that is,  
 for every  $t \ge 0$,  $\tau \in \R$ and $\omega \in \Omega_2$,  
$ 
 \Phi (t, \tau + T,  \omega, \cdot ) = \Phi (t, \tau,   \omega, \cdot)$.  Then
 it follows    from Theorems \ref{periodatt}
 and \ref{periodatt2}  that  the $\cald$-pullback attractor of $\Phi$
 is also periodic  under certain circumstances.

\begin{prop}
\label{app6} 
Suppose  $\Phi$   is  a continuous  periodic   cocycle
with period $T>0$  on $X$
over $(\R,  \{\thonet\}_{t \in \R})$
and
$(\Omega_2, \calftwo, P,  \{\thtwot\}_{t \in \R})$, where
$\{\thonet\}_{t \in \R}$ is defined by \eqref{shift1}.
 Let   $\cald$   be  a  neighborhood closed
   and $T$-translation invariant collection of some  families of   nonempty subsets of
$X$.
 If
$\Phi$ is $\cald$-pullback asymptotically
compact in $X$ and $\Phi$ has a  closed
   measurable (w.r.t. the $P$-completion of $\calftwo$)
     $\cald$-pullback absorbing set
  $K$ in $\cald$, then $\Phi$
  has a unique periodic 
   $\cald$-pullback
attractor $\cala \in \cald$
with period  $T$, i.e.,    
for each $\tau \in \R$ 
and $\omega \in \Omega_2$,  
$\cala(\tau + T, \omega ) = \cala (\tau , \omega)$.
\end{prop}

We now discuss   the second approach to deal with 
non-autonomous equations by taking $\Omega_1
=\Omega_f$, where $f: \R \to  X$  is 
a  parameter contained
in an equation.
Note that for any $h_1, \ h_2 \in \R$
with $h_1 \neq h_2$,  we have
$f^{h_1} \neq f^{h_2}$
as long as $f$ is not periodic in $X$.
In this case, we can define a group
$\{\thonet\}_{t \in \R}$ acting on $\Omega_f$   by
\be\label{shift2}
\thonet f^h = f^{t+h} \quad
\mbox{  for all } \ t \in \R
\quad \mbox{ and } \ f^h \in \Omega_f.
\ee

 If  $f$ is   a  periodic   function
    with   minimal period $T>0$, then
    $\Omega_f =\Omega_f^T$,  where
 $\Omega_f^T
      =\{f^h:   0\le h <T \}$. 
For  each $ t \in \R$,
 we  define a map
$\thonet$: $\Omega_f^T \to \Omega_f^T$  such that
for each   
$0\le h <T$,   
\be
\label{shift3}
\thonet f^h = f^\tau, 
\ee
where  $\tau \in [0, T)$ is the unique number
such that  
  $t+h =mT  + \tau$  for
  some   integer $m$.  
Then $\{\thonet\}_{t \in \R}$ given
by \eqref{shift3} is  a group acting on $\Omega_f^T$
  for  a periodic  $f$   with  minimal period $T$.
  
  We first consider   the case where $f$ is not periodic in $X$.  
   Let  $\Phi$:  $ \R^+ \times \Omega_f \times  \Omega_2  \times X$  
    $  \to X$  be    
  a continuous  cocycle on $X$
over $(\Omega_f,  \{\thonet\}_{t \in \R})$
and $(\Omega_2, \calftwo, P, \{\thtwot\}_{t \in \R})$, 
where $\{\thonet\}_{t \in \R}$  is defined by \eqref{shift2}.
Suppose  $\psi$:   
    $\R \times \Omega_f  \times \Omega_2 
    \to X$ is a complete orbit of $\Phi$, that is,  
       for every  $t \ge 0$,  $\tau \in \R$,  
 $f^h \in \Omega_f$  and  $\omega \in \Omega_2$, 
\be
\label{nonau27}
 \Phi (t,   f^{\tau +h},  \theta_{2, \tau} \omega,
  \psi (\tau, f^h, \omega ) )
  = \psi (t + \tau, f^h, \omega).
\ee
Let $\xi: \R  \times \Omega_2 \to X$ satisfy,
 for all $t \ge 0$,  $f^h \in \Omega_f$  and $\omega \in \Omega_2$,
\be
\label{nonau28}
\Phi(t, f^h, \omega,  \xi (h, \omega)) = \xi (t+ h, \thtwot \omega).
\ee 
Such  a   function $\xi$  is called a complete 
quasi-solution of $\Phi$ in the sense   of  \eqref{nonau28}.
If,  in addition,  
 there exists $D=\{ D(f^t, \omega ): f^t \in \Omega_f,
 \omega \in \Omega_2
  \} \in \cald$
such that $\xi (t, \omega  ) \in D(f^t, \omega)$ for all $ t \in \R$ and $\omega \in \Omega_2$,
then, we say       $\xi$  is   a $\cald$-complete  quasi-solution
of $\Phi$    in the sense of
\eqref{nonau28}.
If $\psi: \R \times \Omega_f \times \Omega_2$  $ \to X$ is a $\cald$-complete  orbit
of $\Phi$  by  Definition \ref{comporbit}, 
then 
for every $h \in \R$,  we can show    that
 the  map
$\xi^h$:  $ (\tau, \omega)  \in \R \times \Omega_2$
$\to    \psi (\tau -h, f^h,  \theta_{2, h-\tau} \omega)$ is a $\cald$-complete 
quasi-solution
of $\Phi$ in the sense of \eqref{nonau28}.   
 Vice versa,   if   $\xi: \R \times \Omega_2  \to X$  is a 
  $\cald$-complete    quasi-solution
   of  $\Phi$ by
   \eqref{nonau28},
  then  the  map $\psi: (t, f^h, \omega)\in$ 
  $\R \times \Omega_f \times \Omega_2 \to  
  \xi (t+h, \thtwot\omega)
$ is a $\cald$-complete  orbit of $\Phi$
in terms of  Definition \ref{comporbit}, 
and $\psi(0, f^h, \omega) =\xi(h, \omega)$.
Thus,    we have the following result from Theorem \ref{att}.

\begin{prop}
\label{app7}
Suppose    $f: \R \to X$ is   not   a    periodic function. 
Let $\cald$ be a  neighborhood closed  collection of some 
 families of   nonempty subsets of
$X$ and $\Phi$  be a continuous   cocycle on $X$
over $(\Omega_f,  \{\thonet\}_{t \in \R})$
and
$(\Omega_2, \calftwo, P,  \{\thtwot\}_{t \in \R})$,
where $\{\thonet\}_{t \in \R}$ is defined by \eqref{shift2}.
  Then
$\Phi$ has a  $\cald$-pullback
attractor $\cala$  in $\cald$
if and only if
$\Phi$ is $\cald$-pullback asymptotically
compact in $X$ and $\Phi$ has a  closed
     measurable (w.r.t. the $P$-completion of $\calftwo$)
       $\cald$-pullback absorbing set
  $K$ in $\cald$.
  The $\cald$-pullback
attractor $\cala $  is unique   and is given  by,
for each $f^h \in \Omega_f$ and $\omega \in \Omega_2$,
$$
\cala (f^h, \omega )
=\Omega(K, f^h , \omega)
=\bigcup_{B \in \cald} \Omega(B, f^h , \omega)
$$
$$
 =\{\psi(0, f^h, \omega ): \psi \ \mbox{     is a  }\
  \cald{\rm -}  \mbox{complete  orbit  of }  \Phi
   \  \mbox{ under Definition }
  \eqref{comporbit}
  \}
 $$
$$
   = \{ \xi (h, \omega): \xi \mbox{ is a }\  \cald{\rm -}
\mbox{complete  quasi-solution  of }
\Phi \mbox{  in the sense of } \ \eqref{nonau28} \}.
$$
  \end{prop}

  We next examine 
the relations  between  Propositions
\ref{app5} and \ref{app7}.
In this respect,    we will prove    that
  these two  propositions
are  essentially  the  same  
if  $f$  is  not a periodic  function. 
Suppose   $\Phi$ is a 
  continuous   cocycle on $X$
over 
$(\Omega_f,  \{\thonet\}_{t \in \R})$
and
$(\Omega_2, \calftwo, P,  \{\thtwot\}_{t \in \R})$
with  $\{\thonet\}_{t \in \R} $  being
given by \eqref{shift2}.
Then 
we can associate  with $\Phi$ 
another continuous cocycle $\tilde{\Phi}$:
$ \R^+ \times \R \times\Omega_2 \times X
\to X$ by
\be\label{nonau29}
{\tilde{\Phi}} (t, \tau, \omega,  \cdot)
=\Phi (t, f^\tau, \omega,  \cdot), \quad \mbox{ for all }
t \ge 0, \  \tau \in \R \ \mbox{ and } \ 
 \omega \in \Omega_2.
\ee
Indeed,  one can show    that ${\tilde{\Phi}}$
is a continuous cocycle  on $X$ over
$(\R,  \{\thonet\}_{t \in \R})$
and
$(\Omega_2, \calftwo, P,  \{\thtwot\}_{t \in \R})$
 with  $\{\thonet\}_{t \in \R} $ being
 given by \eqref{shift1}.
Suppose  that 
$ 
\cald = \{ D= \{D(f^h, \omega): f^h \in \Omega_f, \omega \in \Omega_2 \} \}
$  is  a collection of some families  of nonempty subsets of $X$.
Give   $D= \{D(f^h, \omega): f^h \in \Omega_f, \omega \in \Omega_2  \} \in \cald$,  denote by 
\be
\label{nonau30}
{\tilde{D}} = \{ {\tilde{D}} (\tau, \omega):  
{\tilde{D}} (\tau, \omega) = D(f^\tau, \omega), \ \tau  \in \R, \  \omega \in \Omega_2 \}.
\ee
Let  ${\tilde{\cald}}$  be the collection corresponding to $\cald$ 
  as defined by  $${\tilde{\cald}}  = \{ 
{\tilde{D}} :  {\tilde{D}} \mbox{ is given by }  \eqref{nonau30},   \  D \in \cald \}.
$$
 If   $\cala = \{ \cala(f^h, \omega): 
  f^h \in \Omega_f , \ \omega \in \Omega_2
 \}$   is
the $\cald$-pullback attractor of $\Phi$,  then it is easy to check 
   that
  ${\tilde{\cala}} = \{{\tilde{ \cala}} (\tau, \omega): \tau \in \R , \
  \omega \in \Omega_2
  \}$    is
  the  ${\tilde{\cald}}$-pullback attractor
  of ${\tilde{\Phi}}$, where
  ${\tilde{ \cala}} (\tau, \omega) = \cala(f^\tau, \omega)$  for   each
  $\tau \in \R$
  and $\omega \in \Omega_2$.  The converse of this statement   is also  true if
  $f$ is not  a periodic function because, in 
  that   case, one can define
  $\Phi$, $\cald$   and $\cala$  for given
  ${\tilde{\Phi}}$,  ${\tilde{\cald}} $  and 
  ${\tilde{\cala}} $  by the reverse
  of  the above process.  However,    for    periodic
  $f$ 
  with  minimal period $T>0$,    we can not define
   the collection  $\cald$ from  ${\tilde{\cald}}$
  by the  inverse of \eqref{nonau30}  if  there  are
   $\tilde{D} \in {\tilde{\cald}}$,
      $\tau \in \R$
      and $\omega \in \Omega_2$
  such that ${\tilde{D}}( \tau, \omega )
  \neq {\tilde{D}}(  \tau +T, \omega )$. 
  In this case,  we can take $\Omega_1 = \Omega_f^T$
  with $\{\thonet\}_{t \in \R }$  being given  by
  \eqref{shift3}.    
Then applying  Theorem \ref{att},  we  find  the following result.

\begin{prop}
\label{app8}
Suppose 
 $f: \R \to X$   is   periodic     with   minimal period $T>0$.
Let $\cald$ be a
neighborhood closed  collection of some  families of   nonempty subsets of
$X$ and $\Phi$  be a continuous   cocycle on $X$
over $(\Omega_f^T,  \{\thonet\}_{t \in \R})$
and
$(\Omega_2, \calftwo, P,  \{\thtwot\}_{t \in \R})$,
where $\{\thonet\}_{t \in \R}$ is defined by \eqref{shift3}.
  Then
$\Phi$ has a  $\cald$-pullback
attractor $\cala$  in $\cald$
if and only if
$\Phi$ is $\cald$-pullback asymptotically
compact in $X$ and $\Phi$ has a  closed
     measurable (w.r.t. the $P$-completion of $\calftwo$)
       $\cald$-pullback absorbing set
  $K$ in $\cald$.
  The $\cald$-pullback
attractor $\cala $  is unique   and is given  by,
for each $f^h \in \Omega_f^T$ and $\omega \in \Omega_2$,
$$
\cala (f^h, \omega )
=\Omega(K, f^h , \omega)
=\bigcup_{B \in \cald} \Omega(B, f^h , \omega)
$$
$$
 =\{\psi(0, f^h, \omega ): \psi \ \mbox{     is a   }\
  \cald{\rm -}  \mbox{complete  orbit  of }  \Phi
   \  \mbox{ under Definition }
  \eqref{comporbit}
  \}
 $$
$$
   = \{ \xi (h, \omega): \xi \mbox{ is a }\  \cald{\rm -}
\mbox{complete  quasi-solution  of }
\Phi \mbox{  in the sense of } \ \eqref{nonau28} \}.
$$
  \end{prop}

Note that the  $\cald$-pullback attractor $\cala
=\{\cala (f^h, \omega): h \in [0, T), \omega \in \Omega_2 \}$ of $\Phi$
obtained in Proposition \ref{app8} can be extended to a periodic
attractor by setting,  for every $h \in \R$ and $\omega \in \Omega_2$, 
$\cala(f^h, \omega)= \cala ( f^ \tau,  \omega)$
with $\tau \in [0, T)$ such that $h = mT + \tau$  for some
integer $m$.

\subsection{Attractors  for  Equations with only Random Forcing  Terms}
 
  To deal with differential equations with only stochastic forcing
  terms but without non-autonomous deterministic terms,
  we may take $\Omega_1 = \{\omega^*\}$ as a singleton, and define
  $\{\thonet\}_{t \in \R}$ by $\thonet \omega^* = \omega^*$
  for all $ t \in \R$.
  In this case, Definition \ref{ds1}   is the same as the concept of cocycles 
    as introduced in \cite{arn1, cra2, fla1, schm1}.
  We now  drop the dependence of
  all variables on $\Omega_1$ and write
   $(\Omega_2, \calftwo, P,  \{\thtwot\}_{t \in \R})$
   as $(\Omega, \calf, P,  \{\theta_t\}_{t \in \R})$.
  For the reader's convenience,  we rephrase Theorem \ref{att}
  as follows.

    \begin{prop}
\label{app4}
Let 
 $\cald$ be  a  neighborhood closed  collection of some  families of   nonempty subsets of
$X$ and $\Phi$  be a continuous   cocycle on $X$
over $(\Omega, \calf, P,  \{\theta_t\}_{t \in \R})$.
Then
$\Phi$ has a  $\cald$-pullback
attractor $\cala$  in $\cald$
if and only if
$\Phi$ is $\cald$-pullback asymptotically
compact in $X$ and $\Phi$ has a  closed
   measurable (w.r.t. the $P$-completion of $\calf$)
     $\cald$-pullback absorbing set
  $K$ in $\cald$.
  The $\cald$-pullback
attractor $\cala$   is unique   and is given  by,
for each
$\omega  \in \Omega $,
$$
\cala ( \omega )
=\Omega(K, \omega )
=\bigcup_{B \in \cald} \Omega(B, \omega )
$$
$$
=\{\psi(0, \omega ): \psi  \mbox{  is a }  \cald {\rm-}
  \mbox{complete  orbit of }   \Phi 
   \  \mbox{ under Definition }
  \eqref{comporbit}
  \}.
  $$
  \end{prop}
  
  It is worth noticing    that,
  in the case where $\Omega_1$ is a singleton,
  a map $\psi$:  $\R   \times \Omega_2 \to X$
  is  a complete orbit of $\Phi$
  by Definition \ref{comporbit}  if and only   if
  for each   $t \ge 0$,  $\tau \in \R$  and
  $\omega \in \Omega_2$,  there holds
  \be\label{cbitran1}
   \Phi (t, \theta_{2, \tau} \omega, \psi(\tau, \omega))=
  \psi (t+ \tau, \omega).
  \ee
  In other words,  if   a system contains  only
 random forcing without  other  non-autonomous  deterministic
 forcing, then the concept  of  a  complete orbit
 of the system should be defined by 
 \eqref{cbitran1}.   
 Note that  \eqref{cbitran1}   is different  from \eqref{nonau22} 
 and that is why  a function with property \eqref{nonau22} 
 is called a complete quasi-solution   instead of  a solution.
 
 As mentioned in the Introduction,  if the cocycle $\Phi$ is compact in
  the sense that it has a compact $\cald$-pullback absorbing set, 
   then the existence of 
 $\cald$-pullback attractors of $\Phi$ was proved in \cite{fla1}. 
 The compactness of $\Phi$  was later replaced by the 
    $\cald$-pullback asymptotic   compactness  and the existence of attractors
    in this   case was proved in \cite{bat1}. 
    Under further assumptions that $\cald$ is neighborhood closed, 
    Proposition \ref{app4}  shows that
    the $\cald$-pullback  asymptotic compactness
     and the existence
    of closed measurable $\cald$-pullback absorbing sets 
     of $\Phi$
      are  not only sufficient, but also necessary for existence of 
      $\cald$-pullback attractors. In addition, Proposition \ref{app4} 
      provides a characterization of 
   pullback    attractors  in terms of the $\cald$-complete orbits of $\Phi$.

\subsection{Attractors for Non-autonomous Deterministic   Systems}

As a special case of systems with both deterministic and random terms as discussed in Section \ref{sec31},  we now consider the systems with only
non-autonomous deterministic parameters.
Of course, all results presented in Section
 \ref{sec31} are   valid    in this special case.
 Because  pullback attractors  for
 non-autonomous deterministic systems
 are  interesting in their own right, 
 it is worth  to rewrite  the main results of Section \ref{sec31} for  this specific case.

  To handle non-autonomous deterministic systems,
  we may take
  $\Omega_2 =\{\omega^*\}$ as a singleton and
  either $\Omega_1 = \R$ or $\Omega_1 =\Omega_f$.
     Suppose  $\Omega_2 =\{\omega^*\}$ 
     is a singleton.  Let $P$ be the probability on
  $(\{\omega^*\}, \calftwo )$ with $\calftwo =\{ \{\omega^*\}, \emptyset \}$.
  For all  $t \in \R$, set $\thtwot \omega^* = \omega^*$.
   We first consider  the  case where $\Omega_1 =\R$
   with $ \{\thonet\}_{t \in \R}$ being given by
   \eqref{shift1}.
   Since $\Omega_2$ is a singleton,   from now on, we will drop the dependence of all variables on $\Omega_2$.
   
    Let  $\Phi$:    $ \R^+ \times \R   \times X$
        $ \to  X$  be  a continuous  cocycle on $X$
over $(\R,  \{\thonet\}_{t \in \R})$ with
$\{\thonet\}_{t \in \R}$ being given by
\eqref{shift1}.  The existence of  pullback 
attractors  for such $\Phi$ has been  investigated
by many authors in \cite{car3, car4, wan4}
and the references  therein.  Our results  here will provide
not only  sufficient    but also necessary  criteria for existence of $\cald$-pullback attractors. 
 
    Let $\psi$: 
    $\R \times \R  \to X$ be a 
      a complete orbit of $\Phi$  in the sense of  Definition
      \ref{comporbit},  that is,   
       for every $t \ge 0$,  $\tau \in \R$ 
      and
 $s \in \R$,   there  holds:
\be
\label{nonau1}
 \Phi (t,   \tau + s,
  \psi (\tau, s ) )
  = \psi (t + \tau, s).
\ee
Let $\xi: \R \to X$ satisfy, for all $t \ge 0$   and $\tau \in \R$,
\be
\label{nonau2}
\Phi(t, \tau) \xi (\tau) = \xi (t+ \tau).
\ee
In the literature, any $\xi$ with property
\eqref{nonau2} is called a complete solution of $\Phi$,
see, e.g., \cite{bal2, che1, tem1}.
If there exists $D=\{ D(t): t \in \R \} \in \cald$
such that $\xi (t ) \in D(t)$ for all $ t \in \R$,
then  we  say   $\xi$  is   a $\cald$-complete solution of $\Phi$  in the sense of
\eqref{nonau2}.
Our definition \eqref{nonau1} is different but closely related
to \eqref{nonau2}.  Actually, if $\psi: \R \times \R \to X$
has   property  \eqref{nonau1}, then for every fixed $s \in \R$,
  $\xi^s (\cdot) = \psi(\cdot -s, s)$
  maps $\R$ into  $X$  with property \eqref{nonau2}.
In other words, for every $s \in \R$, $\xi^s$ is a complete
solution in the sense of \eqref{nonau2}, and 
   $\xi^s (s) = \psi (0, s)$.
Similarly, given  $\xi: \R \to X$  with property
\eqref{nonau2}, define a map $\psi: \R \times \R \to X$ by
\be\label{nonau3}
\psi (t , s) = \xi (t+s),
\quad \mbox{ for all } \ t, \ s  \in \R.
\ee
Then    
$\psi$ is a complete orbit  of $\Phi$  in
the sense of \eqref{nonau1}. 
 It follows from  \eqref{nonau3} and  Lemma  \ref{attstr1}
 that,  if $\Phi$ has a $\cald$-pullback
attractor $\cala=\{\cala(\tau): \tau \in \R \}$, then
for every $\tau \in \R$,
\be
\label{nonau5}
\cala(\tau)
= \{ \xi (\tau): \xi \mbox{ is a }\  \cald{\rm -complete}
\mbox{ solution in the sense of } \ \eqref{nonau2} \}.
\ee
The characterization of $\cala$ given by
\eqref{nonau5}
can be found in  \cite{che1} for non-autonomous deterministic equations,
where $\cala$ is called kernel sections instead of pullback attractors.
By  dropping the dependence of  variables on  $\Omega_2$,  we have the following  result from 
Proposition \ref{app5}.

\begin{prop}
\label{app1}
 Let  
 $\cald$ be   a  neighborhood closed  collection of some  
 families of   nonempty subsets of
$X$ and $\Phi$  be a continuous   cocycle on $X$
over $(\R,  \{\thonet\}_{t \in \R})$
with  $\{\thonet\}_{t \in \R} $ given by \eqref{shift1}.
 Then
$\Phi$ has a  $\cald$-pullback
attractor $\cala$  in $\cald$
if and only if
$\Phi$ is $\cald$-pullback asymptotically
compact in $X$ and $\Phi$ has a  closed
      $\cald$-pullback absorbing set
  $K$ in $\cald$.
  The $\cald$-pullback
attractor $\cala = \{ \cala(\tau)\}_{\tau \in \R}$   is unique   and is given  by,
for each $\tau  \in \R$,
$$
\cala (\tau )
=\Omega(K, \tau )
=\bigcup_{B \in \cald} \Omega(B, \tau )
$$
$$
 =\{\psi(0, \tau ): \psi \ \mbox{     is a  }\
  \cald{\rm -}  \mbox{complete  orbit  of }  \Phi 
   \  \mbox{ under Definition }
  \eqref{comporbit}
  \}
 $$
$$
   = \{ \xi (\tau): \xi \mbox{ is a }\  \cald{\rm -}
\mbox{complete  solution  of }
\Phi \mbox{  in the sense of } \ \eqref{nonau2} \}.
$$
  \end{prop}

In the case where 
 $\Phi$ is a periodic cocycle with period $T$
 (i.e.,  $  \Phi (t, s + T, \cdot ) = \Phi (t, s, \cdot)$
 for all $t \ge 0$   and $s \in \R$),  Proposition \ref{app6}
 implies the result  below.

\begin{prop}
\label{app2} 
Suppose  $\Phi$   is  a continuous  periodic   cocycle
with period $T>0$  on $X$
 over $(\R,  \{\thonet\}_{t \in \R})$
with  $\{\thonet\}_{t \in \R} $ given by \eqref{shift1}.
  Let   $\cald$   be  a  neighborhood closed
   and $T$-translation invariant collection of some 
    families of   nonempty subsets of
$X$.
 If
$\Phi$ is $\cald$-pullback asymptotically
compact in $X$ and $\Phi$ has a  closed 
     $\cald$-pullback absorbing set
  $K$ in $\cald$, then $\Phi$
  has a unique periodic 
   $\cald$-pullback
attractor $\cala =\{ \cala(\tau) \}_{\tau \in \R} \in \cald$,
that  is,   
for each $\tau \in \R$,  
$\cala(\tau + T) = \cala (\tau )$.
\end{prop}

 We now   discuss  the case  where  
     $\Omega_2 =\{\omega^*\}$  is a singleton
       and $\Omega_1 = \Omega_f$.
       Suppose $f: \R \to X$ is not a periodic function.
 Let  $\Phi$:  $ \R^+ \times \Omega_f \times \{\omega^*\} \times X$   $  \to X$  be    
  a continuous  cocycle on $X$
over $(\Omega_f,  \{\thonet\}_{t \in \R})$
and $(\{\omega^*\}, \calftwo, P, \{\thtwot\}_{t \in \R})$, 
where $\{\thonet\}_{t \in \R}$ is given by \eqref{shift2}.
For  convenience, we drop the dependence  of variables on $\Omega_2$ from now on.
Suppose  $\psi$:  
    $\R \times \Omega_f  \to X$  is  
      a complete orbit of $\Phi$ in the sense of Definition \ref{comporbit}, i.e.,  
       for every $\tau \in \R$, $t \ge 0$
      and
 $f^h \in \Omega_f$, 
\be
\label{nonau7}
 \Phi (t,   f^{\tau +h}, 
  \psi (\tau, f^h ) )
  = \psi (t + \tau, f^h).
\ee
Let $\xi: \R \to X$ satisfy, for all $t \ge 0$   and $f^h \in \Omega_f$,
\be
\label{nonau8}
\Phi(t, f^h) \xi (h) = \xi (t+ h).
\ee 
If,  in addition,  
 there exists $D=\{ D(f^t ): f^t \in \Omega_f
  \} \in \cald$
such that $\xi (t ) \in D(f^t)$ for all $ t \in \R$,
then, we say       $\xi$  is   a $\cald$-complete solution
of $\Phi$    in the sense of
\eqref{nonau8}.
If $\psi: \R \times \Omega_f \to X$
is a complete orbit of $\Phi$ in the sense of  Definition 
\ref{comporbit},  then by \eqref{nonau7}, 
  for every fixed $h \in \R$,
  $\xi^h  (\cdot) = \psi(\cdot -h, f^h)$
  is a complete orbit of $\Phi$ in the sense of  \eqref{nonau8}
  and   $\xi^h (h) = \psi (0, f^h)$.
 Similarly,  if    $\xi: \R \to X$   is a complete   solution
  of $\Phi$ under \eqref{nonau8}, then the mapping 
   $\psi:  (t, f^h)  \in  \R \times \Omega_f 
   \to  \psi (t ,  f^h) = \xi (t+h)$  is a complete 
   solution of $\Phi$  under Definition \ref{comporbit}
   and $\psi(0, f^h) = \xi (h)$. 
   Thus,  as a special    case of 
   Proposition  \ref{app7}   with  
   $\Omega_2 =\{\omega^*\}$,  we have the following result.

\begin{prop}
\label{app3}
 Suppose  
  $f: \R \to X$ is not a periodic function.
 Let $\cald$ be a  neighborhood closed  collection of some
   families of   nonempty subsets of
$X$ and $\Phi$  be a continuous   cocycle on $X$
over $(\Omega_f,  \{\thonet\}_{t \in \R})$
with  $\{\thonet\}_{t \in \R} $ given by \eqref{shift2}.
 Then
$\Phi$ has a  $\cald$-pullback
attractor $\cala$  in $\cald$
if and only if
$\Phi$ is $\cald$-pullback asymptotically
compact in $X$ and $\Phi$ has a  closed
      $\cald$-pullback absorbing set
  $K$ in $\cald$.
  The $\cald$-pullback
attractor $\cala = \{ \cala(f^h)\}_{f^h \in \Omega_f}$   is unique   and is given  by,
for each $f^h \in \Omega_f$,
$$
\cala (f^h )
=\Omega(K, f^h )
=\bigcup_{B \in \cald} \Omega(B, f^h )
$$
$$
 =\{\psi(0, f^h ): \psi \ \mbox{     is a  }\
  \cald{\rm -}  \mbox{complete  orbit  of }  \Phi  \  \mbox{ under Definition }
  \eqref{comporbit} \}
 $$
$$
   = \{ \xi (h): \xi \mbox{ is a }\  \cald{\rm -}
\mbox{complete  solution  of }
\Phi \mbox{  in the sense of } \ \eqref{nonau8} \}.
$$
  \end{prop}

  Based on the translations of $f$, the attractors  of 
  deterministic equations  can also be studied by the
  skew semigroup method, see, e.g., \cite{che1}.  The idea of
  this approach  is to lift  a   cocycle   to a semigroup 
  defined in an extended space.  Let  $C_b(\R, X)$ be the
  space of  bounded continuous functions from $\R $ to $X$ with
  the supremum norm.   
  Given  $f \in C_b(\R, X)$,    let  $\Sigma$ be the closure of $\Omega_f$
  with respect to the topology of $C_b(\R, X)$. 
  If $\Sigma$ is compact, then the skew semigroup method can
  be  used
  to deduce the existence of non-autonomous attractors under
  certain circumstances, see  \cite{che1} for instance.
  Our approach with $\Omega_1 =\Omega_f$  is  different from the
  skew semigroup method in the sense that we consider
  the  set $\Omega_f$   only as a parametric space;
  neither the boundedness of $f$, nor the precompactness of $\Omega_f$,  is needed.    This is why Proposition \ref{app3} 
  could  be applied to   an   unbounded $f$ with 
  even exponentially  growing rate as $t \to \pm  \infty$.

 Note that   Propositions 
\ref{app1} and \ref{app3}
are essentially the same as long as $f$ is not a periodic  function.
This follows  from   the  discussions 
presented  in Section \ref{sec31} by setting
$\Omega_2 =\{\omega^*\}$.  
For  periodic  $f$, we have the following result  from 
Proposition \ref{app8}.

  \begin{prop}
\label{app20}
Suppose  
 $f: \R \to X$   is   periodic     with   minimal period $T>0$.
Let $\cald$ be a  neighborhood closed  collection of some 
 families of   nonempty subsets of
$X$ and $\Phi$  be a continuous   cocycle on $X$
over $(\Omega_f^T,  \{\thonet\}_{t \in \R})$, 
where $\{\thonet\}_{t \in \R}$ is defined by \eqref{shift3}.
  Then
$\Phi$ has a  $\cald$-pullback
attractor $\cala$  in $\cald$
if and only if
$\Phi$ is $\cald$-pullback asymptotically
compact in $X$ and $\Phi$ has a  closed
       $\cald$-pullback absorbing set
  $K$ in $\cald$.
  The $\cald$-pullback
attractor $\cala $  is unique   and is given  by,
for each $f^h \in \Omega_f^T$,
$$
\cala (f^h )
=\Omega(K, f^h )
=\bigcup_{B \in \cald} \Omega(B, f^h )
$$
$$
 =\{\psi(0, f^h ): \psi \ \mbox{     is a  }\
  \cald{\rm -}  \mbox{complete  orbit  of }  \Phi
   \  \mbox{ under Definition }
  \eqref{comporbit}
  \}
 $$
$$
   = \{ \xi (h): \xi \mbox{ is a }\  \cald{\rm -}
\mbox{complete  solution  of }
\Phi \mbox{  in the sense of } \ \eqref{nonau8} \}.
$$
  \end{prop}

\section{Pullback Attractors for Reaction-Diffusion Equations}
\setcounter{equation}{0}

As   an application of our main results, in this section, we study pullback
attractors   for Reaction-Diffusion  equations defined
on $\R^n$    with both non-autonomous deterministic and stochastic
forcing  terms. 
We  first  define  a continuous cocycle  for the equation in $\ltwo$, and then
prove the pullback asymptotic compactness of solutions as well as
the existence of  pullback  absorbing sets. We  will also discuss
periodic pullback attractors of the equation when the deterministic forcing terms
are periodic   functions in $\ltwo$.

\subsection{Cocycles  for Reaction-Diffusion Equations on $\R^n$}
\setcounter{equation}{0}

 Given   $\tau \in\R$ and $t > \tau$,  consider      the
 following non-autonomous  Reaction-Diffusion  equation
 defined for  $x \in \R^n$,
\be
  \label{41}
  du +  (\lambda u - \Delta u) dt    =  f(x, u) dt  +  g(x, t) dt + h(x) d \omega,
 \ee
 with the initial data
 \be\label{42}
 u(x, \tau) = u_\tau (x),   \quad x\in \R^n,
 \ee
where   $\lambda$ is a  positive constant,
$g \in L^2_{loc}(\R, \ltwo)$,  $h \in H^2(\R^n)\bigcap W^{2,p} (\R^n)$
for some $p \ge 2$,
$\omega$  is a  two-sided real-valued Wiener process on a probability space.
 The nonlinearity    $f$   is 
  a smooth
function  that satisfies, for some positive constants $\alpha_1$,
$\alpha_2$ and $\alpha_3$,
\begin{equation}
\label{f1}
f(x, s) s \le -  \alpha_1 |s|^p + \psi_1(x), \quad \forall \ x \in \R^n, \ \ \forall \ s \in \R, 
\end{equation}
\begin{equation}
\label{f2}
|f(x, s) |   \le \alpha_2 |s|^{p-1} + \psi_2 (x),
\quad \forall \ x \in \R^n, \ \ \forall \ s \in \R, 
\end{equation}
\begin{equation}
\label{f3}
{\frac {\partial f}{\partial s}} (x, s)   \le \alpha_3,
\quad \forall \ x \in \R^n, \ \ \forall \ s \in \R, 
\end{equation}
\begin{equation}
\label{f4}
| {\frac {\partial f}{\partial x}} (x, s) | \le  \psi_3(x),
\quad \forall \ x \in \R^n, \ \ \forall \ s \in \R, 
\end{equation}
where 
$\psi_1 \in L^1(R^n) \cap L^\infty(R^n)$  and $\psi_2, \psi_3 \in L^2(R^n)$.

To describe   the   probability space  that   will be   used
in this paper,  we write  
$$
\Omega = \{ \omega   \in C(\R, \R ): \ \omega(0) =  0 \}.
$$
Let $\calf$  be
 the Borel $\sigma$-algebra induced by the
compact-open topology of $\Omega$, and $P$
be  the corresponding Wiener
measure on $(\Omega, \calf)$.   
 Define  a group  $\{\thtwot \}_{t \in \R}$  acting on  
 $(\Omega, \calf, P)$
   by
\be\label{rdshift1}
 \thtwot \omega (\cdot) = \omega (\cdot +t) - \omega (t), \quad  \omega \in \Omega, \ \ t \in \R .
\ee
Then $(\Omega, \mathcal{F}, P, \{\thtwot\}_{t\in \R})$ is a  parametric 
dynamical  system.

 Let $\{\thonet\}_{t \in \R}$ be the group acting on $\R$  given  by 
 \eqref{shift1}. 
 We next   define  a continuous  cocycle  
   for  equation \eqref{41}
  over $ (\R, \{\thonet\}_{t \in \R} )$    and  $(\Omega, \mathcal{F},
P, \{\thtwot \}_{t\in \R})$.  This can be done by
first transferring  the stochastic    equation  into a  corresponding non-autonomous deterministic one.  
Given $\omega  \in \Omega$,   denote by
\be\label{zomega}
z ( \omega)=   -\lambda \int^0_{-\infty} e^{\lambda \tau}    \omega  (\tau) d \tau.
\ee
Then it is easy   to check  that   the random variable
$z$  given by
\eqref{zomega}  is a stationary  solution of the    one-dimensional Ornstein-Uhlenbeck
equation:
$$
dz + \lambda z dt = d w.
$$
In other words,   we have
\be
\label{z1}
dz(\thtwot \omega )  + \lambda z(\thtwot \omega)  dt = d w.
\ee
It is known  that  there exists a $\thtwot$-invariant set $\tilde{\Omega}\subseteq \Omega$
of full $P$ measure  such that
  $z(\thtwot\omega)$  is
 continuous in $t$ for every $\omega \in \tilde{\Omega}$,
and
the random variable $|z(\omega)|$ is tempered
(see, e.g., \cite{arn1, cra1, cra2}).
Hereafter,   we will not distinguish 
$\tilde{\Omega}$  and $\Omega$, and   write 
$\tilde{\Omega}$ as 
  $\Omega$.

Formally, if $u$  solves  equation \eqref{41}, 
then    the variable  
  $v(t) = u(t) - hz(\thtwot \omega)$   should satisfy 
\be
\label{v1}
{\frac {\partial v}{\partial t}} + \lambda v -\Delta v =
 f(x,  v + h z(\thtwot \omega) ) + g(x,t)   +  z (\thtwot  \omega )\Delta h,
\ee 
for $t > \tau$ with $\tau \in \R$ and $x \in \R^n$.
Since   \eqref{v1}  is a deterministic  equation, 
following the arguments of \cite{tem1},  one 
can show    that  under    \eqref{f1}-\eqref{f4},
for each  $\omega \in \Omega$,   $\tau \in \R$    and $v_\tau \in \ltwo$, 
equation \eqref{v1} 
has a unique  solution $v(\cdot,\tau,  \omega, v_\tau)
  \in C([\tau, \infty), \ltwo) \bigcap L^2((\tau, \tau+T); \hone)$
   with $v(\tau, \tau,  \omega, v_\tau) = v_\tau$ for every $T>0$. 
   Furthermore,   for each  $t \ge \tau$,
   $v (t, \tau, \omega, v_\tau)$ is  ($\calf, \calb (\ltwo))$-measurable in $\omega$
   $\in \Omega$ and  continuous  in  $v_\tau $
 with respect to the norm of  $\ltwo$ . 
Let $u(t, \tau,  \omega, u_\tau) = v(t, \tau,  \omega, v_\tau) + hz(\thtwot \omega)$
with $u_\tau = v_\tau + h z(\theta_{2, \tau} \omega)$. Then
we find that
$u$  is continuous
in both $t \ge \tau$   and $u_\tau \in \ltwo$     and is 
  ($\calf, \calb (\ltwo))$-measurable in $\omega$
   $\in \Omega$.   In addition, it 
   follows  from
\eqref{v1}   that $u$   is  a solution of problem \eqref{41}-\eqref{42}.
We now   define a cocycle $\Phi: \R^+ \times \R \times \Omega \times \ltwo$
$\to \ltwo$ for    the stochastic  problem
\eqref{41}-\eqref{42}.
Given $t \in \R^+$,  $\tau \in \R$, $\omega \in \Omega$  and $u_\tau \in \ltwo$,
set
 \be \label{rdphi}
 \Phi (t, \tau,  \omega, u_\tau) =  u (t+\tau,  \tau, \theta_{2, -\tau} \omega, u_\tau) =
v(t+\tau, \tau,  \theta_{2, -\tau} \omega,  v_\tau)
+ hz(\thtwot \omega ),  
\ee
where $v_\tau = u_\tau - h z( \omega)$. 
By  the properties of $u$, it is easy   to check   that $\Phi$ 
is a continuous   cocycle on $\ltwo$ over $(\R,  \{\thonet\}_{t \in \R})$
and 
$(\Omega, \mathcal{F},
P, \{\thtwot \}_{t\in \R})$, where 
$\{\thonet\}_{t \in \R}$ and
$\{\thtwot \}_{t\in \R}$ are given by
\eqref{shift1} and \eqref{rdshift1}, respectively.
In other words,  the  mapping  $\Phi$    given
by \eqref{rdphi}  satisfies   all   conditions  (i)-(iv) in Definition \ref{ds1}.
   In  what   follows,  we establish   uniform estimates
  for the solutions of problem \eqref{41}-\eqref{42} and prove the existence
  of  pullback  attractors  for $\phi$ in $\ltwo$.
  To this end, we first need   to specify   a collection $\cald$ 
  of families of subsets  of $\ltwo$.

 Let  $B $  be a  bounded nonempty  subset of $\ltwo$, and denote   by
   $  \| B\| = \sup\limits_{\varphi \in B}
   \| \varphi\|_{\ltwo }$. 
Suppose 
   $D =\{ D(\tau, \omega): \tau \in \R, \omega \in \Omega \}$   is   a family of
  bounded nonempty   subsets of $\ltwo $  satisfying, 
  for every $\tau \in \R$   and $\omega \in \Omega$, 
 \be
 \label{attdom1}
 \lim_{s \to  - \infty} e^{   \lambda  s} \| D( \tau + s, \theta_{2, s} \omega ) \|^2 =0,
 \ee 
 where $\lambda$ is  as in \eqref{41}.
Denote   by $\cald_\lambda$     the  collection of all families of
bounded nonempty  subsets of $\ltwo$
which  fulfill condition \eqref{attdom1}, i.e.,
 \be
 \label{Dlambda}
\cald_\lambda = \{ 
   D =\{ D(\tau, \omega): \tau \in \R, \omega \in \Omega \}: \ 
 D  \ \mbox{satisfies} \  \eqref{attdom1} \} .
\ee
It is evident that
$\cald_\lambda$ is  neighborhood closed.
 The following condition   will be needed  when deriving 
   uniform estimates of solutions:
 \be
 \label{gcond1}
 \int_{-\infty}^\tau e^{\lambda  s } \| g(\cdot, s)\|^2_\ltwo d s
<  \infty, \quad \forall \ \tau \in \R,
 \ee
 which implies   that
 \be
 \label{gcond2}
\lim_{k \to \infty}  \int_{-\infty}^\tau  \int_{|x| \ge k}  e^{\lambda s}   |g(x,  s) |^2 dx d s =0,
\quad \forall \  \tau  \in \R.
 \ee

 Notice that condition  \eqref{gcond1} 
does  not require that    $g$ be bounded in $L^2(\R^n)$
when $s \to \pm \infty$.  Particularly,
this  condition   has  no any restriction
 on  $g(\cdot, s)$ when $s \to +\infty$.

\subsection{Uniform Estimates of Solutions}

      In this section, we
 derive uniform estimates of  solutions  of  problem \eqref{41}-\eqref{42} 
 which are needed for     proving  the existence of
$\cald_\lambda$-pullback  absorbing sets
 and the   $\cald_\lambda$-pullback  asymptotic compactness of  the cocycle  $\Phi$.
 Especially,  we will show that  the tails of solutions
  are uniformly small   for large space  and time  variables.
  The estimates of solutions  in $\ltwo$    are provided below.

\begin{lem}
\label{lemrde1}
 Suppose  \eqref{f1}-\eqref{f4}  and \eqref{gcond1} hold.
Then for every $\tau \in \R$, $\omega \in \Omega$   and $D=\{D(\tau, \omega)
: \tau \in \R,  \omega \in \Omega\}  \in \cald_\lambda$,
 there exists  $T=T(\tau, \omega,  D)>0$ such that for all $t \ge T$, the solution
 $v$ of equation \eqref{v1}  with $\omega$ replaced by
 $\theta_{2, -\tau} \omega$  satisfies
$$
\| v(\tau, \tau -t,  \theta_{2, -\tau} \omega, v_{\tau -t}  ) \|^2 
 \le M+  M  e^{- \lambda  \tau}
\int_{-\infty}^\tau
e^{\lambda s}  
     \| g(\cdot, s ) \|^2    d s
+M \int_{-\infty}^0 e^{\lambda s} |z(\theta_{2, s} \omega) |^p ds ,
$$
and
 $$
  \int_{\tau -t}^\tau e^{\lambda s}
\left (\| v(s, \tau -t,\theta_{2, -\tau} \omega, v_{\tau -t} ) \|^2_{H^1(\R^n)} 
+
\|   v(s, \tau -t, \theta_{2, -\tau} \omega, v_{\tau -t} ) + h z (\theta_{2, s-\tau} \omega ) \|^p_p \right ) ds
$$
$$
\le Me^{\lambda \tau} +  M 
\int_{-\infty}^\tau
e^{\lambda s}  
     \| g(\cdot, s ) \|^2    d s
+Me^{\lambda \tau}  \int_{-\infty}^0 e^{\lambda s} |z(\theta_{2, s} \omega) |^p ds ,
$$
 where $v_{\tau -t}\in D(\tau -t, \theta_{2, -t} \omega)$ and
  $M$ is a  positive constant depending   on      $\lambda$, but independent of $\tau$, $\omega$   and $D$.
\end{lem}

\begin{proof}  
It follows   from   \eqref{v1} that
\be
\label{p41_1}
{\frac 12} {\frac d{dt}} \|v\|^2 + \lambda \| v\|^2 + \| \nabla v \|^2 =
\int_{\R^n} f(x, v+  h \zto )  \ v  dx  + (g, v) + \zto (  \Delta  h, v).
 \ee
 By   \eqref{f1}  and \eqref{f2},    the first   term on the right-hand side 
 of \eqref{p41_1}   satisfies 
$$
\int_{\R^n}  f(x, v+ h \zto  )  \  v  dx
$$
$$
=\int_{\R^n}  f(x, v+ h \zto )   \  ( v +  h\zto )  ) dx
- \zto \int_{\R^n}  f(x, v+ h \zto ) \  h(x)   dx
$$
$$
\le -\alpha_1 \ii |  v+ h \zto  |^p dx + \ii \psi_1(x) dx
$$
$$
 + \alpha_2  \ii
|  v+ h \zto |^{p-1} \; |h \zto | dx + \ii |\psi_2|\;  |h\zto| dx
$$
\be
\label{p41_2}
\le - {\frac 12} \alpha_1 \|   v+ h \zto  \|^p_p
+ c_1 (1+  |\zto |^p    ) , 
\ee
where we have used  Young's  inequality    to obtain \eqref{p41_2}.
Note that  the    last two  terms on the
right-hand side of \eqref{p41_1} are  bounded by 
\be
\label{p41_3}
\| g \| \| v \| + \|\zto\nabla h \| \| \nabla v \|
\le
{\frac 14} \lambda \| v \|^2 + {\frac 1{\lambda}} \| g \|^2
+{\frac 12} \| \zto \nabla h  \|^2
+{\frac 12} \| \nabla v \|^2.
\ee
By    \eqref{p41_1}-\eqref{p41_3} we find  that
\be
\label{p41_4}
 {\frac d{dt}} \|v\|^2 + {\frac 32} \lambda \| v\|^2 + \| \nabla v \|^2
 + \alpha_1 \| v+ h \zto \|^p_p
 \le {\frac 2\lambda} \| g \|^2
+  c_2 (1+   |\zto  |^p ).
\ee
Multiplying  \eqref{p41_4} by $e^{\lambda t}$ and then
integrating the resulting inequality on $(\tau -t, \tau)$ with $t \ge  0$,
we get  that  for every $\omega \in \Omega$, 
$$
\| v(\tau, \tau -t, \omega, v_{\tau -t} ) \|^2
+ \int_{\tau -t}^\tau e^{\lambda (s-\tau)}
\| \nabla v(s, \tau -t, \omega, v_{\tau -t} ) \|^2 ds
$$
$$
+ {\frac 12} \lambda
\int_{\tau -t}^\tau e^{\lambda (s-\tau)}
\|   v(s, \tau -t, \omega, v_{\tau -t} ) \|^2 ds
+\alpha_1 
 \int_{\tau -t}^\tau e^{\lambda (s-\tau)}
\|   v(s, \tau -t, \omega, v_{\tau -t} ) + h z (\theta_{2, s} \omega ) \|^p_p ds
$$
$$
\le e^{-\lambda t} \| v_{\tau -t} \|^2
+{\frac 2\lambda} e^{-\lambda \tau} \int_{-\infty}^\tau
e^{\lambda s} \| g(\cdot, s) \|^2   ds
+c_3 +c_3  \int_{\tau -t}^\tau e^{\lambda (s-\tau)}|z(\theta_{2,s} \omega)|^p ds.
$$
Replacing $\omega$ by $\theta_{2, -\tau}  \omega$, we find   that
$$
\| v(\tau, \tau -t, \theta_{2, -\tau} \omega, v_{\tau -t} ) \|^2
+ \int_{\tau -t}^\tau e^{\lambda (s-\tau)}
\| \nabla v(s, \tau -t,\theta_{2, -\tau} \omega, v_{\tau -t} ) \|^2 ds
$$
$$
+ {\frac 12}\lambda
\int_{\tau -t}^\tau e^{\lambda (s-\tau)}
\|   v(s, \tau -t,\theta_{2, -\tau} \omega, v_{\tau -t} ) \|^2 
+\alpha_1 
 \int_{\tau -t}^\tau e^{\lambda (s-\tau)}
\|   v(s, \tau -t, \theta_{2, -\tau} \omega, v_{\tau -t} ) + h z (\theta_{2, s-\tau} \omega ) \|^p_p
$$
$$
\le e^{-\lambda t} \| v_{\tau -t} \|^2
+{\frac 2\lambda} e^{-\lambda \tau} \int_{-\infty}^\tau
e^{\lambda s} \| g(\cdot, s) \|^2   ds
+c_3 +c_3  \int_{\tau -t}^\tau e^{\lambda (s-\tau)}|z(\theta_{2,s-\tau} \omega)|^p ds
$$
\be\label{p41_5}
\le e^{-\lambda t} \| v_{\tau -t} \|^2
+{\frac 2\lambda} e^{-\lambda \tau} \int_{-\infty}^\tau
e^{\lambda s} \| g(\cdot, s) \|^2   ds
+c_3 +c_3  \int_{-\infty}^0  e^{\lambda s}|z(\theta_{2,s} \omega)|^p ds.
\ee
Since
$v_{\tau -t} \in D(\tau -t, \theta_{2, -t} \omega)$ we  see that
$$ \limsup_{t \to \infty}
e^{-\lambda t} \| v_{\tau -t} \|^2
\le \limsup_{t \to \infty}  e^{-\lambda t} \|  D(\tau -t, \theta_{2, -t} \omega)   \|^2 =0.
$$
Therefore, there exists $T=T(\tau, \omega, D)>0$   such that  for all
$t \ge T$,  $e^{-\lambda t} \| v_{\tau -t} \|^2 \le 1$, which  along   with
\eqref{p41_5} completes     the proof.\end{proof}

As a  consequence of Lemma \ref{lemrde1}, we have the following
inequality  which is  useful for deriving the uniform estimates of solutions in $\hone$.

\begin{lem}
\label{lemrde2}
 Suppose  \eqref{f1}-\eqref{f4}  and \eqref{gcond1} hold.
Then for every $\tau \in \R$, $\omega \in \Omega$   and $D=\{D(\tau, \omega)
: \tau \in \R,  \omega \in \Omega\}  \in \cald_\lambda$,
 there exists  $T=T(\tau, \omega,  D) \ge 1$ such that for all $t \ge T$, the solution
 $v$ of equation \eqref{v1}  with $\omega$ replaced by
 $\theta_{2, -\tau} \omega$  satisfies
 $$
  \int_{\tau -1}^\tau  
\left (\| \nabla v(s, \tau -t,\theta_{2, -\tau} \omega, v_{\tau -t} ) \|^2 
+
\|   v(s, \tau -t, \theta_{2, -\tau} \omega, v_{\tau -t} ) + h z (\theta_{2, s-\tau} \omega ) \|^p_p \right ) ds
$$
$$
\le M  +  M e^{-\lambda \tau}
\int_{-\infty}^\tau
e^{\lambda s}  
     \| g(\cdot, s ) \|^2    d s
+M   \int_{-\infty}^0 e^{\lambda s} |z(\theta_{2, s} \omega) |^p ds ,
$$
 where $v_{\tau -t}\in D(\tau -t, \theta_{2, -t} \omega)$ and
  $M$ is a  positive constant depending   on      $\lambda$, but independent of $\tau$, $\omega$   and $D$.
\end{lem}

\begin{proof}
By   Lemma \ref{lemrde1}   we see   that  there exists
$T =T(\tau, \omega, D) \ge 1$  such that  for    all $t \ge T$,
 $$
  \int_{\tau -1}^\tau e^{\lambda s}
\left (\| \nabla v(s, \tau -t,\theta_{2, -\tau} \omega, v_{\tau -t} ) \|^2 
+
\|   v(s, \tau -t, \theta_{2, -\tau} \omega, v_{\tau -t} ) + h z (\theta_{2, s-\tau} \omega ) \|^p_p \right ) ds
$$
 $$ \le 
  \int_{\tau -t}^\tau e^{\lambda s}
\left (\| \nabla v(s, \tau -t,\theta_{2, -\tau} \omega, v_{\tau -t} ) \|^2 
+
\|   v(s, \tau -t, \theta_{2, -\tau} \omega, v_{\tau -t} ) + h z (\theta_{2, s-\tau} \omega ) \|^p_p \right ) ds
$$
$$
\le Me^{\lambda \tau} +  M 
\int_{-\infty}^\tau
e^{\lambda s}  
     \| g(\cdot, s ) \|^2    d s
+Me^{\lambda \tau}  \int_{-\infty}^0 e^{\lambda s} |z(\theta_{2, s} \omega) |^p ds  
$$
Note that $e^{\lambda s} \ge e^{\lambda \tau} e^{-\lambda}$ for  all
$s \ge  \tau -1$. Thus Lemma \ref{lemrde2}  follows immediately. \end{proof}

  \begin{lem}
\label{lemrde3}
 Suppose  \eqref{f1}-\eqref{f4}  and \eqref{gcond1} hold.
Then for every $\tau \in \R$, $\omega \in \Omega$   and $D=\{D(\tau, \omega)
: \tau \in \R,  \omega \in \Omega\}  \in \cald_\lambda$,
 there exists  $T=T(\tau, \omega,  D) \ge 1$ such that for all $t \ge T$, the solution
 $v$ of equation \eqref{v1}  with $\omega$ replaced by
 $\theta_{2, -\tau} \omega$  satisfies
 $$
  \| \nabla v(\tau, \tau -t,\theta_{2, -\tau} \omega, v_{\tau -t} ) \|^2 
\le M  +  M e^{-\lambda \tau}
\int_{-\infty}^\tau
e^{\lambda s}  
     \| g(\cdot, s ) \|^2    d s
+M   \int_{-\infty}^0 e^{\lambda s} |z(\theta_{2, s} \omega) |^p ds ,
$$
 where $v_{\tau -t}\in D(\tau -t, \theta_{2, -t} \omega)$ and
  $M$ is a  positive constant depending   on      $\lambda$, but independent of $\tau$, $\omega$   and $D$.
\end{lem}

\begin{proof}
Multiplying    \eqref{v1} by 
    $\Delta v$ and then integrating the equation we get  
\be
\label{p43_1}
{\frac 12} {\frac d{dt}} \| \nabla v \|^2
+ \lambda \| \nabla  v \|^2
+ \| \Delta v \|^2
=-\ii f(x, v + h  \zto ) \Delta v dx
-(g + \zto \Delta h , \Delta v ).
\ee
By \eqref{f2}-\eqref{f4} we have the  following estimates  for  the  nonlinear 
term: 
$$
-\ii f(x, v + h  \zto ) \;  \Delta v dx
= -\ii f(x,  u) \;  \Delta  u dx
 +  \ii f(x,  u) \;  \zto \Delta  h dx
$$
$$
=\ii {\frac {\partial f}{\partial x}} (x, u) \; \nabla u dx
+ \ii {\frac {\partial f}{\partial u}} (x,u) \;  | \nabla  u |^2 dx
+ \ii f(x, u)\;  \zto \Delta h dx
$$
$$
\le \| \psi_3\| \| \nabla u \|
+ \beta \| \nabla u \|^2
+  \alpha_2 \ii |u|^{p-1}  \; |\zto\Delta h  | dx
+
\ii | \psi_2 (x) |  \; |\zto\Delta h  | dx
$$
$$
\le c_1 \left (
 \| \nabla v  + \zto \nabla h \|^2 + \| v+\zto h \|^p_p \right )
 + c_1 \left ( 1+    | \zto   |^p \right ).
$$
\be
\label{p43_2}
\le c_2 \left (
 \| \nabla v  \|^2 + \| v\|^p_p \right )
 + c_2 \left ( 1+    | \zto   |^p \right ).
\ee
Note that  the last term on the right-hand side
of \eqref{p43_1} is bounded by
\be
\label{p43_3}
|(g, \Delta v)| + |  ( \zto\Delta h  , \Delta v )|
\le {\frac 12} \| \Delta v \|^2
+ \| g \|^2 + \| \zto \Delta h \|^2.
\ee
It follows from  \eqref{p43_1}-\eqref{p43_3} that
$$
  {\frac d{dt}} \| \nabla v \|^2
+ 2 \lambda \| \nabla  v \|^2
+ \| \Delta  v \|^2 
 \le c_2 \left (
 \| \nabla v \|^2 + \| v\|^p_p \right )
 + \| g \|^2
 + c_3 \left (1  + 
  |   \zto   |^p \right ),
$$
which implies that
\be
\label{p43_7}
{\frac d{dt}} \| \nabla v \|^2
 \le c_2 \left (
 \| \nabla v \|^2 + \| v\|^p_p \right )
 + \| g \|^2
 + c_3 \left (1  + 
  |   \zto   |^p \right ).
  \ee
 Given $t \ge 0$, $\tau \in \R$,
 $\omega \in \Omega$ and
  $s \in (\tau-1, \tau)$, integrating  \eqref{p43_7}  on 
 $(s, \tau)$     we  get
 $$
 \| \nabla v(\tau, \tau -t,   \omega, v_{\tau -t} ) \|^2
 \le
 \| \nabla v(s, \tau -t,  \omega, v_{\tau -t} ) \|^2 
 + c_3 \int_s^{\tau}  (1 + |z(\theta_{2, \xi} \omega )|^p ) d \xi
 $$
 $$
 +
 c_2 \int_s^{\tau}
 \left (
 \| \nabla v(\xi, \tau -t,  \omega, v_{\tau -t} ) \|^2  
 +
 \|  v(\xi, \tau -t,  \omega, v_{\tau -t} ) \|^p_p 
 \right ) d \xi 
 + \int_s^\tau  \|g(\cdot, \xi )\|^2 d\xi.
 $$
 Integrating the above with respect to $s$ on  $(\tau-1, \tau)$, we  obtain that
 $$
 \| \nabla v(\tau, \tau -t,  \omega, v_{\tau -t} ) \|^2
 \le
 \int_{\tau -1}^\tau \| \nabla v(s, \tau -t,  \omega, v_{\tau -t} ) \|^2  ds
 + c_3 \int_{\tau -1}^{\tau}  (1 + |z(\theta_{2, \xi} \omega )|^p ) d \xi
 $$
 $$
 +
 c_2 \int_{\tau -1}^{\tau}
 \left (
 \| \nabla v(\xi, \tau -t, \omega, v_{\tau -t} ) \|^2  
 +
 \|  v(\xi, \tau -t,   \omega, v_{\tau -t} ) \|^p_p 
 \right ) d \xi 
 + \int_{\tau -1} ^\tau  \|g(\cdot, \xi )\|^2 d\xi.
 $$
 Now  replacing $\omega$ by
 $ \theta_{2, -\tau} \omega$, we get     that 
 $$
 \| \nabla v(\tau, \tau -t,  \theta_{2, -\tau} \omega, v_{\tau -t} ) \|^2
 \le
 \int_{\tau -1}^\tau \| \nabla v(s, \tau -t,  \theta_{2, -\tau} \omega, v_{\tau -t} ) \|^2  ds
 + c_3 \int_{\tau -1}^{\tau}  (1 + |z(\theta_{2, \xi-\tau} \omega )|^p ) d \xi
 $$
 $$
 +
 c_2 \int_{\tau -1}^{\tau}
 \left (
 \| \nabla v(\xi, \tau -t,  \theta_{2, -\tau} \omega, v_{\tau -t} ) \|^2  
 +
 \|  v(\xi, \tau -t,  \theta_{2, -\tau} \omega, v_{\tau -t} ) \|^p_p 
 \right ) d \xi 
 + \int_{\tau -1} ^\tau  \|g(\cdot, \xi )\|^2 d\xi.
 $$
  Let   $T=T( \tau, \omega, D)\ge 1$  be  the positive number  found in Lemma \ref{lemrde2}.
  Then it follows   from the above inequality 
   and  Lemma \ref{lemrde2} 
  that, for all $t \ge T$, 
 $$
 \| \nabla v(\tau, \tau -t,  \theta_{2, -\tau} \omega, v_{\tau -t} ) \|^2
 \le 
  c_4  +  c_4 e^{-\lambda \tau}
\int_{-\infty}^\tau
e^{\lambda s}  
     \| g(\cdot, s ) \|^2    d s
     $$
    \be
    \label{p43_10}
+c_4   \int_{-\infty}^0 e^{\lambda s} |z(\theta_{2, s} \omega) |^p ds  
 + c_4 \int_{\tau -1}^{\tau}   |z(\theta_{2, s-\tau} \omega )|^p d s
 + \int_{\tau -1} ^\tau  \|g(\cdot, s)\|^2 ds.
\ee
Note that   the last  two   terms
on the right-hand  side of \eqref{p43_10}
can be  controlled by the first three terms.
Indeed, we have
\be\label{p43_11}
  \int_{\tau -1}^{\tau}   |z(\theta_{2, s-\tau} \omega )|^p d s
=  \int_{-1}^0 |z(\theta_{2, s} \omega)|^p ds 
\le   e^\lambda
\int_{-\infty}^0 e^{\lambda s} |z(\theta_{2, s} \omega)|^p ds.
\ee
Similarly,  we have
\be\label{p43_12}
\int_{\tau -1} ^\tau  \|g(\cdot, s)\|^2 ds
\le
e^\lambda e^{-\lambda \tau}
\int_{\tau -1} ^\tau  e^{\lambda s}
 \|g(\cdot, s)\|^2 ds
 \le
e^\lambda e^{-\lambda \tau}
\int_{-\infty} ^\tau  e^{\lambda s}
 \|g(\cdot, s)\|^2 ds.
\ee
Thus, Lemma \ref{lemrde3}
   follows   from \eqref{p43_10}-\eqref{p43_12}. 
\end{proof}

We now  derive   uniform estimates on the tails
of solutions    for  large space and time variables.
   Such  estimates are  essential  for proving
the asymptotic compactness of  equations defined on unbounded domains.

  \begin{lem}
\label{lemrde4}
 Suppose  \eqref{f1}-\eqref{f4}  and \eqref{gcond1}
  hold.
Let  $\tau \in \R$, $\omega \in \Omega$   and $D=\{D(\tau, \omega)
: \tau \in \R,  \omega \in \Omega\}  \in \cald_\lambda$.
Then  for every $\eta>0$,   there exist
    $T=T(\tau, \omega,  D, \eta) \ge 1$
 and $K=K(\tau, \omega, \eta) \ge 1$ such that for all $t \ge T$, the solution
 $v$ of equation \eqref{v1}  with $\omega$ replaced by
 $\theta_{2, -\tau} \omega$  satisfies
 $$
 \int_{|x| \ge K}
 |v(\tau, \tau -t, \theta_{2, -\tau} \omega, v_{\tau -t} ) (x)|^2 dx \le \eta,
 $$
 where $v_{\tau -t}\in D(\tau -t, \theta_{2, -t} \omega)$.
\end{lem}

\begin{proof}
Let $\rho$ be a smooth function defined on $   \R^+$ such that
$0\le \rho(s) \le 1$ for all $s \in \R^+$, and
$$
\rho (s) = \left \{
\begin{array}{ll}
  0 & \quad \mbox{for} \ 0\le s \le 1; \\
 1 & \quad \mbox{for}  \  s \ge 2.
\end{array}
\right.
$$
Note that 
$   \rho^\prime    $  is bounded on $  \R^+$.
Multiplying  \eqref{v1} 
by   $\rho({\frac {|x|^2}{k^2}})v$
and then integrating the equation, we get
 $$
 {\frac 12} {\frac d{dt}} \ii \rh |v|^2 dx
 +\lambda \ii \rh |v|^2 dx
 - \ii  \rh v    \Delta v dx
 $$
 \be
 \label{p44_1}
=  \ii f(x, u) \rh v dx
 + \ii \left ( g + \zto\Delta  h  \right ) \rh v dx.
 \ee
 We first estimate  the term  involving $\Delta v$, for which we have
   $$
     \ii  \rh v  \Delta v  dx
    =  - \ii | \nabla v |^2 \rh dx
    - \ii v \rhp {\frac {2x}{k^2}} \cdot \nabla v dx
    $$
    \be
    \label{p44_3}
    \le  - \int_{k\le |x| \le \sqrt{2} k} v \rhp {\frac {2x}{k^2}} \cdot \nabla v dx
    \le {\frac {c_1}k} 
    \int_{k\le |x| \le \sqrt{2} k}
      |v| \; |\nabla v | dx
    \le {\frac {c_1}k} (\| v \|^2 + \| \nabla v \|^2 ).
    \ee
    Dealing with 
          the nonlinear term as in \eqref{p41_2}, 
          by \eqref{f1} and \eqref{f2}  we can verify that 
$$
  \ii f(x, u) \rh v dx
  \le
 -  {\frac 12} \alpha_1
\ii |u|^p \rh dx + \ii \psi_1 \rh dx
$$
\be
\label{p44_8}
+  \ii \psi_2^2 \rh dx
 + c_2 \ii  \left (  |\zto h|^p +  | \zto h|^2 \right ) \rh dx.
\ee
In addition, 
  the last term on the right-hand side of \eqref{p44_1} satisfies
\be
\label{p44_9}
| \ii (g + \zto \Delta h ) \rh v dx|
\le
{\frac 12} \lambda \ii \rh |v|^2 dx
+ {\frac 1\lambda} \ii (g^2 + |\zto\Delta h  |^2 ) \rh dx.
\ee
Thus, it follows   from 
  \eqref{p44_1}-\eqref{p44_9}  that 
$$
  {\frac d{dt}} \ii \rh |v|^2 dx
+   \lambda \ii \rh |v|^2 dx
 $$
$$
\le {\frac {c_3} k} (\| \nabla v \|^2 + \| v \|^2)
+
c_3 \ii
\left (   |\psi_1| +  |\psi_2|^2
+   g^2 \right )\rh
dx
$$
\be
\label{p44_11}
+c_3 \ii \left (
|\zto\Delta h  |^2  + |\zto h |^2 + | \zto h |^p \right )
\rh dx.
\ee
Since $\psi_1 \in L^1(\R^n)$, given $\eta_0>0$, there exists
$K_1 =K_1(\eta_0) \ge 1$ such that   for all  $k\ge K_1$,
\be\label{p44_12}
c_3 \ii|\psi_1| \rh dx
=c_3 \int_{|x| \ge k} |\psi_1| \rh dx
\le
c_3 \int_{|x| \ge k}   |\psi_1(x)| dx
\le \eta_0.
\ee
Similarly,  since $\psi_2  \in \ltwo$  and
$h\in H^2(\R^n) \bigcap W^{2,p}(\R^n)$, 
there exists $K_2 =K_2(\eta_0) \ge K_1$ such that   for all
$k \ge K_2$, 
$$
c_3 \ii \left ( |\psi_2|^2 +
|\zto\Delta h  |^2  + |\zto h |^2 + | \zto h |^p \right )
\rh dx
$$
\be\label{p44_13}
\le \eta_0 (1 + |\zto|^2 + |\zto|^p)
\le c_4 \eta_0 (1 +   |\zto|^p).
\ee
It  follows    from \eqref{p44_11}-\eqref{p44_13} that
there exists $K_3=K_3(\eta_0) \ge K_2$   such that   for all
$k \ge K_3$
$$
  {\frac d{dt}} \ii \rh |v|^2 dx
+   \lambda \ii \rh |v|^2 dx
 $$
$$
\le  \eta_0   (\| \nabla v \|^2 + \| v \|^2)
+ c_5  \eta_0 (1 +   |\zto|^p)
+ c_3 \int_{|x| \ge k} 
     g^2(x, t)  dx.
$$
Multiplying the above by $e^{\lambda t}$ and then integrating
the inequality on $(\tau-t, \tau)$   with $t \ge 0$, we get
that   for each $\omega \in \Omega$, 
$$
\ii \rh |v(\tau, \tau -t, \omega, v_{\tau -t})|^2 dx
- e^{-\lambda t} \ii \rh |v_{\tau -t} (x)|^2 dx
$$
$$
\le  \eta_0 e^{ -\lambda \tau}
\int_{\tau -t}^\tau e^{\lambda s}
\|v(s, \tau -t, \omega, v_{\tau -t})\|^2_{H^1(\R^n)} ds
+{\frac {c_5 \eta_0}\lambda}
$$
\be\label{p44_20}
+c_5 \eta_0  
  \int_{\tau -t}^\tau e^{\lambda (s-\tau)}
|z(\theta_{2,s} \omega )|^p ds  
+ c_3  \int_{\tau -t}^\tau
\int_{|x| \ge k}  e^{\lambda (s-\tau)} g^2(x,s) dx ds.
\ee
Replacing $\omega$ by $\theta_{2, -\tau} \omega$ in
\eqref{p44_20}   we find   that,   for $\tau \in \R$,
$t \ge  0$,  $\omega \in \Omega$
and $k \ge K_3$, 
$$
\ii \rh |v(\tau, \tau -t, \theta_{2, -\tau}\omega, v_{\tau -t})|^2 dx
- e^{-\lambda t} \ii \rh |v_{\tau -t} (x)|^2 dx
$$
$$
\le  \eta_0 e^{ -\lambda \tau}
\int_{\tau -t}^\tau e^{\lambda s}
\|v(s, \tau -t, \theta_{2, -\tau}\omega, v_{\tau -t})\|^2_{H^1(\R^n)} ds
+{\frac {c_5 \eta_0}\lambda}
$$
$$
+c_5 \eta_0  
  \int_{\tau -t}^\tau e^{\lambda (s-\tau)}
|z(\theta_{2,s-\tau} \omega )|^p ds  
+ c_3  \int_{\tau -t}^\tau
\int_{|x| \ge k}  e^{\lambda (s-\tau)} g^2(x,s) dx ds
$$
$$
\le  \eta_0 e^{ -\lambda \tau}
\int_{\tau -t}^\tau e^{\lambda s}
\|v(s, \tau -t, \theta_{2, -\tau}\omega, v_{\tau -t})\|^2_{H^1(\R^n)} ds
+{\frac {c_5 \eta_0}\lambda}
$$
\be\label{p44_21}
+c_5 \eta_0  
  \int_{-\infty}^0 e^{\lambda s}
|z(\theta_{2,s} \omega )|^p ds  
+ c_3  \int_{-\infty}^\tau
\int_{|x| \ge k}  e^{\lambda (s-\tau)} g^2(x,s) dx ds.
\ee
By \eqref{gcond2} we see    that there exists
$K_4 =K_4(\tau, \eta_0 ) \ge K_3$ such that   for all
$k \ge K_4$,
\be
\label{p44_22}
c_3  \int_{-\infty}^\tau
\int_{|x| \ge k}  e^{\lambda (s-\tau)} g^2(x,s) dx ds
\le \eta_0.
\ee
It follows   from  \eqref{p44_21}-\eqref{p44_22}
and Lemma \ref{lemrde1}  that  there 
exists  $T_1=T_1(\tau, \omega, D) >0$ such that
for   all $t \ge T_1$   and $k \ge K_4$,
$$
\ii \rh |v(\tau, \tau -t, \theta_{2, -\tau}\omega, v_{\tau -t})|^2 dx
- e^{-\lambda t} \ii \rh |v_{\tau -t} (x)|^2 dx
$$
\be\label{p44_40}
\le
c_6 \eta_0
+
c_6 \eta_0  
  \int_{-\infty}^0 e^{\lambda s}
|z(\theta_{2,s} \omega )|^p ds 
+
c_6 \eta_0  \int_{-\infty}^\tau e^{\lambda (s-\tau)}
\|g(\cdot, s)\|^2 ds.
\ee
Since $v_{\tau -t}  \in D(\tau -t, \theta_{2, -t} \omega)$
and $D \in \cald_\lambda$ we find   that
$$
\limsup_{t \to \infty}
e^{-\lambda t} \ii \rh |v_{\tau -t} (x)|^2 dx
\le
\limsup_{t \to \infty}
e^{-\lambda t}   \|D(\tau -t, \theta_{2, -t} \omega) \|^2   =0,
$$
which along   with \eqref{p44_40}
implies    that    there    exists 
$T_2=T_2(\tau, \omega, D, \eta_0) \ge T_1$   such   that
for  all $t \ge T_2$   and $k \ge K_4$, 
$$
\int_{|x| \ge \sqrt{2} k} 
  |v(\tau, \tau -t, \theta_{2, -\tau}\omega, v_{\tau -t})|^2 dx
\le
\ii \rh |v(\tau, \tau -t, \theta_{2, -\tau}\omega, v_{\tau -t})|^2 dx
$$
\be\label{p44_41}
\le
c_7 \eta_0 \left  (
1 +  
  \int_{-\infty}^0 e^{\lambda s}
|z(\theta_{2,s} \omega )|^p    ds
+
   \int_{-\infty}^\tau e^{\lambda (s-\tau)}
\|g(\cdot, s)\|^2  ds \right ).
\ee
Then Lemma \ref{lemrde4}   follows   from
\eqref{p44_41} by choosing 
$\eta_0$  appropriately    for  a given $\eta>0$.
\end{proof}

We are now   ready  to   derive  uniform estimates
on the  solutions  $u$ of the stochastic   equation
\eqref{41} based on those estimates  of 
the solutions  $v$ of equation \eqref{v1}.
By \eqref{rdphi}  we find  that, 
for each $\tau \in \R$, $t \ge 0$  and $\omega \in \Omega$,
\be
\label{uv1}
u(\tau, \tau -t,  \theta_{2, -\tau} \omega,  u_{\tau -t})
= v(\tau, \tau -t,  \theta_{2, -\tau} \omega,  v_{\tau -t})
+ h z(\omega),
\ee
where $v_{\tau -t} = u_{\tau -t}-hz(\theta_{2,  -t} \omega)$.  Let $D=\{D(\tau, \omega): \tau \in \R, \omega \in \Omega\} \in \cald_\lambda$. We then  define another
family  ${\tilde{D}}$   of subsets of $\ltwo$ from $D$.  Given $\tau \in \R$   and $\omega \in \Omega$,  set
\be\label{tilded1}
{\tilde{D}}(\tau, \omega)
 = \{ \varphi \in \ltwo:  \|\varphi \|^2
 \le  2\|D(\tau, \omega)\|^2 + 2 |z( \omega)|^2 \|h \|^2 \}.
 \ee
 Let ${\tilde{D}}$  be the family  consisting  of those sets given by \eqref{tilded1}, i.e.,
 \be\label{tilded2}
 {\tilde{D}} = 
\{ {\tilde{D}}(\tau, \omega) :   {\tilde{D}}(\tau, \omega) \mbox{ is    defined by }  \eqref{tilded1},  \tau \in \R,  \omega \in \Omega \}.
\ee
Since $|z(\omega)|$ is  tempered    and $D \in \cald_\lambda$,
it is   easy   to check    that
${\tilde{D}}$ given  by \eqref{tilded2}  also belongs
to $\cald_\lambda$.
Furthermore,  if $u_{\tau- t}
\in D(\tau-t, \theta_{2, -t} \omega )$, then
$v_{\tau -t} = u_{\tau -t}-hz(\theta_{2,  -t} \omega)$
belongs  to  $  {\tilde{D}} (\tau-t, \theta_{2, -t} \omega )$.  
Based on the above   analysis,   the uniform estimates
in $\hone$  of the solutions
of equation \eqref{41}  follows
immediately from  \eqref{uv1}, Lemma \ref{lemrde1}
and Lemma \ref{lemrde3}.

  \begin{lem}
\label{lemrde5}
 Suppose  \eqref{f1}-\eqref{f4}  and \eqref{gcond1} hold.
Then for every $\tau \in \R$, $\omega \in \Omega$   and $D=\{D(\tau, \omega)
: \tau \in \R,  \omega \in \Omega\}  \in \cald_\lambda$,
 there exists  $T=T(\tau, \omega,  D) \ge 1$ such that for all $t \ge T$, the solution
 $u$ of problem  \eqref{41}-\eqref{42}    satisfies
 \be\label{lemrde5_a1}
  \|   u(\tau, \tau -t,\theta_{2, -\tau} \omega, u_{\tau -t} ) \|^2 
\le \beta (1+ z^2(\omega))  +  \beta e^{-\lambda \tau}
\int_{-\infty}^\tau
e^{\lambda s}  
     \| g(\cdot, s ) \|^2    d s
+\beta   \int_{-\infty}^0 e^{\lambda s} |z(\theta_{2, s} \omega) |^p ds ,
\ee
 where $u_{\tau -t}\in D(\tau -t, \theta_{2, -t} \omega)$ and
  $\beta$ is a  positive constant depending   on      $\lambda$, but independent of $\tau$, $\omega$   and $D$.
\end{lem}

Similarly,  by \eqref{uv1} and Lemma \ref{lemrde4},  we have    the following
uniform estimates on the tails of solutions   of equation \eqref{41}. 
 
  \begin{lem}
\label{lemrde6}
 Suppose  \eqref{f1}-\eqref{f4}  and \eqref{gcond1}
  hold.
Let  $\tau \in \R$, $\omega \in \Omega$   and $D=\{D(\tau, \omega)
: \tau \in \R,  \omega \in \Omega\}  \in \cald_\lambda$.
Then  for every $\eta>0$,   there exist
   $T=T(\tau, \omega,  D, \eta) \ge 1$
 and $K=K(\tau, \omega, \eta) \ge 1$ such that for all $t \ge T$, the solution
 $u$ of  problem  \eqref{41}-\eqref{42}      satisfies
 $$
 \int_{|x| \ge K}
 |u(\tau, \tau -t, \theta_{2, -\tau} \omega, u_{\tau -t} ) (x)|^2 dx \le \eta,
 $$
 where $u_{\tau -t}\in D(\tau -t, \theta_{2, -t} \omega)$.
\end{lem}

\subsection{Existence of   Attractors for  Reaction-Diffusion Equations}
 
 In  this subsection, we establish  the existence of a $\cald_\lambda$-pullback attractor   for the cocycle $\Phi$  associated  with the stochastic problem \eqref{41}-\eqref{42}.  We first notice  that, by Lemma \ref{lemrde5},  $\Phi$  has a closed
 $\cald_\lambda$-pullback  absorbing   set  $K$  in  $\ltwo$.
 More precisely, given $\tau \in \R$    and $\omega \in \Omega$, let
$
 K(\tau, \omega) = \{ u\in \ltwo: \|u \|^2
 \le  L(\tau, \omega) \}$,  where
 $L(\tau, \omega)$ is the   constant 
 given by  the right-hand side 
 of  \eqref{lemrde5_a1}. 
 It is evident   that, for each $\tau \in \R$,
  $L(\tau, \cdot): \Omega \to \R$ is 
  $(\calf, \calb (\R))$-measurable.
  In addition,  by simple calculations,  one can  verify     that
  $$
  \lim_{r \to -\infty}
  e^{\lambda r} \| K(\tau +r, \theta_{2, r} \omega)\|^2
  = \lim_{r \to -\infty}
  e^{\lambda r}L (\tau +r, \theta_{2, r} \omega ) =0.
  $$
  In other  words, $K=\{K(\tau, \omega): \tau \in \R,  \omega \in \Omega \}$
  belongs   to $\cald_\lambda$. 
 For every $\tau \in \R$,
  $\omega \in \Omega$
  and $D \in \cald_\lambda$,  it follows  from Lemma \ref{lemrde5} that
  there  exists $T=T(\tau, \omega, D) \ge 1$  such that   for   all
  $t \ge T$,
  $$
  \Phi (t, \tau-t, \theta_{2, -t} \omega, D(\tau -t, \theta_{2, -t} \omega ))
  \subseteq K(\tau, \omega).
  $$
  Thus we find 
   that $K=\{K(\tau, \omega): \tau \in \R,  \omega \in \Omega \}$
  is a closed measurable $\cald_\lambda$-pullback  absorbing set
  in  $\cald_\lambda$   for $\Phi$. 
  We next  show   that $\Phi$ is $\cald_\lambda$-pullback asymptotically compact in $\ltwo$.

 \begin{lem}
\label{acrde}
 Suppose  \eqref{f1}-\eqref{f4}  and \eqref{gcond1}
  hold.
 Then $\Phi$ is $\cald_\lambda$-pullback
 asymptotically compact  in $\ltwo$,
that is, for  every $\tau \in \R$, $\omega \in \Omega$, 
 $D=\{D(\tau, \omega): \tau \in \R, \omega \in \Omega \}$
 $  \in \cald_\lambda$,
and $t_n \to \infty$,
 $u_{0,n}  \in D(\tau -t_n, \theta_{2, -t_n} \omega )$,  the sequence
 $\Phi(t_n, \tau -t_n,  \theta_{2, -t_n} \omega,   u_{0,n}  ) $   has a
   convergent
subsequence in $\ltwo $.
\end{lem}

\begin{proof}
 We  need to show  that for every
 $\eta>0$,  the sequence
 $\Phi(t_n, \tau -t_n,  \theta_{2, -t_n} \omega,   u_{0,n}  ) $
  has a finite covering
 of balls of radii less than $\eta$.
 Given $K>0$,   denote by
$ {Q}_K = \{ x \in \R^n:  |x|     \le K \}$
and  ${Q}^c_K = \R^n\setminus Q_K$. 
It follows  from Lemma \ref{lemrde6}  that  
there exist $K= K(\tau, \omega, \eta) \ge 1$  and $N_1=N_1(\tau, \omega,  D, \eta) \ge 1$
such that for  all $n \ge N_1$,
\be\label{pacrde_1}
\|   \Phi  (t_n  ,  \tau -t_n  ,  \theta_{2, -t_n} \omega,   u_{0, n}  )  \| _{L^2 ({Q}^c_{K} )
   }
\le \frac{\eta}{2}.
\ee
By Lemma   \ref{lemrde5}     there
exists $N_2=N_2(\tau, \omega, D, \eta) \ge N_1$
such that for all $n \ge N_2 $,
$$
\|   \Phi  (t_n  ,  \tau -t_n  , \theta_{2, -t_n} \omega,  u_{0,n}    )  \| _{H^1( {Q}_{K } )
} \le  L(\tau, \omega),
$$
where $L(\tau, \omega)$ is the constant given by  the right-hand side
of \eqref{lemrde5_a1}.
By the compactness of embedding
$ H^1( {Q}_{K } )  \hookrightarrow  L^2 ( {Q}_{K } )$,
 the sequence $   \Phi  (t_n  ,  \tau -t_n  , \theta_{2, -t_n} \omega,  u_{0,n}    ) $ is precompact in
$L^2 ( {Q}_{K } )$, and  hence  it has 
  a finite covering in
$L^2 ( {Q}_{K } ) $ of
balls of radii less than ${\frac \eta{2}}$.
This  and    \eqref{pacrde_1}
imply  that $     \Phi  (t_n  ,  \tau -t_n  , \theta_{2, -t_n} \omega,  u_{0,n}    )    $ has a finite covering in
$ \ltwo   $ of balls of radii less than $\eta$, as desired.
  \end{proof}

So far,   we have proved   the $\cald_\lambda$-pullback 
asymptotic compactness and the existence of 
  a closed measurable $\cald_\lambda$-pullback  absorbing set
  of $\Phi$. Thus the existence of a
  $\cald_\lambda$-pullback attractor  of $\Phi$  follows  from
  Proposition \ref{app5} immediately.

 \begin{thm}
\label{thmrde1}
 Suppose  \eqref{f1}-\eqref{f4}  and \eqref{gcond1}
  hold.
 Then  the cocycle $\Phi$ associated with problem \eqref{41}-\eqref{42}
 has a unique $\cald_\lambda$-pullback attractor $\cala \in \cald_\lambda$
 in $\ltwo$  whose structure  is characterized  by
 \eqref{app5_a1}-\eqref{app5_a3}.
\end{thm}

We now  consider the case where the non-autonomous deterministic forcing $g: \R \to \ltwo$
is in $ L^2_{loc}(\R, \ltwo)$
and  is   periodic  
with period $T>0$.  In this case,
condition \eqref{gcond1} is fulfilled, and 
for every $t \ge 0$, $\tau \in \R$ and $\omega \in \Omega$,
   we  have from \eqref{rdphi}   that 
$$
\Phi (t, \tau +T, \omega, \cdot)
=v(t+\tau +T, \tau +T, \theta_{2, -\tau-T} \omega,\  \cdot -h z(\omega)) + hz(\thtwot \omega)$$
$$
=v(t+\tau , \tau , \theta_{2, -\tau} \omega,\  \cdot -h z(\omega)) + hz(\thtwot \omega)
= \Phi (t, \tau,  \omega, \cdot). $$
This  shows   that  $\Phi$ is  periodic with period  $T$ by Definition \ref{periodic-ds1}. 
Let $D \in \cald_\lambda$ and $D_T$  be the $T$-translation
of $D$.   Then by \eqref{attdom1} and \eqref{Dlambda} we have,  for every $r \in \R$   and $\omega \in \Omega$,
\be\label{tranrde1}
\lim_{s \to -\infty}
e^{\lambda s} \| D(r+s, \theta_{2, s} \omega )\|^2 =0.
\ee
In particular, for $r = \tau +T$   with $\tau \in \R$,  we get
from \eqref{tranrde1}  that 
\be\label{tranrde2}
\lim_{s \to -\infty}
e^{\lambda s} \| D_T(\tau +s, \theta_{2, s} \omega )\|^2  
=
\lim_{s \to -\infty}
e^{\lambda s} \| D(\tau +T+s, \theta_{2, s} \omega )\|^2 =0.
\ee
From \eqref{tranrde2}  we see   that
$D_T \in \cald_\lambda$,    and hence
$\cald_\lambda$  is $T$-translation closed.
By the same argument, we also see   that
$\cald_\lambda$  is $-T$-translation closed.
Thus,     from   Lemma \ref{Tlation2} we find 
that  $\cald_\lambda$  is $T$-translation invariant.
Applying Proposition \ref{app6} to problem \eqref{41}-\eqref{42},   we get the periodicity of the 
$\cald_\lambda$-pullback attractor of $\Phi$
as stated below.
Note that the periodicity of the  attractor can also be obtained
by Theorem  \ref{periodatt2}.

 \begin{thm}
\label{thmrde2}
Let $g: \R \to \ltwo$  be    periodic  
with period $T>0$ and 
belong   to  $ L^2 ((0,T), \ltwo)$.
 Suppose  \eqref{f1}-\eqref{f4}   hold.
 Then  the cocycle $\Phi$ associated with problem \eqref{41}-\eqref{42}
 has a unique  periodic $\cald_\lambda$-pullback attractor $\cala \in \cald_\lambda$
 in $\ltwo$  whose structure  is characterized  by
 \eqref{app5_a1}-\eqref{app5_a3}.
\end{thm}

As    discussed in Section \ref{sec31},  the pullback  attractors
of problem \eqref{41}-\eqref{42}  can also  be investigated by constructing a cocycle over the set $\Omega_g$ of all translations of  the function $g$.
Given $ t \ge  0$,  $\tau \in \R$, $\omega \in \Omega$ and
$u_0 \in \ltwo$,  let 
$u(t, 0, g^\tau, \omega, u_0)$  be the solution of
equation \eqref{41} 
 with $g$  being   replaced by
$g^\tau$.  Here, $u_0$ is the initial condition of $u$ at
$t =0$.   
Suppose $g: \R \to \ltwo$ is not a periodic  function.
Then define a map 
${\tilde{\Phi}}: \R^+ \times \Omega_g \times \Omega\times \ltwo \to \ltwo$ by
 ${\tilde{\Phi}} (t, g^\tau, \omega, u_0)
= u(t, 0, g^\tau, \omega, u_0)$
for  all  $ t \ge  0$,  $\tau \in \R$, $\omega \in \Omega$ and
$u_0 \in \ltwo$.
One can check  that ${\tilde{\Phi}}$ is a continuous cocycle
in $\ltwo$  
over $(\Omega_g,  \{\thonet\}_{t \in \R})$
and
$(\Omega, \calf, P,  \{\thtwot\}_{t \in \R})$,
where  $ \{\thonet\}_{t \in \R} $ is given by
\eqref{shift2}.
Actually,  ${\tilde{\Phi}} (t, g^\tau, \omega, \cdot)
=
{ {\Phi}} (t, \tau, \omega, \cdot)$  for all
$t \ge 0$, $\tau \in \R$ and $\omega \in \Omega$.
Let  ${\tilde{\cald}}_\lambda$  be    the collection of all
families   ${\tilde{D}} =\{ {\tilde{D}} (g^\tau, \omega):
g^\tau \in \Omega_g, \omega \in \Omega\}$ such that
for  every $  g^\tau \in \Omega_g$
and $\omega \in \Omega$,
$$
\lim_{s \to -\infty}
e^{\lambda s} \| {\tilde{D}} (g^{\tau + s}, \theta_{2,s} \omega )\|^2 =0. 
$$
Then, by Proposition \ref{app7} we can show   that
${\tilde{\Phi}}$ has a unique ${\tilde{\cald}}_\lambda$-pullback attractor ${\tilde{\cala}}$
$=\{ {\tilde{\cala}} (g^\tau, \omega): g^\tau \in \Omega_g, \omega \in \Omega \}  \in {\tilde{\cald}}_\lambda$.
Actually, in this case, we have
${\tilde{\cala}} (g^\tau, \omega) =
{ {\cala}} (\tau, \omega)$  for all
$\tau \in \R$   and $\omega \in \Omega$, where
$\cala$ is the  $\cald_\lambda$-pullback attractor of $\Phi$. 

If $g$  is a periodic function with minimal period $T>0$, then define 
${\tilde{\cald}}_\lambda$  be    the collection of all
families   ${\tilde{D}} =\{ {\tilde{D}} (g^\tau, \omega):
g^\tau \in \Omega_g^T, \omega \in \Omega\}$ such that
for  each $  g^\tau \in \Omega_g^T$
and $\omega \in \Omega$,
$$
\lim_{s \to -\infty}
e^{\lambda s} \| {\tilde{D}} (\theta_{1,s} g^{\tau }, \theta_{2,s} \omega )\|^2 =0,
$$
where $\{\theta_{1,t}\}_{t \in \R}$ is given by
\eqref{shift3}. 
By Proposition \ref{app8}, one can prove that
${\tilde{\Phi}}$ has a unique ${\tilde{\cald}}_\lambda$-pullback attractor ${\tilde{\cala}}$
$=\{ {\tilde{\cala}} (g^\tau, \omega): g^\tau \in \Omega_g^T, \omega \in \Omega \}$. 
In this case, we also have
${\tilde{\cala}} (g^\tau, \omega) =
{ {\cala}} (\tau, \omega)$  for all
$\tau \in [0, T)$   and $\omega \in \Omega$, where
$\cala$ is the  periodic  $\cald_\lambda$-pullback attractor of $\Phi$.


\begin{thebibliography}{99}

\bibitem{arn1}
L. Arnold, {\em Random Dynamical Systems}, Springer-Verlag, 1998.


\bibitem{aubin1}
J.P. Aubin and H. Franskowska,
{\em Set-Valued Analysis}, Birkhauser, Bosten, Basel, Berlin, 1990.
   
  

\bibitem{bab1}
A.V. Babin and M.I. Vishik,
{\em Attractors of Evolution Equations}, North-Holland,
Amsterdam, 1992.

 


\bibitem{bal2}
J.M.   Ball,  Continuity properties and global
attractors of generalized semiflows and the Navier-Stokes
equations, {\em  J. Nonl. Sci.},  {\bf 7} (1997),
475-502.


\bibitem{bat1}
P.W. Bates,  H. Lisei and  K.  Lu,
 Attractors for stochastic lattice
 dynamical systems,
{\em Stoch. Dyn.},   {\bf 6}  (2006),      1-21.




\bibitem{bat2}
P.W. Bates,   K.  Lu   and B. Wang,
 Random attractors for  stochastic reaction-diffusion equations
on unbounded domains,  {\em J. Differential Equations},
 {\bf  246}   (2009),   845-869.

 
 


\bibitem{car1} T. Caraballo, H. Crauel,  J. A. Langa and J. C. Robinson,
 The effect of noise on the Chafee-Infante equation: a nonlinear case study,
 {\em Proceedings of AMS},  {\em 135} (2007),
 373-382.


\bibitem{car2}
T. Caraballo, J. Real, I.D. Chueshov,
 Pullback attractors for stochastic heat
 equations in materials with memory,
  {\em Discrete Continuous Dynamical Systems B},
  {\bf 9 }  (2008),   525-539.


 
\bibitem{car3}
T. Caraballo, G. Lukaszewicz and J. Real,
Pullback attractors for asymptotically compact
non-autonomous dynamical systems,
{\em Nonlinear Anal.},  {\bf 64} (2006), 484-498.


\bibitem{car4}
T. Caraballo, G. Lukaszewicz and J. Real,
Pullback attractors for
non-autonomous 2D-Navier-Stokes equations
in some unbounded domains,
{\em  C. R. Acad. Sci. Paris I},  {\bf 342} (2006), 263-268.



 \bibitem{carva1}
 A. N. Carvalho   and   J. A. Langa
 Non-autonomous perturbation of autonomous semilinear differential equations:
 continuity of  local stable and unstable manifolds,
 {\em J. Differential Equations}, {\bf 233} (2007),   622-653.

 \bibitem{carva2}
 A. N. Carvalho   and   J. A. Langa,
 An extension of the concept of gradient semigroups
which is stable under perturbation,
{\em J. Differential Equations}, {\bf 246} (2009),   2646-2668.



\bibitem{che1}
V.V. Chepyzhov and M.I. Vishik,  Attractors of non-autonomous dynamical
systems and their dimensions,
{\em J. Math. Pures. Appl.}, {\bf 73} (1994), 279-333.



\bibitem{chu2}
I. Chueshov and M. Scheutzow,
On the structure of attractors  and invariant measures for a class of
monotone  random systems,
{\em Dynamical Systems}, {\bf 19} (2004), 127-144.





    \bibitem{cra1}
    H. Crauel,  A. Debussche and
    F. Flandoli, Random attractors,
    {\em J. Dyn. Diff. Eqns.}, {\bf 9} (1997), 307-341.


    \bibitem{cra2}
    H. Crauel  and
    F. Flandoli,  Attractors for random dynamical systems,
    {\em Probab. Th. Re. Fields}, {\bf 100} (1994), 365-393.

\bibitem{dua1}
J. Duan and B. Schmalfuss,
 The 3D quasigeostrophic fluid dynamics
  under random forcing on boundary,
  {\em Comm. Math. Sci.},
  {\bf 1} (2003),  133-151.


     \bibitem{fla1}
    F. Flandoli and B. Schmalfu$\beta$,
Random attractors for
the 3D stochastic Navier-Stokes equation with multiplicative
noise,
    {\em Stoch. Stoch. Rep.}, {\bf 59} (1996),  21-45.


\bibitem{gamma1}
L. Gammaitoni, P. Hanggi, P. Jung and F. Marchesoni,
Stochastic resonance,
{\em Reviews of Modern Physics}, {\bf 70} (1998), 223-287.




 \bibitem{hal1}
  J.K.  Hale,   Asymptotic Behavior of
Dissipative Systems,
American Mathematical Society,
  Providence, RI, 1988.

     

\bibitem{huang1}
J. Huang and W. Shen,
Pullback attractors for nonautonomous and random parabolic equations on non-smooth domains,
{\em Discrete and Continuous Dynamical Systems},
{\bf 24} (2009),   855-882.





\bibitem{kloe1}
P.E. Kloeden and J.A. Langa,
 Flattening, squeezing and the existence of
random attractors,
{\em Proc. Royal Soc. London Serie A.}, {\bf 463}  (2007), 163-181.



\bibitem{moss1}
F. Moss,  D.  Pierson  and D.  O'Gorman,
 Stochastic resonance:  tutorial and update,
  {\em Internat. J. Bifur. Chaos Appl. Sci. Engrg.}
  {\bf  4 } (1994),  1383-139



\bibitem{schm1}
B.  Schmalfu$\beta$, Backward cocycles  and attractors  of stochastic differential equations,
{\em International Seminar on Applied Mathematics-Nonlinear Dynamics: Attractor Approximation and Global Behavior},  1992,  185-192.


   
  \bibitem{sel1}
 R. Sell and  Y. You,
 Dynamics of Evolutionary Equations,
 Springer-Verlag,
New York, 2002.
  

\bibitem{tem1}
 R.  Temam,    Infinite-Dimensional Dynamical
Systems in Mechanics and Physics,
  Springer-Verlag,
  New York, 1997.




\bibitem{tuckwell1}
H.C. Tuckwell and J. Jost,
The effects of various spatial distributions of weak noise
on rhythmic spiking,
{\em J. Comput. Neurosci.}, {\bf 30} (2011), 361-371.



\bibitem{tuckwell2}
H.C. Tuckwell,
{\em Introduction to Theoretical Neurobiology: Volume 2, Nonlinear
and Stochastic Theories}, Cambridge University Press, Cambridge, 1998.



 \bibitem{wan1}
 B. Wang,  Attractors for reaction-diffusion equations in
   unbounded domains,
 {\em Physica D}, {\bf 128}  (1999), 41-52.


 \bibitem{wan2}
 B.  Wang,
Asymptotic behavior of stochastic wave equations with critical
exponents on $\R^3$,  
{\em  Transactions of  American Mathematical Society}, 
{\bf 363} (2011), 3639-3663.
 
 
 \bibitem{wan3}
 B. Wang,
 Random Attractors for the Stochastic Benjamin-Bona-Mahony Equation on Unbounded Domains,
{ \em J. Differential Equations},  {\bf 246} (2009), 2506-2537.

\bibitem{wan31}
B. Wang, 
Random attractors for the stochastic FitzHugh-Nagumo
 system on unbounded domains, 
 {\em Nonlinear Analysis TMA}, 
 {\bf 71 } (2009),  2811-2828.
 
 \bibitem{wan32}
 B.  Wang,
  Periodic random attractors for stochastic 
  Navier-Stokes equations on unbounded 
  domains, 
  {\em Electronic J.  Differential Equations},
  {\bf  2012}  (2012),  No.  59, 1-18. 
  
  \bibitem{wan33}
  B. Wang,
  Existence, stability and bifurcation of random 
  periodic solutions
  of stochastic parabolic equations,
  submitted.

\bibitem{wan4}
      B. Wang,
      Pullback attractors for   non-autonomous
      reaction-diffusion equations on  ${\mathbb R}^n$,
      {\em Frontiers of Mathematics in China},  {\bf 4} (2009),
      563-583.


\bibitem{wiesen1}
K.  Wiesenfeld,  D.  Pierson,
 E. Pantazelou, C. Dames and F.  Moss,
 Stochastic resonance on a circle,
 {\em  Phys. Rev. Lett.} {\bf  72} (1994),  2125-2129.



\bibitem{wiesen2}
K. Wiesenfeld and F. Moss,
Stochastic resonance and the benefits of noise: from ice ages to
crayfish and SQUIDs,
{\em Nature}, {\bf 373} (1995), 33-35.
  

 \end{thebibliography}
\end{document}